\documentclass[11pt]{article}
\usepackage{amsfonts, amsmath, amssymb, latexsym, eucal, amscd,cite} 
\usepackage{pb-diagram} 
\newtheorem{theorem}{Theorem}[section]
\newtheorem{lemma}[theorem]{Lemma}
\newtheorem{proposition}[theorem]{Proposition}
\newtheorem{corollary}[theorem]{Corollary}
\newtheorem{exAux}[theorem]{Example}
\newenvironment{example}{\begin{exAux} \rm}{\end{exAux}}
\newtheorem{Def}[theorem]{Definition}
\newenvironment{definition}{\begin{Def} \rm}{\end{Def}}
\newtheorem{Note}[theorem]{Note}
\newenvironment{note}{\begin{Note} \rm}{\end{Note}}
\newtheorem{Problem}[theorem]{Problem}

\newtheorem{Rem}[theorem]{Remark}

\newtheorem{Not}[theorem]{Notation}

\newtheorem{Conj}[theorem]{Conjecture}

\newtheorem{Ass}[theorem]{Assumption}
\newenvironment{assumption}{\begin{Ass} \rm}{\end{Ass}}

\newenvironment{proof}{\medskip\noindent{\bf Proof.\ }}{\qed\medskip}

\newcommand{\qed}{\hfill\mbox{$\Box$\qquad\qquad}}

\newcommand{\A}{{\sf A}}
\newcommand{\B}{{\sf B}}
\newcommand{\W}{{\sf W}}
\newcommand{\I}{{\sf I}}
\newcommand{\J}{{\sf J}}
\newcommand{\E}{{\sf E}}

\newcommand{\bH}{{\sf H}}
\renewcommand{\k}{{\sf k}}
\newcommand{\m}{{\sf m}}
\renewcommand{\a}{{\sf a}}
\renewcommand{\b}{{\sf b}}
\renewcommand{\c}{{\sf c}}
\newcommand{\bR}{{\sf R}}
\newcommand{\bS}{{\sf S}}
\newcommand{\M}{{\sf M}}
\newcommand{\p}{{\sf p}}
\newcommand{\q}{{\sf q}}
\newcommand{\T}{{\sf T}}
\newcommand{\pU}{{\sf U}}
\newcommand{\V}{{\sf V}}

\newcommand{\F}{\mathbb{F}}
\newcommand{\C}{\mathbb{C}}
\newcommand{\Mat}{\text{\rm Mat}_X(\C)}

\newcommand{\vphi}{\varphi}
\renewcommand{\th}{\theta}
\newcommand{\gen}[1]{\langle #1 \rangle}

\newif\ifDRAFT

\begin{document}

\title{Leonard pairs, spin models, and distance-regular graphs}

\author{Kazumasa Nomura and Paul Terwilliger}

\date{}
\maketitle

\begin{quote}
\begin{center}
{\bf Abstract}
\end{center}
\small
A Leonard pair is an ordered pair of diagonalizable linear maps on a finite-dimensional
vector space, that each act on an eigenbasis for the other one in an irreducible tridiagonal fashion.
In the present paper we consider a type of Leonard pair,
said to have spin.
The notion of a spin model was introduced by V.F.R. Jones to construct 
link invariants.
A spin model is a symmetric matrix over $\C$
that satisfies two conditions, called the type II and type III conditions.
It is known that a spin model $\W$ is contained in a certain finite-dimensional algebra
$N(\W)$, called the Nomura algebra.
It often happens that a spin model $\W$ satisfies  $\W \in \M \subseteq N(\W)$,
where $\M$ is the Bose-Mesner algebra of a distance-regular graph $\Gamma$;
in this case we say that $\Gamma$ affords $\W$.
If $\Gamma$ affords a spin model,
then each irreducible module for every Terwilliger algebra of $\Gamma$ takes a certain form,
recently described by Caughman, Curtin, Nomura,  and Wolff.
In the present paper we show that the converse is true;
if each irreducible module for every Terwilliger algebra of $\Gamma$ takes this form,
then $\Gamma$ affords a spin model.
We explicitly construct this spin model when $\Gamma$ has
$q$-Racah type.
The proof of our main result relies heavily on the theory of spin Leonard pairs.
\end{quote}

\section{Introduction}
\label{sec:intro}

The notion of a Leonard pair was introduced by the second author in \cite{T:Leonard}.
A Leonard pair is roughly described as follows.  
Let $V$ denote a vector space with finite positive dimension.
A  Leonard pair on $V$ is an ordered pair of diagonalizable linear maps $A: V \to V$ and
$A^* : V \to V$ that
each act on an eigenbasis for the other one in an irreducible tridiagonal fashion.
The Leonard pairs were classified up to isomorphism in \cite{T:Leonard};
the isomorphism classes correspond to a family of orthogonal polynomials consisting of
the $q$-Racah polynomials and their relatives in the terminating branch of the
Askey scheme \cite{T:survey}.
See \cite{T:survey, T:Leonard, T:qRacah, NT:affine, NT:lineartrans} 
for more information about Leonard pairs.

In the present paper we will consider a type of Leonard pair, said to have spin \cite{C:spinLP}.
The spin condition is described as follows.
Let $A, A^*$ denote a Leonard pair on $V$.
A Boltzmann pair  for $A,A^*$ is an ordered pair of invertible linear maps
$W : V \to V$ and $W^* : V \to V$
such that  $W A = A W$ and $W^* A^* = A^* W^*$ and
$W A^* W^{-1} = (W^*)^{-1} A W^*$.
The Leonard pair $A,A^*$ is said to have spin whenever there exists a
Boltzmann pair for $A,A^*$.
In \cite{C:spinLP} the spin Leonard pairs are classified up to isomorphism
and their Boltzmann pairs are described.

Next we summarize how Leonard pairs come up in algebraic graph theory.
Let  $\Gamma$ denote a distance-regular graph \cite[p.~126]{BCN} with vertex set $X$
and adjacency matrix $\A \in \Mat$.
Let $\M$ denote the subalgebra of $\Mat$ generated by $\A$.
By \cite[Theorem 2.6.1]{BCN} the algebra $\M$ has a basis $\{\A_i\}_{i=0}^D$
where $D$ is the diameter of $\Gamma$ \cite[p.~432]{BCN} and $\A_i$ is the $i^\text{th}$ distance-matrix
of $\Gamma$ \cite[p.~127]{BCN} for $0 \leq i \leq D$.
The algebra $\M$ is called the Bose-Mesner algebra of $\Gamma$;
it is a useful tool in the study of $\Gamma$  \cite{BI, BCN}.
To carry out this study in more detail,
it is helpful to bring in another algebra called the subconstituent algebra or Terwilliger algebra.
This algebra is described as follows.
For $x \in X$ and $0 \leq i \leq D$,
the diagonal matrix $\E^*_i = \E^*_i (x) \in \Mat$ is the projection onto the $i^\text{th}$ subconstituent
of $\Gamma$ with respect to $x$ \cite[p.~434]{BCN}.
The matrices $\{\E^*_i\}_{i=0}^D$ form a basis for a commutative subalgebra $\M^* = \M^*(x)$
of $\Mat$, called the dual Bose-Mesner algebra of $\Gamma$ with respect to $x$.
The Terwilliger algebra $\T = \T(x)$ is the subalgebra of $\Mat$ generated by $\M$ and $\M^*$ 
\cite{T:subconst1}.
The algebra $\T$ is semi-simple \cite[Lemma 3.4]{T:subconst1},
so it is natural to investigate the irreducible $\T$-modules.
For an irreducible $\T$-module $U$, there is a condition called thin;
this means that for the action of  $\M$ and $\M^*$ on $U$ the eigenspaces all
have dimension $1$.
Most of the known distance-regular graphs with large diameter satisfy 
an algebraic condition called $Q$-polynomial \cite[Section 4.1.E]{BCN};
if $\Gamma$ is $Q$-polynomial then
$\M^*$ contains a certain matrix $\A^*$ called the dual adjacency matrix of $\Gamma$
with respect to $x$ \cite[p.\ 379]{T:subconst1}.
In this case $\A, \A^*$ acts on each thin irreducible $\T$-module
as a Leonard pair \cite[Lemmas 3.9, 3.12]{T:subconst1}.

Motivated by state models in statistical mechanics,
V.F.R. Jones \cite{Jones} introduced the notion of a spin model
to construct link invariants.
The construction is roughly described as follows. 
A link is a finite collection of mutually disjoint closed curves in
$3$-dimensional Euclidean space. 
A link diagram is a projection of a link onto a plane.
For a link diagram $L$ and a square matrix $\W$ over $\C$ with all entries nonzero,
Jones defined a scalar $Z_\W(L) \in \C$.
The function $Z_\W$ from the set of all link diagrams to $\C$
gives a link invariant, 
provided that $\W$ satisfies two conditions called type II and type III;
these are given in \cite{Jones} and  \eqref{eq:type2}, \eqref{eq:type3} below.
The matrix $\W$ is called a spin model whenever it is  type II and type III.

In \cite{Jones} Jones gave three examples of spin models; 
the Potts model, the square model, and the odd cyclic model.
In \cite{Jaeger1} F. Jaeger observed that each of these examples is
contained in the Bose-Mesner algebra of a distance-regular graph $\Gamma$.
Specifically,
for the Potts model $\Gamma$ is a complete graph,
for the square model $\Gamma$ is the $4$-cycle,
and for the odd cyclic model $\Gamma$ is an odd cycle.
Jaeger then found a new spin model contained in the Bose-Mesner algebra 
of the Higman-Sims graph \cite{Jaeger1}.
Motivated by this result,
other authors found new spin models contained in the Bose-Mesner algebra
of  the even cycle \cite{BB},
the Hadamard graph \cite{N:Hadamard},
and the double cover of the Higman-Sims graph \cite{Mune}.

For the rest of this section let $\Gamma$ denote a distance-regular graph with vertex set $X$
and Bose-Mesner algebra $\M$.
For a spin model $\W$ contained in $\M$,
it often happens that $\W$ and $\M$ are related by the following condition.
For $b,c \in X$ let $\text{\bf u}_{b,c}$ denote the column vector with entries indexed by $X$
and $y$-entry $\W(b,y) \W(c,y)^{-1}$ for $y \in X$.
The condition is that $\text{\bf u}_{b,c}$ is a common eigenvector of $\M$ for all $b,c \in X$.
In this case we say that {\em $\Gamma$ affords $\W$}.
Each of the above mentioned spin models \cite{Jones, Jaeger1, BB, N:Hadamard, Mune}
is afforded by the associated distance-regular graph;
this will be demonstrated in Section \ref{sec:example} below.

Next we describe how spin models are related to spin Leonard pairs.
Assume that $\Gamma$ affords a spin model $\W$.
It was shown in \cite[Lemma 5.1]{CN:hyper} that $\Gamma$ satisfies a condition
called formally self-dual,
which implies that $\Gamma$ is $Q$-polynomial \cite[p.\ 49]{BCN}.
Pick $x \in X$ and write $\A^*=\A^*(x)$, $\M^* = \M^*(x)$,  $\T = \T(x)$.
Define $\W^* \in \M^*$ that has $(y,y)$-entry $|X|^{1/2} \W(x,y)^{-1}$ for $y \in X$.
Let $U$ denote a thin irreducible $\T$-module.
Then $\A, \A^*$ act on $U$ as a spin Leonard pair,
and  $\W,\W^*$ act on $U$ as a Boltzmann pair  for this Leonard pair;
see Proposition \ref{prop:C2} and Lemma \ref{lem:spinLP3} below.
Consequently the intersection numbers of $U$ \cite[Definition 11.1]{T:survey}
are given by the formulas in \cite[Theorem 8.5]{CW}.
Another consequence for $U$ is that its endpoint \cite[Lemma 3.12]{T:subconst1}
is equal to its dual endpoint \cite[Lemma 3.9]{T:subconst1}.

Now we describe our main result.
Assume that our distance-regular graph $\Gamma$ is formally self-dual and has
$q$-Racah type \cite[Definition 5.1]{TZ}.
Assume that for all $x \in X$ each irreducible $\T(x)$-module $U$ is thin 
with matching endpoint and dual-endpoint, and the intersection numbers of $U$
are given by the formulas in \cite[Theorem 8.5]{CW}.
Then there exists a spin model afforded by $\Gamma$;
this spin model is explicitly constructed.
Our main result is Theorem \ref{thm:main3}.
We comment on the nature of this result. 
We have not discovered any new spin model to date. 
What we have shown, is that a new spin model would result from
the discovery of a distance-regular graph with the right sort of
irreducible $\T$-modules.

The proof of our main result relies heavily on the theory of spin Leonard pairs.
In the first half of the paper we develop this theory,
and obtain several new results that may be of independent interest;
for instance Theorem \ref{thm:WtiinvAi} and Lemmas \ref{lem:spinLP0}, \ref{lem:new}.
The paper is organized as follows.

\medskip

\begin{center}  
 CONTENTS
\end{center}

\begin{quote}
1.   Introduction  \dotfill  $\; 1$    \\
2.  Leonard pairs and Leonard systems  \dotfill  $\; 3$   \\
3. The map $\pi$  \dotfill  $\; 8$  \\  
4.   The map $\rho$   \dotfill  $\; 11$  \\   
5.  Boltzmann pairs and spin Leonard pairs  \dotfill $\; 18$  \\
6.  Leonard pairs of $q$-Racah type \dotfill  $\; 24$ \\
7.  Type II matrices and spin models  \dotfill $\; 29$  \\
8.  Hadamard matrices and spin models  \dotfill $\, 30$\\
9.   Distance-regular graphs     \dotfill  $\; 31$  \\
10.  Formally self-dual distance-regular graphs  \dotfill  $\; 39$  \\ 
11.  From a spin model to spin Leonard pairs  \dotfill  $\; 41$  \\
12.   From a spin model to spin Leonard pairs; the $q$-Racah case  \dotfill $\; 44$  \\
13.  From spin Leonard pairs to a spin model     \dotfill  $\; 48$  \\
14.  From spin Leonard pairs to a spin model; the $q$-Racah case  \dotfill  $\; 51$  \\ 
15.   Examples   \dotfill  $\; 54$  \\
References \dotfill  $\; 55$
\end{quote}

\section{Leonard pairs and Leonard systems}
\label{sec:LP}

We now begin our formal argument.
In this section we recall the notion of a Leonard pair.
We use the following terms.
A square matrix is said to be {\em tridiagonal} whenever
each nonzero entry lies on either the diagonal, the subdiagonal,
or the superdiagonal.
A tridiagonal matrix is said to be {\em irreducible} whenever
each entry on the subdiagonal is nonzero and each entry on
the superdiagonal is nonzero.
Let $\F$ denote a field.

\begin{definition}  \cite[Definition 1.1]{T:Leonard}
\label{def:LP}    \samepage
\ifDRAFT {\rm def:LP}. \fi
Let $V$ denote a vector space over $\F$ with finite positive dimension.
By a {\em Leonard pair} on $V$ we mean an ordered pair
of $\F$-linear maps $A: V \to V$ and $A^* : V \to V$ 
that satisfy the following {\rm (i), (ii)}.
\begin{itemize}
\item[\rm (i)]
There exists a basis for $V$ with respect to which the matrix representing $A$
is irreducible tridiagonal and the matrix representing $A^*$ is diagonal.
\item[\rm (ii)]
There exists a basis for $V$ with respect to which the matrix representing $A^*$
is irreducible tridiagonal and the matrix representing $A$ is diagonal.
\end{itemize}
\end{definition}

\begin{note}
According to a common notational convention,
for a matrix $A$ its conjugate-transpose is denoted by $A^*$.
We are not using this convention.
In a Leonard pair $A,A^*$ the linear maps $A,A^*$ are arbitrary subject to (i) and (ii) above.
\end{note}

\begin{note}
Referring to Definition \ref{def:LP},
let $A,A^*$ denote a Leonard pair on $V$.
Then $A^*,A$ is a Leonard pair on $V$.
\end{note}

We refer the reader to \cite{T:survey} for background on Leonard pairs.

From now until the end of Section \ref{sec:qRacah}, fix an integer $d \geq 0$ and
let $V$ denote a vector space over $\F$ with dimension $d+1$.
Let $\text{\rm End}(V)$ denote the $\F$-algebra consisting of the
$\F$-linear maps $V \to V$.
We denote by $I$ the identity element of $\text{\rm End}(V)$.
For $A \in \text{\rm End}(V)$ let $\gen{A}$ denote the subalgebra
of $\text{\rm End}(V)$ generated by $A$.
By \cite[Corollary 2.133]{Rot} the center of $\text{\rm End}(V)$
consists of the elements $\alpha I$ $(\alpha \in \F)$.

\begin{lemma} {\rm \cite[Corollary 5.6]{T:survey}}
\label{lem:generate}    \samepage
\ifDRAFT {\rm lem:generate}. \fi
Let $A,A^*$ denote a Leonard pair on $V$.
Then $A$, $A^*$ together generate the algebra $\text{\rm End}(V)$.
\end{lemma}

For an $\F$-algebra $\cal A$, by an {\em automorphism} of $\cal A$ we mean
an $\F$-algebra isomorphism ${\cal A} \to {\cal A}$.
By an {\em antiautomorphism} of $\cal A$ we mean an $\F$-linear bijection
$\tau : {\cal A} \to {\cal A}$ such that
$(Y Z)^\tau = Z^\tau Y^\tau$ for all $Y, Z \in {\cal A}$.

\begin{lemma}  {\rm\cite[Corollary 7.126]{Rot}  (Skolem-Noether) }
\label{lem:SN}   \samepage
\ifDRAFT {\rm lem:SN}. \fi
A map $\sigma : \text{\rm End}(V) \to \text{\rm End}(V)$ is an automorphism of
$\text{\rm End}(V)$
if and only if there exists an invertible $K \in \text{\rm End}(V)$ such that
$Y^\sigma = K Y K^{-1}$ for all $Y \in \text{\rm End}(V)$.
\end{lemma}

\begin{lemma} {\rm   \cite[Theorem 6.1]{T:survey} }
\label{lem:dagger}    \samepage
\ifDRAFT {\rm lem:dagger}. \fi
Let $A,A^*$ denote a Leonard pair on $V$.
Then there exists a unique antiautomorphism $\dagger$ of $\text{\rm End}(V)$
such that $A^\dagger = A$ and $A^{* \dagger} = A^*$.
Moreover $(Y^\dagger)^\dagger = Y$ for all $Y \in \text{\rm End}(V)$.
The map $\dagger$ fixes every element in $\gen{A}$ and every element in $\gen{A^*}$.
\end{lemma}

We recall the notion of an isomorphism for Leonard pairs.
Let $A,A^*$ denote a Leonard pair on $V$.
Let $V'$ denote a vector space over $\F$ with dimension $d+1$,
and let $A', A^{*\prime}$ denote a Leonard pair on $V'$.
By an {\em isomorphism of Leonard pairs} from $A,A^*$ to $A', A^{*\prime}$
we mean an $\F$-algebra isomorphism
$\text{\rm End}(V) \to \text{\rm End}(V')$ that sends
$A \mapsto A'$ and $A^* \mapsto A^{* \prime}$.
The Leonard pairs $A,A^*$ and $A',A^{*\prime}$ are said to be {\em isomorphic}
whenever there exits an isomorphism of Leonard pairs
from $A,A^*$ to $A', A^{*\prime}$.
In this case the isomorphism is unique by Lemma \ref{lem:generate}.

We recall the self-dual Leonard pairs.
A Leonard pair $A,A^*$ is said to be {\em self-dual} whenever $A,A^*$ is
isomorphic to $A^*,A$.
In this case, the isomorphism of Leonard pairs from $A,A^*$ to $A^*,A$
is called the {\em duality} of $A,A^*$.

\begin{lemma}    \label{lem:s2}    \samepage
\ifDRAFT {\rm lem:s2}. \fi
Assume that $A,A^*$ is self-dual with duality $\sigma$.
Then $\sigma^2 = 1$.
\end{lemma}

\begin{proof}
By construction $\sigma^2$ fixes each of $A,A^*$.
By this and Lemma \ref{lem:generate}, $\sigma^2$ fixes every element
of $\text{\rm End}(V)$.
The result follows.
\end{proof}

When working with a Leonard pair, it is convenient to consider a closely related object
called a Leonard system \cite{T:Leonard}.
Before we define a Leonard system,
we recall a few concepts from linear algebra.

Let $\text{\rm Mat}_{d+1}(\F)$ denote the $\F$-algebra consisting of
the $d+1$ by $d+1$ matrices that have all entries in $\F$.
We index the rows and columns by $0,1,\ldots,d$.
Let $\{v_i\}_{i=0}^d$ denote a basis for $V$.
For $A \in \text{\rm End}(V)$ and $M \in \text{\rm Mat}_{d+1}(\F)$,
we say that {\em $M$ represents $A$ with respect to $\{v_i\}_{i=0}^d$} whenever
$A v_j = \sum_{i=0}^d M_{i,j} v_i$ for $0 \leq j \leq d$.
Let $A \in \text{End}(V)$.
For $\th \in \F$ define 
$V(\th) = \{ v \in V \,|\, A v = \th v\}$.
Observe that $V(\th)$ is a subspace of $V$.
The scalar $\th$ is called an {\em eigenvalue} of $A$ whenever $V(\th) \neq 0$.
In this case, $V(\th)$ is called the {\em eigenspace} of $A$ corresponding to $\th$.
We say that
$A$ is {\em diagonalizable}  whenever $V$ is spanned by the eigenspaces of $A$.
We say that $A$ is {\em multiplicity-free}
whenever $A$ is diagonalizable, and each eigenspace of $A$ has dimension one.
Assume that $A$ is multiplicity-free, and let
$\{V_i\}_{i=0}^d$ denote an ordering of the eigenspaces of $A$.
Then $V = \sum_{i=0}^d V_i$ (direct sum).
For $0 \leq i \leq d$ let $\th_i$ denote the eigenvalue of $A$ corresponding to $V_i$.
For $0 \leq i \leq d$ define $E_i \in \text{End}(V)$ such that
$(E_i - I)V_i = 0$ and $E_i V_j = 0$ if $j \neq i$ $(0   \leq j \leq d)$.
Thus $E_i$ is the projection onto $V_i$.
Observe that 
(i) $V_i = E_i V$ $(0 \leq i \leq d)$;
(ii) $E_i E_j = \delta_{i,j} E_i$ $(0 \leq i,j \leq d)$;
(iii) $I = \sum_{i=0}^d E_i$;
(iv) $A = \sum_{i=0}^d \th_i E_i$.
Also 
\begin{align}
   E_i &= \prod_{\begin{smallmatrix} 0 \leq j \leq d \\ j \neq i \end{smallmatrix}   }
           \frac{A-\th_j I}{\th_i - \th_j}    
   \qquad\qquad  (0 \leq i \leq d).                              \label{eq:Ei0}
\end{align}
We call $E_i$ the {\em primitive idempotent} of $A$ for $\th_i$ $(0 \leq i \leq d)$.
Observe that $\{A^i\}_{i=0}^d$ is a basis for the $\F$-vector space $\gen{A}$,
and $\prod_{i=0}^d (A-\th_i I) =0$.
Also observe that $\{E_i\}_{i=0}^d$ is a basis for the
$\F$-vector space $\gen{A}$.

Let $A,A^*$ denote a Leonard pair on $V$.
By \cite[Lemma 1.3]{T:Leonard} each of $A$, $A^*$ is multiplicity-free.
Let $\{E_i\}_{i=0}^d$ denote an ordering of the primitive idempotents of $A$.
For $0 \leq i \leq d$ pick  $0 \neq v_i \in E_i V$. 
Then $\{v_i\}_{i=0}^d$ is a basis for $V$.
The ordering $\{E_i\}_{i=0}^d$ is said to be {\em standard}
whenever the basis $\{v_i\}_{i=0}^d$ satisfies Definition \ref{def:LP}(ii).
A standard ordering of the primitive idempotents of $A^*$ is similarly defined.

\begin{definition}   {\rm \cite[Definition 1.4]{T:Leonard} }
 \label{def:LS}   \samepage
\ifDRAFT {\rm def:LS}. \fi
By a {\em Leonard system} on $V$  we mean a sequence
\begin{equation}
  \Phi = (A; \{E_i\}_{i=0}^d; A^*; \{E^*_i\}_{i=0}^d)     \label{eq:Phi}
\end{equation}
of elements in $\text{\rm End}(V)$
that satisfy the following (i)--(iii):
\begin{itemize}
\item[\rm (i)]
$A,A^*$ is a Leonard pair on $V$;
\item[\rm (ii)]
$\{E_i\}_{i=0}^d$ is a standard ordering of the primitive idempotents of $A$;
\item[\rm (iii)]
$\{E^*_i\}_{i=0}^d$ is a standard ordering of the primitive idempotents of $A^*$.
\end{itemize}
\end{definition}

For the rest of this section, we fix a Leonard system $\Phi$ as in \eqref{eq:Phi}.

\begin{lemma}       \label{lem:dagger2}   \samepage
\ifDRAFT {\rm lem:dagger2}. \fi
The antiautomorphism $\dagger$ from Lemma \ref{lem:dagger}
fixes each of $E_i$, $E^*_i$ for $0 \leq i \leq d$.
\end{lemma}

\begin{proof}
By Lemma \ref{lem:dagger}.
\end{proof}

We recall the notion of isomorphism for Leonard systems.
Let $V'$ denote a vector space over $\F$ with dimension $d+1$,
and let $\Phi' = (A'; \{E'_i\}_{i=0}^d; A^{*\prime}; \{E^{* \prime}_i\}_{i=0}^d)$
denote a Leonard system on $V'$.
By an {\em isomorphism of Leonard systems} from $\Phi$ to $\Phi'$
we mean an $\F$-algebra isomorphism
$\text{\rm End}(V) \to \text{\rm End}(V')$ that sends
$A \mapsto A'$, $A^* \mapsto A^{* \prime}$  and
$E_i \mapsto E'_i$, $E^*_i \mapsto E^{* \prime}_i$ for $0 \leq i \leq d$.
The Leonard systems $\Phi$ and $\Phi'$ are said to be {\em isomorphic}
whenever there exits an isomorphism of Leonard systems
from $\Phi$ to $\Phi'$.
In this case the isomorphism is unique.

Each of the following is a Leonard system on $V$:
\begin{align*}
  \Phi^* &= (A^*; \{E^*_i\}_{i=0}^d; A; \{E_i\}_{i=0}^d),
\\
  \Phi^{\downarrow} &= (A; \{E_i\}_{i=0}^d; A^*; \{E^*_{d-i}\}_{i=0}^d),
\\
  \Phi^{\Downarrow} &= (A; \{E_{d-i}\}_{i=0}^d; A^*; \{E^*_i\}_{i=0}^d).
\end{align*}
Moreover, for $\alpha, \beta, \alpha^*, \beta^* \in \F$ with $\alpha \alpha^* \neq 0$,
the sequence
\[
   ( \alpha A + \beta I; \{E_i\}_{i=0}^d; \alpha^* A^* + \beta^* I ; \{E^*_i\}_{i=0}^d)
\]
is a Leonard system on $V$. We call this Leonard system {\em an affine transformation of $\Phi$}.

For $0 \leq i \leq d$ let $\th_i$ (resp.\ $\th^*_i$) denote the eigenvalue of $A$ (resp.\ $A^*$)
corresponding to $E_i$ (resp.\ $E^*_i$).
We call $\{\th_i\}_{i=0}^d$ (resp.\ $\{\th^*_i\}_{i=0}^d$) the {\em eigenvalue sequence}
(resp.\ {\em dual eigenvalue sequence}) of $\Phi$.
Throughout this paper we often discuss the Leonard system $\Phi^*$ as well as $\Phi$.
The following notational convention will simplify this discussion.

\begin{definition}    \label{def:star}    \samepage
\ifDRAFT {\rm def:star}. \fi
Referring to our Leonard system $\Phi$,
for any object $\omega$ associated with $\Phi$,
let $\omega^*$ denote the corresponding object for $\Phi^*$.
\end{definition}

We recall the $\Phi$-standard basis.
Pick $0 \neq u \in E_0 V$.
By \cite[Lemma 10.2]{T:survey} the vectors $\{E^*_i u\}_{i=0}^d$
form a basis for $V$, said to be  {\em $\Phi$-standard}.
With respect to this basis the matrix representing $A$ is irreducible
tridiagonal,
and the matrix representing $A^*$ is diagonal with $(i,i)$-entry $\th^*_i$
for $0 \leq i \leq d$.
So with respect to the basis $\{E^*_i u\}_{i=0}^d$ the matrices representing
$A,A^*$ are
\begin{align}
A &:
 \begin{pmatrix}
   a_0 & b_0 & & & & \bf{0} \\
   c_1 & a_1 & b_1 \\
   & c_2 & \cdot & \cdot  \\
   && \cdot & \cdot & \cdot\\
  &&& \cdot & \cdot & b_{d-1} \\
  \bf{0} & & & & c_d & a_d
 \end{pmatrix},
&
A^* &:
 \begin{pmatrix}
   \th^*_0 & & & & & \bf{0}  \\
  & \th^*_1   \\
  && \th^*_2  \\
  & & & \cdot   \\
  & & & & \cdot   \\
 \bf{0} & & & & & \th^*_d
 \end{pmatrix},                                      \label{eq:matrixAAs}
\end{align}
where $\{c_i\}_{i=1}^d$, $\{a_i\}_{i=0}^d$, $\{b_i\}_{i=0}^{d-1}$ are scalars in $\F$
such that $b_{i-1} c_i \neq 0$ for $1 \leq i \leq d$.
The scalars  $\{c_i\}_{i=1}^d$, $\{a_i\}_{i=0}^d$, $\{b_i\}_{i=0}^{d-1}$ are uniquely
determined by $\Phi$, and called the {\em intersection numbers of $\Phi$}.
The sum of the vectors in a $\Phi$-standard basis is contained in $E_0 V$,
and therefore an eigenvector for $A$ with eigenvalue $\th_0$.
Consequently
\begin{align*}
  \th_0 &= c_i + a_i + b_i     &&   (0 \leq i \leq d),   
\end{align*}
where $c_0 = 0$ and $b_d = 0$.

We define some polynomials.
Let $\lambda$ denote an indeterminate,
and let $\F[\lambda]$ denote the $\F$-algebra consisting of the
polynomials in $\lambda$ with all coefficients in $\F$.
For $0 \leq i \leq d$ define the following polynomials in $\F[\lambda]$:
\begin{align*}
\tau_i (\lambda) &= (\lambda-\th_0)(\lambda-\th_1) \cdots (\lambda-\th_{i-1}),
\\
\eta_i (\lambda) &= (\lambda-\th_d)(\lambda-\th_{d-1}) \cdots (\lambda-\th_{d-i+1}).
\end{align*}

For a nonzero $u \in E^*_0 V$,
consider the vectors $\{\tau_i (A) u\}_{i=0}^d$.
By \cite[Section 21]{T:survey} these vectors form a basis for $V$.
This basis is called a {\em $\Phi$-split basis} for $V$.

\begin{lemma}  {\rm \cite[Theorem 3.2]{T:Leonard} }
\label{lem:split}    \samepage
\ifDRAFT {\rm lem:split}. \fi
There exist scalars $\{\vphi_i\}_{i=1}^d$ in $\F$ such that
with respect to a $\Phi$-split basis for $V$
the matrices representing $A$ and $A^*$ are
\begin{align*}
A &:\:
 \begin{pmatrix}
   \th_0 & & & & & \bf{0} \\
   1 & \th_1 \\
   & 1 & \th_2  \\
   && \cdot & \cdot \\
  &&& \cdot & \cdot  \\
  \bf{0} & & & & 1 & \th_d
 \end{pmatrix},
&
A^* &:
 \begin{pmatrix}
   \th^*_0 & \vphi_1 & & & & \bf{0}  \\
  & \th^*_1 & \vphi_2  \\
  && \th^*_2 & \cdot  \\
  & & & \cdot & \cdot  \\
  & & & & \cdot & \vphi_d  \\
 \bf{0} & & & & & \th^*_d
 \end{pmatrix}.
\end{align*}
The sequence $\{\vphi_i\}_{i=1}^d$ is uniquely determined by $\Phi$.
Moreover $\vphi_i \neq 0$ for $1 \leq i \leq d$.
\end{lemma}

Referring to Lemma \ref{lem:split},
the sequence $\{\vphi_i\}_{i=1}^d$ is called the {\em first split sequence of $\Phi$}.
Let $\{\phi_i\}_{i=1}^d$ denote the first split sequence of $\Phi^\Downarrow$.
We call $\{\phi_i\}_{i=1}^d$ the {\em second split sequence of $\Phi$}.
By the {\em parameter array} of $\Phi$
we mean the sequence
\begin{equation}
   (\{\th_i\}_{i=0}^d; \{\th^*_i\}_{i=0}^d; \{\vphi_i\}_{i=1}^d; \{\phi_i\}_{i=1}^d).   \label{eq:parray}
\end{equation}

\begin{lemma}  {\rm \cite[Theorem 1.9]{T:Leonard}}
 \label{lem:isomorphic}   \samepage
\ifDRAFT {\rm lem:isomorphic}. \fi
A Leonard system is uniquely determined up to isomorphism by its parameter array.
\end{lemma}

\begin{lemma} {\rm  \cite[Theorem 23.5]{T:survey} }
\label{lem:intersection}    \samepage
\ifDRAFT {\rm lem:intersection}. \fi
The intersection numbers $\{c_i\}_{i=1}^d$, $\{b_i\}_{i=0}^{d-1}$ of $\Phi$
are determined by the parameter array in the following way:
\begin{align}
  c_i &= \phi_i \frac{\eta^*_{d-i} (\th^*_i) } { \eta^*_{d-i+1} (\th^*_{i-1}) }  
          && (1 \leq i \leq d),                                           \label{eq:ci0}
\\
  b_i &= \vphi_{i+1} \frac{ \tau^*_i (\th^*_i) } { \tau^*_{i+1}(\th^*_{i+1}) }
          && (0 \leq i \leq d-1).                                        \label{eq:bi0}
\end{align}
\end{lemma}

\begin{lemma}   \label{lem:parray}   \samepage
\ifDRAFT {\rm lem:parray}. \fi
The parameter array of $\Phi$ is determined by
the eigenvalue sequence, dual eigenvalue sequence,
and the $\{c_i\}_{i=1}^d$, $\{b_i\}_{i=0}^{d-1}$.
\end{lemma}

\begin{proof}
By Lemma \ref{lem:intersection}.
\end{proof}

We recall the the self-dual Leonard systems.
The Leonard system $\Phi$ is said to be {\em self-dual}
whenever $\Phi$ is isomorphic to $\Phi^*$.
In this case, the isomorphism of Leonard systems from $\Phi$ to $\Phi^*$
is called the {\em duality} of $\Phi$.
We remark that if $\Phi$ is self-dual then so is the Leonard pair $A,A^*$,
and in this case the duality of $\Phi$ is the duality of $A,A^*$.

\begin{lemma}   {\rm \cite[Lemma 8.6, Proposition 8.7]{NT:affine} }
\label{lem:selfdualparam}    \samepage
\ifDRAFT {\rm lem:selfdualparam}. \fi
The Leonard system $\Phi$ is self-dual if and only if
$\th_i = \th^*_i$ for $0 \leq i \leq d$.
In this case, 
$\phi_i = \phi_{d-i+1}$ for $1 \leq i \leq d$.
\end{lemma}

\section{The map $\pi$}
\label{sec:pi}

In this section, we introduce a certain map
$\pi :  \text{\rm End}(V) \to \F$ attached to a Leonard system.
Before defining this map, we briefly recall some basis facts about the trace function 
$\text{\rm tr} \! :  \text{\rm End}(V) \to \F$.
We have $\text{\rm tr} (Y Z) = \text{tr} (Z Y)$ for $Y$, $Z \in \text{\rm End}(V)$.

\begin{lemma}   \label{lem:trYsigma}    \samepage
\ifDRAFT {\rm lem:trYsigma}. \fi
Let $\sigma$ denote an automorphism of $\text{\rm End}(V)$.
Then  $\text{\rm tr}(Y^\sigma) = \text{\rm tr}(Y)$ for all $Y \in \text{\rm End}(V)$.
\end{lemma}

\begin{proof}
By Lemma \ref{lem:SN} there exists  an invertible $K \in \text{\rm End}(V)$ such that
$Y^\sigma = K Y K^{-1}$ for all $Y \in \text{\rm End}(V)$.
We have
$\text{\rm tr}(K Y K^{-1}) = \text{\rm tr}(K^{-1} K Y) = \text{\rm tr}(Y)$.
The result follows.
\end{proof}

\begin{lemma}   \label{lem:trYgamma}    \samepage
\ifDRAFT {\rm lem:trYgamma}. \fi
Let $\tau$ denote an antiautomorphism of $\text{\rm End}(V)$.
Then  $\text{\rm tr}(Y^\tau) = \text{\rm tr}(Y)$ for all $Y \in \text{\rm End}(V)$.
\end{lemma}

\begin{proof}
Without loss we may identify $\text{\rm End}(V)$ with $\text{\rm Mat}_{d+1}(\F)$.
The transpose map is an antiautomorphism of the algebra $\text{\rm Mat}_{d+1}(\F)$.
So the composition of $\tau$ and the transpose map is an automorphism of 
$\text{\rm Mat}_{d+1}(\F)$.
The result follows in view of Lemma \ref{lem:trYsigma} and the fact that a matrix
and its transpose have the same trace.
\end{proof}

We now define the map $\pi$.
Let $\Phi = (A; \{E_i\}_{i=0}^d; A^*; \{E^*_i\}_{i=0}^d)$ denote a Leonard system
on $V$ with parameter array as in \eqref{eq:parray}.

\begin{definition}    \label{def:pipis}    \samepage
\ifDRAFT {\rm def:pipis}. \fi
Define the map
\[
  \pi \; :  \;
  \begin{array}{ccc}
   \text{\rm End}(V)  & \longrightarrow &  \F   \\
       Y  &   \longmapsto &   \text{\rm tr} (Y E_0)
  \end{array}
\]
Observe that $\pi$ is $\F$-linear.
\end{definition}

\begin{lemma}    \label{lem:E0YE0}     \samepage
\ifDRAFT {\rm lem:E0YE0}. \fi
For $Y \in \text{\rm End}(V)$,
\begin{align*}
  E_0 Y E_0 &= \pi(Y) E_0.        
\end{align*}
\end{lemma}

\begin{proof}
Abbreviate ${\cal A} = \text{\rm End}(V)$.
The primitive idempotent $E_0$ is a basis for the  subspace
$E_0 {\cal A} E_0$.
This subspace contains $E_0 Y E_0$,
so there exists $\alpha \in \F$ such that $E_0 Y E_0 = \alpha E_0$.
In this equation, take the trace of each side to get
$\text{\rm tr} (E_0 Y E_0) =\alpha \text{\rm tr}(E_0)$.
We have $\text{\rm tr}(E_0 Y E_0) = \text{\rm tr}(Y E_0) = \pi(Y)$.
Also $\text{\rm tr}(E_0) = 1$ since $E_0$ is a primitive idempotent.
By these comments, $\alpha = \pi (Y)$.
The result follows.
\end{proof}

\begin{lemma}    \label{lem:piY}    \samepage
\ifDRAFT {\rm lem:piY}. \fi
For $0 \leq i \leq d$,
\begin{align}
  \pi (E_i) &= \delta_{i,0}.            \label{eq:piEi}
\end{align}
\end{lemma}

\begin{proof}
We have $\text{\rm tr}(E_0)= 1$.
For $1 \leq i \leq d$ we have $E_i E_0=0$.
By these comments and Definition \ref{def:pipis} we obtain \eqref{eq:piEi}.
\end{proof}

\begin{lemma}    \label{lem:piA}    \samepage
\ifDRAFT {\rm lem:piA}. \fi
We have
\begin{align*}
 \pi (I) &= 1, &
 \pi (A) &= \th_0.
\end{align*}
\end{lemma}

\begin{proof}
We have $\pi(I)=1$ since $\text{\rm tr}(E_0)=1$.
In $A = \sum_{i=0}^d \th_i E_i$, apply $\pi$ to each side and use \eqref{eq:piEi}
to get $\pi(A)=\th_0$.
\end{proof}

\begin{lemma}    \label{lem:piYZ}    \samepage
\ifDRAFT {\rm lem:piYZ}. \fi
For $Y \in \gen{A}$ and $Z \in \text{\rm End}(V)$,
\begin{equation*}
   \pi (Y Z) = \pi (Z Y) = \pi (Y) \pi (Z).        
\end{equation*}
\end{lemma}

\begin{proof}
Write $Y = \sum_{i=0}^d \alpha_i E_i$.
We have
\[
\pi (Y Z) = \text{\rm tr}(Y Z E_0) = \text{\rm tr}(Z E_0 Y)
  = \alpha_0 \text{\rm tr} (Z E_0)  = \alpha_0 \pi (Z).
\]
Similarly $\pi (Z Y) = \alpha_0 \pi (Z)$.
The result follows since $\alpha_0 = \pi (Y)$ by Lemma \ref{lem:piY}.
\end{proof}

\begin{lemma}   \label{lem:pihom}    \samepage
\ifDRAFT {\rm lem:pihom}. \fi
The restriction of $\pi$ to $\gen{A}$ is an $\F$-algebra homomorphism $\gen{A} \to \F$.
\end{lemma}

\begin{proof}
By Lemma \ref{lem:piYZ}.
\end{proof}

Recall the antiautomorphism $\dagger$ from Lemma \ref{lem:dagger}.

\begin{lemma}    \label{lem:dagger3}     \samepage
\ifDRAFT {\rm lem:dagger3}. \fi
The map $\pi$ makes the following diagram commute.
\[
\begin{diagram}
\node{ \text{\rm End}(V) } \arrow{e,t}{\dagger}  \arrow{s,l}{ \pi}
\node{ \text{\rm End}(V) } \arrow{s,r}{ \pi}
\\
\node{\F}  \arrow{e,b}{\text{\rm id}} \node{\F}
\end{diagram}
\]
\end{lemma}

\begin{proof}
Pick any $Y \in \text{\rm End}(V)$.
Using Lemma \ref{lem:trYgamma} we argue
\[
 \pi (Y^\dagger)
 = \text{\rm tr}(Y^\dagger E_0)
 = \text{\rm tr}(Y^\dagger E_0^\dagger)
 = \text{\rm tr}((E_0 Y)^\dagger)
 = \text{\rm tr}(E_0 Y)
 = \text{\rm tr}(Y E_0)
 = \pi (Y).
\]
The result follows.
\end{proof}

\begin{lemma}   \label{lem:piYsigma}    \samepage
\ifDRAFT {\rm lem:piYsigma}. \fi
Assume that $\Phi$ is self-dual with duality $\sigma$.
Then the following diagram commutes.
\[
\begin{diagram}
\node{ \text{\rm End}(V) } \arrow{e,t}{\sigma}  \arrow{s,l}{ \pi^*}
\node{ \text{\rm End}(V) } \arrow{s,r}{ \pi}
\\
\node{\F}  \arrow{e,b}{\text{\rm id}} \node{\F}
\end{diagram}
\]
\end{lemma}

\begin{proof}
For $Z \in \text{\rm End}(V)$ we have $\text{\rm tr}(Z^\sigma) = \text{\rm tr}(Z)$.
For $Y \in \text{\rm End}(V)$
we have $Y^\sigma E_0 = Y^\sigma (E^*_0)^\sigma = (Y E^*_0)^\sigma$.
By these comments $\text{\rm tr}(Y^\sigma E_0) = \text{\rm tr}(Y E^*_0)$.
Thus $\pi (Y^\sigma) = \pi^* (Y)$.
The result follows.
\end{proof}

\begin{lemma}    {\rm \cite[Lemma 9.4]{T:survey} }
\label{lem:nuE0Es0E0}   \samepage
\ifDRAFT {\rm lem:nuE0Es0E0}. \fi
There exists $0 \neq \nu \in \F$ such that
\begin{align}
  \nu E_0 E^*_0 E_0 &=  E_0,  &
  \nu E^*_0 E_0 E^*_0 &= E^*_0.                        \label{eq:nu}
\end{align}
\end{lemma}

Note that $\nu = \nu^*$.
By \eqref{eq:nu} and Lemma \ref{lem:E0YE0},
\begin{equation}
    \nu^{-1} = \pi (E^*_0).          \label{eq:defnu}
\end{equation}

\begin{lemma}    {\rm \cite[Theorem 23.8]{T:survey} }
\label{lem:nu}    \samepage
\ifDRAFT {\rm lem:nu }. \fi
We have
\begin{equation}
  \nu = \frac{\eta_d(\th_0) \eta^*_d (\th^*_0) } {\phi_1 \phi_2 \cdots \phi_d}.   \label{eq:nu2}
\end{equation}
\end{lemma}

\begin{definition}    \label{def:ki}    \samepage
\ifDRAFT {\rm def:ki}. \fi
For $0 \leq i \leq d$ define 
\begin{align}
   k_i &= \nu \pi (E^*_i).        \label{eq:defki}
\end{align}
\end{definition}

\begin{lemma}    \label{lem:ki0}    \samepage
\ifDRAFT {\rm lem:ki0}. \fi
We have
\begin{align}
  k_0 &= 1,     &      \sum_{i=0}^d k_i &= \nu.                     \label{eq:k0}
\end{align}
\end{lemma}

\begin{proof}
Use \eqref{eq:defnu} and \eqref{eq:defki} to get $k_0 = 1$.
In $\sum_{i=0}^d E^*_i = I$, apply $\pi$ to each side,
and simplify the result using \eqref{eq:defki} and Lemma \ref{lem:piA}.
This gives the equation on the right in \eqref{eq:k0}.
\end{proof}

In Lemma \ref{lem:ki2} the scalars $k_i$ are expressed in terms of the parameter array.

\begin{lemma}    \label{lem:E0EsiE0}    \samepage
\ifDRAFT {\rm lem:E0EsiE0}. \fi
For $0 \leq i \leq d$,
\begin{equation*}
  E_0 E^*_i E_0 = \nu^{-1} k_i E_0.               
\end{equation*}
\end{lemma}

\begin{proof}
Use Lemma \ref{lem:E0YE0} (with $Y=E^*_i$) and \eqref{eq:defki}.
\end{proof}

\section{The map $\rho$}
\label{sec:rho}

Let $\Phi = (A; \{E_i\}_{i=0}^d; A^*; \{E^*_i\}_{i=0}^d)$ denote a Leonard system
on $V$.
In this section we introduce a certain  map $\rho : \gen{A} \to \gen{A^*}$.
We abbreviate ${\cal A} = \text{\rm End}(V)$.

\begin{lemma}     \label{lem:AcalAEs0}    \samepage
\ifDRAFT {\rm lem:AcalAEs0}. \fi
The map $\gen{A} \to {\cal A} E^*_0$, $Y \mapsto Y E^*_0$
is an $\F$-linear bijection.
\end{lemma}

\begin{proof}
The map is clearly $\F$-linear.
By \cite[Corollary 5.8]{T:survey} the elements $\{E_i E^*_0\}_{i=0}^d$
form a basis for ${\cal A} E^*_0$. 
The map $Y \mapsto Y E^*_0$ sends the basis $\{E_i\}_{i=0}^d$ for $\gen{A}$
to the basis $\{E_i E^*_0\}_{i=0}^d$ for ${\cal A} E^*_0$, so it is bijective.
\end{proof}

\begin{lemma}    \label{lem:rho}   \samepage
\ifDRAFT {\rm lem:rho}. \fi
There exists a unique $\F$-linear map $\rho : \gen{A} \to \gen{A^*}$
such that for $Y \in \gen{A}$,
\begin{align}
  Y E^*_0 E_0 &= Y^\rho E_0.       \label{eq:defrho}
\end{align}
\end{lemma}

\begin{proof}
Concerning existence,
consider the $\F$-linear map $g : \gen{A} \to {\cal A} E_0$, $Y \mapsto Y E^*_0 E_0$.
By Lemma \ref{lem:AcalAEs0} applied to $\Phi^*$,
the map $\mu : \gen{A^*} \to {\cal A} E_0$,  
$Y \mapsto Y E_0$ is an $\F$-linear bijection.
The composition
\[
  \rho : \gen{A} \xrightarrow{\;\;\;\; g\;\;\;\; } {\cal A} E_0 
   \xrightarrow{\;\;\;\; \mu^{-1}\;\;} \gen{A^*}
\]
satisfies \eqref{eq:defrho}.
We have shown that $\rho$ exists.
The map $\rho$ is unique by Lemma \ref{lem:AcalAEs0}.
\end{proof}

Recall the scalar $\nu$ from \eqref{eq:nu}.

\begin{lemma}    \label{lem:rhorhospre}    \samepage
\ifDRAFT {\rm lem:rhorhospre}. \fi
The maps $\rho$ and $\nu \rho^*$ are inverses.
In particular, the maps $\rho$, $\rho^*$ are bijective.
\end{lemma}
 
\begin{proof}
Pick $Y \in \gen{A}$.
Using in order \eqref{eq:defrho} for $\Phi^*$, \eqref{eq:defrho}, \eqref{eq:nu} we obtain
\[
  (Y^\rho)^{\rho^*} E^*_0
 = Y^\rho E_0 E^*_0
 = Y E^*_0 E_0 E^*_0
 = \nu^{-1} Y E^*_0.
\]
By this and Lemma \ref{lem:AcalAEs0} we get $(Y^\rho)^{\rho^*} = \nu^{-1} Y$.
Similarly, for $Z \in \gen{A^*}$ we get $(Z^{\rho^*})^\rho = \nu^{-1} Z$.
Thus the maps $\rho$ and $\nu \rho^*$ are inverses.
\end{proof}

\begin{definition}    \label{def:Ai}    \samepage
\ifDRAFT {\rm def:Ai}. \fi
For $0 \leq i \leq d$ define 
\begin{equation}
  A_i = \nu (E^*_i)^{\rho^*}.                                 \label{eq:defAiAsi}
\end{equation}
\end{definition}

\begin{lemma}    \label{lem:rhorhos}    \samepage
\ifDRAFT {\rm lem:rhorhos}. \fi
For $0 \leq i \leq d$,
$\rho$ sends $A_i \mapsto  E^*_i$ and $E_i \mapsto \nu^{-1} A^*_i$.
\end{lemma}

\begin{proof}
By Lemma \ref{lem:rhorhospre} and Definition \ref{def:Ai}.
\end{proof}

\begin{lemma}    \label{lem:rhorhos2}    \samepage
\ifDRAFT {\rm lem:rhorhos2}. \fi
For $0 \leq i \leq d$,
\begin{align}
  A_i E^*_0 E_0 &= E^*_i E_0,       &   
  E_i E^*_0 E_0 &= \nu^{-1} A^*_i E_0,      \label{eq:AiEs0E0}
\\
  A^*_i E_0 E^*_0 &= E_i E^*_0,   &
  E^*_i E_0 E^*_0 &= \nu^{-1} A_i E^*_0.    \label{eq:AsiE0Es0}
\end{align}
\end{lemma}

\begin{proof}
To get \eqref{eq:AiEs0E0}, use Lemmas \ref{lem:rho}, \ref{lem:rhorhos}.
Applying \eqref{eq:AiEs0E0} to $\Phi^*$ we obtain \eqref{eq:AsiE0Es0}.
\end{proof}

\begin{lemma}    \label{lem:selfdualAi}    \samepage
\ifDRAFT {\rm lem:selfdualAi}. \fi
Assume that $\Phi$ is self-dual with  duality $\sigma$.
Then $\sigma$ sends  $A_i \leftrightarrow A^*_i$ $(0 \leq i \leq d)$.
\end{lemma}

\begin{proof}
By Lemmas \ref{lem:rho} and \ref{lem:rhorhos2}, 
$A_i$ is the unique element
in $\gen{A}$ such that
\begin{equation}
  A_i E^*_0 E_0 = E^*_i E_0.            \label{eq:rhorhos2aux1}
\end{equation}
Applying this to $\Phi^*$,
we find that for $0 \leq i \leq d$, $A^*_i$ is the unique element in $\gen{A^*}$
such that
\begin{equation}
  A^*_i E_0 E^*_0 =  E_i E^*_0.                   \label{eq:rhorhos2aux2}
\end{equation}
Apply $\sigma$ to \eqref{eq:rhorhos2aux1}, and compare the result with
\eqref{eq:rhorhos2aux2} to find that $\sigma$ sends $A_i \mapsto A^*_i$.
Applying this fact to $\Phi^*$,
we find that $\sigma$ sends $A^*_i \mapsto A_i$.
\end{proof}

\begin{lemma}     \label{lem:A0As0}    \samepage
\ifDRAFT {\rm lem:A0As0}. \fi
We have $A_0 = I$.
\end{lemma}

\begin{proof}
For the equation on the right in  \eqref{eq:AsiE0Es0},
set $i=0$ and apply \eqref{eq:nu} to obtain $E^*_0 = A_0 E^*_0$.
By this and Lemma \ref{lem:AcalAEs0} we get $A_0 = I$.
\end{proof}

\begin{lemma}    \label{lem:rhoI}    \samepage
\ifDRAFT {\rm lem:rhoI}. \fi
The map
$\rho$ sends $I \mapsto E^*_0$ and $E_0 \mapsto \nu^{-1} I$.
\end{lemma}

\begin{proof}
Set $i=0$ in Lemma \ref{lem:rhorhos} and use Lemma \ref{lem:A0As0}.
\end{proof}

\begin{lemma}   \label{lem:sumAi}    \samepage
\ifDRAFT {\rm lem:sumAi}. \fi
We have
\begin{align*}
  \sum_{i=0}^d A_i &= \nu E_0.             
\end{align*}
\end{lemma}

\begin{proof}
Apply $\rho^*$ to each side of the equation $\sum_{i=0}^d E^*_i = I$, 
and evaluate the result using Lemmas \ref{lem:rhorhos}, \ref{lem:rhoI}
applied to $\Phi^*$. 
\end{proof}

\begin{lemma}    \label{lem:AiAsi}    \samepage
\ifDRAFT {\rm lem:AiAsi}. \fi
The elements $\{A_i\}_{i=0}^d$ form a basis for the $\F$-vector space $\gen{A}$.
\end{lemma}

\begin{proof}
By Lemmas \ref{lem:rhorhospre}, \ref{lem:rhorhos},
and since $\{E^*_i\}_{i=0}^d$ form a basis for  the $\F$-vector space $\gen{A^*}$.
\end{proof}

Recall the antiautomorphism $\dagger$ from Lemma \ref{lem:dagger}.

\begin{lemma}       \label{lem:dagger4}   \samepage
\ifDRAFT {\rm lem:dagger4}. \fi
The antiautomorphism $\dagger$ fixes each of $A_i$, $A^*_i$ for $0 \leq i \leq d$.
\end{lemma}

\begin{proof}
By Lemma \ref{lem:dagger} and since $A_i \in \gen{A}$, $A^*_i \in \gen{A^*}$
for $0 \leq i \leq d$.
\end{proof}

\begin{lemma}    \label{lem::rhodagger}    \samepage
\ifDRAFT {\rm lem:rhodagger}. \fi
The map $\rho$ from Lemma \ref{lem:rho}
makes the following diagram commute.
\[
\begin{diagram}
 \node{\gen{A}} \arrow{e,t}{\rho} \arrow{s,l}{\dagger}
 \node{\gen{A^*}} \arrow{s,r}{\dagger}
\\
 \node{\gen{A}} \arrow{e,b}{\rho}
 \node{\gen{A^*}}
\end{diagram}
\]
\end{lemma}

\begin{proof}
By Lemma \ref{lem:dagger2}.
\end{proof}

Recall the $\F$-linear map $\pi$ from Definition \ref{def:pipis}.

\begin{lemma}   \label{lem:rhopi}    \samepage
\ifDRAFT {\rm lem:rhopi}. \fi
For $Y \in \gen{A}$,
\begin{equation}
   \pi (Y^\rho) = \nu^{-1} \pi (Y).          \label{eq:pirho}
\end{equation}
\end{lemma}

\begin{proof}
By Definition \ref{def:pipis},
$\pi (Y^\rho) = \text{tr} (Y^\rho E_0)$.
By Lemma \ref{lem:rho}, 
$Y^\rho E_0 = Y E^*_0 E_0$.
By these comments and Definition \ref{def:pipis}, $\pi(Y^\rho) = \pi(Y E^*_0)$.
By this and Lemma \ref{lem:piYZ}, $\pi(Y^\rho) = \pi(Y) \pi (E^*_0)$.
By \eqref{eq:defnu}, $\pi(E^*_0) = \nu^{-1}$.
Thus \eqref{eq:pirho} holds.
\end{proof}

\begin{lemma}   \label{lem:AiE0}    \samepage
\ifDRAFT {\rm lem:AiE0}. \fi
For $0 \leq i \leq d$,
\begin{align}
  A_i E_0  &= k_i E_0.     \label{eq:AiE0}
\end{align}
\end{lemma}

\begin{proof}
Since $A_i \in \gen{A}$
there exists $\alpha_i \in \F$ such that $A_i E_0 = \alpha_i E_0$.
In the first equation of Lemma \ref{lem:rhorhos2},
take the trace of each side and
use \eqref{eq:defnu}, \eqref{eq:defki} to obtain $\alpha_i = k_i$ after
a routine computation.
The result follows.
\end{proof}

\begin{lemma}    \label{lem:piAi}    \samepage
\ifDRAFT {\rm lem:piAi}. \fi
For $0 \leq i \leq d$,
\begin{align}
  \pi (A_i) &= k_i,  &  \pi(A^*_i) &= \delta_{i,0}.            \label{eq:piAi}
\end{align}
\end{lemma}

\begin{proof}
In \eqref{eq:AiE0},
take the trace of each side and use $\text{tr}(E_0) = 1$ to get
the equation on the left in \eqref{eq:piAi}.
Next we obtain the equation on the right in \eqref{eq:piAi}.
By \eqref{eq:defAiAsi}, $A^*_i = \nu E_i^\rho$.
By Lemma \ref{lem:rhopi}, $\pi(E_i^\rho) = \nu^{-1} \pi (E_i)$.
By these comments,  $\pi (A^*_i) = \pi (E_i)$.
By this and Lemma \ref{lem:piY} we obtain the equation on the right in \eqref{eq:piAi}.
\end{proof}

Recall from \eqref{eq:matrixAAs}
the intersection numbers $\{c_i\}_{i=1}^d$, $\{a_i\}_{i=0}^d$, $\{b_i\}_{i=0}^{d-1}$ of $\Phi$.

\begin{lemma}   \label{lem:AAi}    \samepage
\ifDRAFT {\rm lem:AAi}. \fi
Assume that $d \geq 1$.
Then
\begin{align}
  A A_0 &= c_1 A_1 + a_0 A_0,                            \label{eq:AA0}
\\
  A A_i &= c_{i+1} A_{i+1} + a_i A_i + b_{i-1} A_{i-1}   
               &&  (1 \leq i \leq d-1),                      \label{eq:AAi}
\\
 A A_d &= a_d A_d + b_{d-1} A_{d-1}.                   \label{eq:AAd}
\end{align}
\end{lemma}

\begin{proof}
Pick $0 \neq u \in E_0 V$ and consider the $\Phi$-standard basis $\{E^*_i u\}_{i=0}^d$ for $V$.
By \eqref{eq:matrixAAs},
\begin{align*}
 A E^*_i u &= c_{i+1} E^*_{i+1} u + a_i E^*_i u + b_{i-1} E^*_{i-1} u
                         &&   (0 \leq i \leq d),     
\end{align*}
where $b_{-1}=0$, $c_{d+1}=0$ and $E^*_{-1}=0$, $E^*_{d+1}=0$.
By this and since $u$ is basis for $E_0 V$,
\begin{align*}
 A E^*_i E_0 &= c_{i+1} E^*_{i+1} E_0 + a_i E^*_i E_0 + b_{i-1} E^*_{i-1} E_0
                         &&   (0 \leq i \leq d).     
\end{align*}
By this and the first equation in Lemma \ref{lem:rhorhos2},
\begin{align*}
 (A A_i - c_{i+1} A_{i+1} - a_i A_i - b_{i-1} A_{i-1} ) E^*_0 E_0 &= 0  &&  (0 \leq i \leq d),
\end{align*}
where $A_{-1}=0$, $A_{d+1}=0$.
The map $\gen{A} \to {\cal A} E_0$, $Y \mapsto Y E^*_0 E_0$ is an $\F$-linear bijection,
since this map is the composition of $\rho$ and the bijection 
$\gen{A^*} \to {\cal A} E_0$,
$Y \mapsto Y E_0$ from Lemma \ref{lem:AcalAEs0}.
By these comments we obtain the result.
\end{proof}

\begin{lemma}   \label{lem:AA1}    \samepage
\ifDRAFT {\rm lem:AA1}. \fi
Assume that $d \geq 1$.
Then $A = c_1 A_1 + a_0 I$.
\end{lemma}

\begin{proof}
By Lemma \ref{lem:A0As0} and \eqref{eq:AA0}.
\end{proof}

\begin{lemma}   \label{lem:defp}    \samepage
\ifDRAFT {\rm lem:defp}. \fi
There exist scalars $p^h_{i j}$ $(0 \leq h,i,j \leq d)$ in $\F$ such that
\begin{align}
  A_i A_j &= \sum_{h=0}^d p^h_{i j} A_h   &&   (0 \leq i,j \leq d).       \label{eq:AiAj}
\end{align}
\end{lemma}

\begin{proof}
By Lemma \ref{lem:AiAsi}.
\end{proof}

For notational convenience, define
\begin{align*}
   q^h_{i j} &= (p^h_{i j})^*        &&   (0 \leq h,i,j \leq d).        
\end{align*}
Applying \eqref{eq:AiAj} to $\Phi^*$,
\begin{align}
  A^*_i A^*_j &= \sum_{h=0}^d q^h_{i j} A^*_h   &&   (0 \leq i,j \leq d).        \label{eq:AsiAsj2}
\end{align}

\begin{lemma}    \label{lem:phijphji}    \samepage
\ifDRAFT {\rm lem:phijphji}. \fi
The following hold:
\begin{itemize}
\item[\rm (i)]
$p^h_{i j} = p^h_{j i} \qquad (0 \leq h,i,j \leq d)$;
\item[\rm (ii)]
$\sum_{r=0}^d p^t_{h r} p^r_{i j} = \sum_{s=0}^d p^t_{s j} p^s_{h i}  \qquad (0 \leq h,i,j, t \leq d)$.
\end{itemize}
\end{lemma}

\begin{proof}
(i) 
Since the algebra $\gen{A}$ is commutative.

(ii)
Expand $A_h (A_i A_j) = (A_h A_i) A_j$ in two ways using \eqref{eq:AiAj},
and compare the coefficients.
\end{proof}

\begin{lemma}   \label{lem:kip0ii}    \samepage
\ifDRAFT {\rm lem:kip0ii}. \fi
We have $p^0_{i j} = \delta_{i,j} k_i$  for $0 \leq i,j \leq d$.
\end{lemma}

\begin{proof}
In \eqref{eq:AiAj},  multiply each side on the left by $E_0 E^*_0$ and on the right
by $E^*_0 E_0$. 
Simplify the result using  the equation on the left in \eqref{eq:AiEs0E0} and the equation obtained from it
by applying $\dagger$.
This gives
\[
  E_0 E^*_i E^*_j E_0 = \sum_{h=0}^d p^h_{i j} E_0 E^*_0 E^*_h E_0.
\]
Thus
\[
\delta_{i,j} E_0 E^*_i E_0 = p^0_{i j} E_0 E^*_0 E_0.
\]
In this line, take the trace of each side and use Definition \ref{def:pipis} along with  \eqref{eq:defnu},
\eqref{eq:defki} to get the result.
\end{proof}

\begin{lemma}    \label{lem:khphij}    \samepage
\ifDRAFT {\rm lem:khphij}. \fi
For $0 \leq h,i,j \leq d$,
\begin{equation*}
  k_h p^h_{i j} = k_i p^i_{h j} = k_j p^j_{h i}.         
\end{equation*}
\end{lemma}

\begin{proof}
In view of Lemma \ref{lem:phijphji}(i),
it suffices to show that $k_h p^h_{i j} = k_j p^j_{h i}$.
To obtain this equation,
set $t=0$ in Lemma \ref{lem:phijphji}(ii),
and evaluate the result using Lemma \ref{lem:kip0ii}.
\end{proof}

\begin{lemma}   \label{lem:phij}    \samepage
\ifDRAFT {\rm lem:phij}. \fi
Assume that $d \geq 1$.
Then
\begin{align*}
 c_i &= c_1 p^i_{1,i-1}  &&  (1 \leq i \leq d),
\\
 a_i &= c_1 p^i_{1i} + a_0  && (0 \leq i \leq d),
\\
 b_i &= c_1 p^i_{1,i+1}   &&   (0 \leq i \leq d-1).
\end{align*}
\end{lemma}

\begin{proof}
Compare Lemma \ref{lem:AAi} and \eqref{eq:AiAj} in light of Lemma \ref{lem:AA1}.
\end{proof}

In view of Lemma \ref{lem:AAi} we consider the following polynomials.

\begin{definition}   {\rm \cite[Lemma 13.3]{T:survey} }
\label{def:vi}    \samepage
\ifDRAFT {\rm def:vi}. \fi
Define polynomials $\{v_i\}_{i=0}^d$ in $\F[\lambda]$
such that $v_0 =1$ and for $d \geq 1$,
\begin{align*}
  \lambda v_0 &= c_1 v_1 + a_0,  
\\
  \lambda v_i &= c_{i+1} v_{i+1} + a_i v_i + b_{i-1} v_{i-1}  &&  (1 \leq i \leq d-1).
\end{align*}
\end{definition}

\begin{lemma}    \label{lem:viA}    \samepage
\ifDRAFT {\rm lem:viA}. \fi
Referring to Definition \ref{def:vi},
the following hold for $0 \leq i \leq d$:
\begin{itemize}
\item[\rm (i)]
$v_i$ has degree $i$ and leading coefficient $(c_1 c_2 \cdots c_i)^{-1}$;
\item[\rm (ii)]
$A_i  = v_i (A)$;
\item[\rm (iii)]
$A_i = \sum_{j=0}^d v_i (\th_j) E_j$;
\item[\rm (vi)]
$E_i = \nu^{-1} \sum_{j=0}^d v^*_i (\th^*_j) A_j$.
\end{itemize}
\end{lemma}

\begin{proof}
(i)
By Definition \ref{def:vi}.

(ii)
By Lemma \ref{lem:AAi} and Definition \ref{def:vi}.

(iii)
We have  $A_i E_j = v_i (A) E_j = v_i (\th_j) E_j$  for $0 \leq i,j \leq d$.
We have $A_i = A_i I = \sum_{j=0}^d A_i E_j$.
The result follows.

(iv)
Applying (iii) to $\Phi^*$ we obtain
$A^*_i = \sum_{j=0}^d v^*_i (\th^*_j) E^*_j$.
In this equation, apply $\rho^*$ to each side.
Simplify the result using Lemma \ref{lem:rhorhos} (applied to $\Phi^*$)
to get the result.
\end{proof}

The following lemma is from \cite[Theorem 12.4]{T:survey};
we give a short proof for the sake of completeness.

\begin{lemma}    \label{lem:ki}    \samepage
\ifDRAFT {\rm lem:ki}. \fi
For $0 \leq i \leq d$,
\begin{align}
  k_i &= \frac{b_0 b_1 \cdots b_{i-1} } {c_1 c_2 \ldots c_i}.     \label{eq:kibici}
\end{align}
\end{lemma}

\begin{proof}
We have $k_0=1$ by \eqref{eq:k0}.
By Lemma \ref{lem:khphij} (with $h=1$, $j=i-1$) and Lemma \ref{lem:phij},
\begin{align*}
  c_i k_i  &= b_{i-1} k_{i-1}         &&  (1 \leq i \leq d). 
\end{align*}
The result follows.
\end{proof}

\begin{lemma}    \label{lem:kinonzero}    \samepage
\ifDRAFT {\rm lem:kinonzero}. \fi
We have $k_i \neq 0$ for $0 \leq i \leq d$.
\end{lemma}

\begin{proof}
By Lemma \ref{lem:ki}.
\end{proof}

\begin{lemma}    \label{lem:phij2}   \samepage
\ifDRAFT {\rm lem:phij2}. \fi
For $0 \leq h,i,j \leq d$ the scalar
$p^h_{i j}$ is zero (resp.\ nonzero) if one of $h,i,j$ is greater than (resp.\ equal to) the
sum of the other two.
\end{lemma}

\begin{proof}
By Lemmas \ref{lem:khphij} and \ref{lem:kinonzero}, it suffices to consider the case $h \geq i+j$.
For this case the result follows from \eqref{eq:AiAj} and the observation that for $0 \leq r \leq d$ the elements
$I, A, A^2, \ldots, A^r$ and $A_0, A_1, \ldots, A_r$ form a basis for the same vector space.
\end{proof}

\begin{lemma}    \label{lem:kivith0}    \samepage
\ifDRAFT {\rm lem:kivith0}. \fi
We have  
\begin{align*}
      k_i &= v_i (\th_0)    && (0 \leq i \leq d).
\end{align*}
\end{lemma}

\begin{proof}
For the equation in Lemma \ref{lem:viA}(iii),
multiply each side by $E_0$ to get $A_i E_0 = v_i (\th_0) E_0$.
The result follows from this and Lemma \ref{lem:AiE0}.
\end{proof}

\begin{lemma}   {\rm \cite[Theorem 23.9]{T:survey} }
\label{lem:ki2}    \samepage
\ifDRAFT {\rm lem:ki2}. \fi
For $0 \leq i \leq d$,
\begin{align}
 k_i &= \frac{\vphi_1 \vphi_2 \cdots \vphi_i} {\phi_1 \phi_2 \cdots \phi_i}
        \frac{\eta^*_d(\th^*_0)} {\tau^*_i(\th^*_i) \eta^*_{d-i}(\th^*_i) }
        &&  (0 \leq i \leq d).                                                        \label{eq:kiparam}
\end{align}
\end{lemma}

\begin{proof}
By \eqref{eq:kibici} and Lemma \ref{lem:intersection}.
\end{proof}

The following lemma is from \cite[Theorem 16.2]{T:survey}.
We give a short proof for the sake of completeness.

\begin{lemma} 
\label{lem:AWduality}    \samepage
\ifDRAFT {\rm lem:AWduality}. \fi
Referring to Definition \ref{def:vi},
\begin{align}
    v_i (\th_j) / k_i &= v^*_j (\th^*_i) / k^*_j   &&  (0 \leq i,j \leq d).   \label{eq:AWduality}
\end{align}
\end{lemma}

\begin{proof}
We evaluate $E^*_0 A_i A^*_j E_0$ in two ways.
By the equation on the right in \eqref{eq:AiEs0E0},
\begin{equation}
  A^*_j E_0  = \nu E_j E^*_0 E_0.   \label{eq:AWaux1}
\end{equation}
By Lemma \ref{lem:viA}(iii),
\begin{equation}
  A_i E_j = v_i (\th_j) E_j.         \label{eq:AWaux2}
\end{equation}
By Lemma \ref{lem:E0EsiE0} (applied to $\Phi^*$),
\begin{align}
   E^*_0 E_j E^*_0 = \nu^{-1} k^*_j E^*_0.   \label{eq:AWaux3}
\end{align}
Evaluating  $E^*_0 A_i A^*_j E_0$ using \eqref{eq:AWaux1}--\eqref{eq:AWaux3}
we obtain
\begin{equation}
E^*_0 A_i A^*_j E_0 = v_i (\th_j) k^*_j E^*_0 E_0.            \label{eq:AWaux6}
\end{equation}
Apply this to $\Phi^*$, then interchange $i$, $j$ and apply $\dagger$ to get
\begin{equation}
E^*_0 A_i A^*_j E_0 =  v^*_j (\th^*_i) k_i E^*_0 E_0.       \label{eq:AWaux11}
\end{equation}
We have  $E_0 E^*_0 \neq 0$ by \eqref{eq:nu}.
Now comparing \eqref{eq:AWaux6} and \eqref{eq:AWaux11} 
we obtain \eqref{eq:AWduality}.
\end{proof}

\section{Boltzmann pairs and Spin Leonard pairs}
\label{sec:spinLP}

The notions of a Boltzmann pair and a spin Leonard pair were introduced by Curtin \cite{C:spinLP}.
In this section we first recall these notions,
and then prove a theorem about these topics.

\begin{lemma}    \label{lem:AsWWs}    \samepage
\ifDRAFT {\rm lem:AsWWs}. \fi
Let $A,A^*$ denote a Leonard pair on $V$.
For invertible $W \in \gen{A}$ and invertible $W^* \in \gen{A^*}$
the following are equivalent:
\begin{align}
  W A^* W^{-1} &= (W^*)^{-1} A W^*,      \label{eq:WAsWinv} 
\\
 W^{-1} A^* W &= W^* A (W^*)^{-1},       \label{eq:WinvAsW}
\\
  W^* W A^* &= A W^* W,     \label{eq:AWsW}
\\
 A^* W W^*  &= W W^* A.          \label{eq:WWsA}
\end{align}
\end{lemma}

\begin{proof}
By linear algebra \eqref{eq:AWsW} is equivalent to \eqref{eq:WAsWinv},
and \eqref{eq:WWsA} is equivalent to  \eqref{eq:WinvAsW}.
The equations \eqref{eq:AWsW}, \eqref{eq:WWsA} are equivalent
since each is obtained from the other by applying  $\dagger$.
\end{proof}

\begin{definition}   \cite[Definition1.2]{C:spinLP}
\label{def:spinLP}    \samepage
\ifDRAFT {\rm def:spinLP}. \fi
Let $A,A^*$ denote a Leonard pair on $V$.
By a {\em Boltzmann pair for $A,A^*$} we mean an ordered pair $W, W^*$
such that
\begin{itemize}
\item[\rm (i)]
$W$ is an invertible element of $\gen{A}$;
\item[\rm (ii)]
$W^*$ is an invertible element of $\gen{A^*}$;
\item[\rm (iii)]
$W$, $W^*$ satisfy the four equivalent conditions in Lemma \ref{lem:AsWWs}.
\end{itemize}
The Leonard pair $A,A^*$ is called a {\em spin Leonard pair} whenever
there exists a Boltzmann pair for $A,A^*$.
\end{definition}

For the rest of this section,
let $A,A^*$ denote a spin Leonard pair on $V$ and let $W,W^*$
denote a Boltzmann pair for $A,A^*$.

\begin{lemma}    \label{lem:AW}    \samepage
\ifDRAFT {\rm lem:AW}. \fi
We have
\begin{align*}
  A W &= W A, &
  A^* W^* &= W^* A^*.            
\end{align*}
\end{lemma}

\begin{proof}
Since $W \in \gen{A}$ and $W^* \in \gen{A^*}$.
\end{proof}

\begin{lemma}   \label{lem:affine}    \samepage
\ifDRAFT {\rm lem:affine}. \fi
For $\alpha$, $\beta \in \F$ with $\alpha \neq 0$,
consider the Leonard pair $\alpha A + \beta I$, $\alpha A^* + \beta  I$.
Then $W,W^*$ is a Boltzmann pair for this Leonard pair.
\end{lemma}

\begin{proof}
By Definition \ref{def:spinLP}.
\end{proof}

\begin{lemma}  {\rm  \cite[Lemma 3.3]{C:spinLP} }
\label{lem:WinvWsinv}    \samepage
\ifDRAFT {\rm lem:WinvWsinv}. \fi
The following hold.
\begin{itemize}
\item[\rm (i)]
For nonzero $\alpha, \alpha^* \in \F$ the pair
$\alpha W, \alpha^* W^*$ is a Boltzmann pair for $A,A^*$.
\item[\rm (ii)]
The pair $W^{-1}, (W^*)^{-1}$ is a Boltzmann pair for $A,A^*$.
\end{itemize}
\end{lemma}

\begin{proof}
(i)
By Definition \ref{def:spinLP}.

(ii)
Compare lines  \eqref{eq:WAsWinv} and \eqref{eq:WinvAsW}.
\end{proof} 

\begin{lemma}    \label{lem:AsA}    \samepage
\ifDRAFT {\rm lem:AsA}. \fi
The pair $W^*,W$ is a Boltzmann pair for the Leonard pair $A^*,A$.
\end{lemma}

\begin{proof}
By Lemma \ref{lem:AsWWs}.
\end{proof}

\begin{lemma}    \label{lem:AWWsW}    \samepage
\ifDRAFT {\rm lem:AWWsW}. \fi
We have
\begin{align}
  A W W^* W &= W W^* W A^*,
&
  A^* W W^* W &= W W^* W A,             \label{eq:AWWsW}
\\
  A W^* W W^* &= W^* W W^* A^*,
&
 A^* W^* W W^* &= W^* W W^* A.        \label{eq:AsWswWs}
\end{align}
\end{lemma}

\begin{proof}
By Lemmas \ref{lem:AsWWs}, \ref{lem:AW}.
\end{proof}

\begin{lemma}    {\rm  \cite[Lemma 5.1]{C:spinLP} }
 \label{lem:WWsWprepre}    \samepage
\ifDRAFT {\rm lem:WWsWprepre}. \fi
The following agree up to a nonzero scalar factor in $\F$:
\[
  W W^* W,   \qquad\qquad\qquad W^* W W^*.
\]
\end{lemma}

\begin{proof}
Set $Z =  W W^* W (W^*)^{-1} W^{-1} (W^*)^{-1}$.
By Lemma \ref{lem:AWWsW}, $Z$ commutes with $A$ and $A^*$.
By this and Lemma \ref{lem:generate},
$Z$ commutes with everything in $\text{\rm End}(V)$.
So $Z$ is contained in the center of $\text{\rm End}(V)$.
Thus there exists $\alpha \in \F$ such that $Z = \alpha I$.
The result follows.
\end{proof}

Using $W$ and $W^*$ we obtain an action of the modular group 
$\text{\rm PSL}_2(\mathbb{Z})$ on $\text{\rm End}(V)$ as a group of automorphisms.
Recall from \cite{Alperin} that $\text{PSL}_2(\mathbb{Z})$ has a presentation
by generators $\psi$, $\sigma$ and relations $\psi^3=1$, $\sigma^2=1$.
The next result is a variation on \cite[Lemma 5.2]{C:spinLP}.

\begin{lemma}    \label{lem:PSL}    \samepage
\ifDRAFT {\rm lem:PSL}. \fi
The group $\text{\rm PSL}_2(\mathbb{Z})$ acts on $\text{\rm End}(V)$
such that $\psi$ sends $Y \mapsto (W W^*)^{-1} Y W W^*$
and $\sigma$ sends $Y \mapsto (W W^* W)^{-1} Y W W^* W$
for $Y \in \text{\rm End}(V)$.
\end{lemma}

\begin{proof}
The automorphism $\sigma$ swaps $A$, $A^*$ by \eqref{eq:AWWsW}.
Therefore $A,A^*$ is self-dual with duality $\sigma$.
Thus $\sigma^2 = 1$ by Lemma \ref{lem:s2}.
By Lemma \ref{lem:WWsWprepre},
there exists $\alpha \in \F$ such that
$W^* W W^* = \alpha W W^* W$.
So
\[
  (W W^*)^3 = W W^* W W^* W W^* = \alpha (W W^* W)^2.
\]
Thus $\sigma^2 = \psi^3$.
The result follows.
\end{proof}

For the rest of this section,  we identify $\sigma$ with the automorphism of $\text{\rm End}(V)$
from Lemma \ref{lem:PSL}.
We record a result from the proof of Lemma \ref{lem:PSL}.

\begin{lemma}    \label{lem:PSL2}    \samepage
\ifDRAFT {\rm lem:PSL2}. \fi
The spin Leonard pair $A,A^*$ is self-dual with duality $\sigma$.
\end{lemma}

The following definition is motivated by Lemma \ref{lem:WWsWprepre}.

\begin{definition}    \label{def:balanced}    \samepage
\ifDRAFT {\rm def:balanced}. \fi
The Boltzmann pair $W,W^*$ is said to be {\em balanced} whenever
\begin{equation}
    W W^* W = W^* W W^*.                       \label{eq:WWsWpre}
\end{equation}
\end{definition}

\begin{lemma}    \label{lem:balanced}    \samepage
\ifDRAFT {\rm lem:balanced}. \fi
The following {\rm (i)--(iii)} hold.
\begin{itemize}
\item[\rm (i)]
There exists a nonzero $\alpha \in \F$ such that 
the Boltzmann pair $\alpha W, W^*$ is balanced.
\item[\rm (ii)]
There exists a nonzero $\alpha^* \in \F$ such that 
the Boltzmann pair $W, \alpha^* W^*$ is balanced.
\item[\rm (iii)]
Assume that $W,W^*$ is balanced.
Then for any nonzero $\xi \in \F$,
the Boltzmann pair $\xi W, \xi W^*$ is balanced.
\end{itemize}
\end{lemma}

\begin{proof}
(i)
By Lemma \ref{lem:WWsWprepre} there exists a nonzero $\alpha \in \F$
such that 
$\alpha W W^* W = W^* W W^*$.
Thus the pair $\alpha W,  W^*$ is balanced.

(ii) 
Similar.

(iii)
By Definition \ref{def:balanced}.
\end{proof}

For the rest of this section, assume that $W,W^*$ is balanced.

\begin{lemma}    \label{lem:sigma}    \samepage
\ifDRAFT {\rm lem:sigma}. \fi
The automorphism $\sigma$ swaps $W$ and $W^*$.
\end{lemma}

\begin{proof}
In \eqref{eq:WWsWpre}, multiply each side on the left by $W$ to get
\[
   W W W^* W = W W^* W W^*.
\]
In this line, multiply each side on the left by $(W W^* W)^{-1}$ to find that
$\sigma$ sends $W \mapsto W^*$.
By this and Lemma \ref{lem:s2}, $\sigma$ sends $W^* \mapsto W$.
\end{proof}

Let $\{E_i\}_{i=0}^d$ denote a standard ordering of the primitive idempotents of $A$.
Define 
\begin{align}
  E^*_i &= E_i^\sigma  &&   (0 \leq i \leq d).             \label{eq:defEsi}
\end{align}
Observe that the sequence
\begin{equation}
\Phi = (A; \{E_i\}_{i=0}^d; A^*; \{E^*_i\}_{i=0}^d)     \label{eq:Phi2}
\end{equation}
is a Leonard system on $V$.

\begin{lemma}    \label{lem:Phisd}    \samepage
\ifDRAFT {\rm lem:Phisd}. \fi
The Leonard system $\Phi$ is self-dual with duality $\sigma$.
\end{lemma}

\begin{proof}
By Lemma \ref{lem:PSL2}, $\sigma$ swaps $A$ and $A^*$.
By \eqref{eq:defEsi} and Lemma \ref{lem:s2}, $\sigma$ swaps $E_i$ and $E^*_i$ for
$0 \leq i \leq d$.
The result follows.
\end{proof}

\begin{lemma}    \label{lem:EiWsW}    \samepage
\ifDRAFT {\rm lem:EiWsW}. \fi
For $0 \leq i \leq d$,
\begin{align}
  W^* W E^*_i  &= E_i W^* W, &
  W W^* E_i     &= E^*_i W W^*.        \label{eq:EiWsW}
\end{align}
\end{lemma}

\begin{proof}
First we obtain the equation on the left in \eqref{eq:EiWsW}.
By Lemma \ref{lem:PSL} and \eqref{eq:defEsi},
\[
   W W^* W E^*_i  =  E_i W W^* W.
\]
In this line, multiply each side on the left by $W^{-1}$.
Simplify the result using $E_i \in \gen{A}$ and  Lemma \ref{lem:AW}
to get the equation on the left in \eqref{eq:EiWsW}.
In this equation, apply $\dagger$ to each side to get 
the equation on the right in \eqref{eq:EiWsW}.
\end{proof}

\begin{lemma}    \label{lem:WtiEi}    \samepage
\ifDRAFT {\rm lem:WtiEi}. \fi
There exist nonzero scalars $f$,  $\{\tau_i\}_{i=0}^d$ in $\F$ such that  $\tau_0=1$ and
\begin{align}
  W &= f \sum_{i=0}^d \tau_i E_i,    &
  W^* &= f \sum_{i=0}^d \tau_i E^*_i.                \label{eq:WtiEi}
\end{align}
\end{lemma}

\begin{proof}
First we obtain the equation on the left in \eqref{eq:WtiEi}.
By $W \in \gen{A}$ and since $\{E_i\}_{i=0}^d$ is a basis for $\gen{A}$,
there exist scalars $\{\alpha_i\}_{i=0}^d$ in $\F$ such that
$W = \sum_{i=0}^d \alpha_i E_i$.
The scalars $\{\alpha_i\}_{i=0}^d$ are nonzero since $W$ is invertible.
Define $f=\alpha_0$ and $\tau_i = \alpha_0^{-1} \alpha_i$ for $0 \leq i \leq d$.
Then the equation on the left in \eqref{eq:WtiEi} holds.
By construction the scalars $f$, $\{\tau_i\}_{i=0}^d$ are all nonzero and $\tau_0 = 1$.
In the equation on the left in \eqref{eq:WtiEi},
apply $\sigma$ to each side,
and use Lemma \ref{lem:sigma} and \eqref{eq:defEsi} to get
the equation on the right in \eqref{eq:WtiEi}.
\end{proof}

\begin{note}
The scalars $\{\tau_i\}_{i=0}^d$ do not change under the 
adjustments of Lemma  \ref{lem:balanced}(iii),
but $f$ can be adjusted to have any nonzero value.
\end{note}

\begin{lemma}     \label{lem:Winv}    \samepage
\ifDRAFT {\rm lem:Winv}. \fi
We have 
\begin{align}
 W^{-1} &=  f^{-1} \sum_{i=0}^d \tau_i^{-1} E_i, &
 (W^*)^{-1} &=  f^{-1} \sum_{i=0}^d \tau_i^{-1} E^*_i.        \label{eq:WinvtiinvEi}
\end{align}
\end{lemma}

\begin{proof}
Follows from \eqref{eq:WtiEi}.
\end{proof}

Recall the map $\pi$ from Definition \ref{def:pipis}.

\begin{lemma}    \label{lem:piWf}    \samepage
\ifDRAFT {\rm lem:piWf}. \fi
We have
\begin{align}
  \pi (W) &= f,  \qquad\qquad\qquad  \pi^* (W^*) = f.              \label{eq:piWf}
\end{align}
\end{lemma}

\begin{proof}
Apply $\pi$ to the equation on the left in \eqref{eq:WtiEi} and use Lemma \ref{lem:piY}
to get the equation on the left in \eqref{eq:piWf}.
The equation on the right in \eqref{eq:piWf} is similarly obtained.
\end{proof}

\begin{lemma}    \label{lem:defgamma}    \samepage
\ifDRAFT {\rm lem:defgamma}.  \fi
There exists $\gamma \in \F$ such that
\begin{equation}
  \pi(W^*) =  f \gamma,   \qquad\qquad\qquad
  \pi^*(W) = f \gamma.                                      \label{eq:fgamma}
\end{equation}
\end{lemma}

\begin{proof}
Define $\gamma = f^{-1} \pi(W^*)$.
Then the equation on the left in \eqref{eq:fgamma} holds.
By Lemma \ref{lem:sigma}, $W^\sigma = W^*$.
By Lemma \ref{lem:piYsigma}, $\pi(W^\sigma) = \pi^*(W)$.
By these comments we get the equation on the right in \eqref{eq:fgamma}.
\end{proof}

\begin{note}    \label{note:gamma}    \samepage
\ifDRAFT {\rm note:gamma}. \fi
The scalar $\gamma$ does not change under the adjustments
of  Lemma \ref{lem:balanced}(iii).
\end{note}

\begin{lemma}    \label{lem:piWWs2}    \samepage
\ifDRAFT {\rm lem:piWWs2}. \fi
We have
\begin{align}
\pi(W W^*) &= f^2 \gamma,   & \pi(W^* W) &=  f^2 \gamma,   \label{eq:piWWsfgamma}
\\  
\pi^*(W W^*) &= f^2 \gamma,  & \pi^*(W^* W) &= f^2 \gamma.    \label{eq:piWWsfgamma2}
\end{align}
\end{lemma}

\begin{proof}
By Lemma \ref{lem:piYZ}(i), $\pi(W W^*) = \pi(W^* W) = \pi(W) \pi(W^*)$.
By this and \eqref{eq:piWf}, \eqref{eq:fgamma},
we obtain \eqref{eq:piWWsfgamma}.
The equations \eqref{eq:piWWsfgamma2} are similarly obtained.
\end{proof}

Recall the map  $\rho$ from Lemma \ref{lem:rho}.

\begin{lemma}   \label{lem:WrhoWsrhos}    \samepage
\ifDRAFT {\rm lem:WrhoWsrhos}. \fi
We have 
\begin{align}
  W^\rho &= f^2 \gamma \, (W^*)^{-1},  &
 (W^*)^{\rho^*} &= f^2 \gamma \, W^{-1}.    \label{eq:WrhoWsrhos}
\end{align}
\end{lemma}

\begin{proof}
First we obtain the equation on the left in \eqref{eq:WrhoWsrhos}.
In the equation on the left in \eqref{eq:EiWsW} for $i=0$,
multiply each side on the right by $E_0$ to get 
\[
    W^* W E^*_0 E_0 = E_0 W^* W E_0.
\]
By Lemma \ref{lem:E0YE0}, $E_0 W^* W E_0 = \pi(W^* W) E_0$.
By these comments and  \eqref{eq:piWWsfgamma},
\begin{align*}
 W^* W E^*_0 E_0 &= f^2 \gamma \, E_0. 
\end{align*}
In this line, multiply each side on the left by $(W^*)^{-1}$ to get
\begin{align*}
 W E^*_0 E_0 &= f^2 \gamma \, (W^*)^{-1} E_0. 
\end{align*}
By this and Lemmas \ref{lem:AcalAEs0}, \ref{lem:rho}
we get the equation on the left in \eqref{eq:WrhoWsrhos}.
The equation on the right in \eqref{eq:WrhoWsrhos} is similarly obtained.
\end{proof}

\begin{lemma}    \label{lem:gammanonzero}    \samepage
\ifDRAFT {\rm lem:gammanonzero}. \fi
We have $\gamma \neq 0$.
\end{lemma}

\begin{proof}
We have $W^\rho \neq 0$ since $W \neq 0$ and $\rho$ is bijective.
By this and \eqref{eq:WrhoWsrhos} we get $\gamma \neq 0$.
\end{proof}

We now give our main result in this section.
Recall the elements $\{A_i\}_{i=0}^d$ from \eqref{eq:defAiAsi}
and the scalar $\nu$ from \eqref{eq:defnu}.

\begin{theorem}    \label{thm:WtiinvAi}     \samepage
\ifDRAFT {\rm thm:WtiinvAi}. \fi
We have
\begin{align}
  W &= f \gamma \sum_{i=0}^d \tau_i^{-1} A_i,  &
  W^{-1} &= \nu^{-1} f^{-1} \gamma^{-1} \sum_{i=0}^d \tau_i A_i,                   \label{eq:WWinv}
\\
  W^* &= f \gamma \sum_{i=0}^d \tau_i^{-1} A^*_i,  &
  (W^*)^{-1} &= \nu^{-1} f^{-1} \gamma^{-1} \sum_{i=0}^d \tau_i A^*_i.                  \label{eq:WsWsinv}
\end{align}
\end{theorem}

\begin{proof}
First we obtain the equation on the left in \eqref{eq:WWinv}.
By  \eqref{eq:WinvtiinvEi} and \eqref{eq:WrhoWsrhos},
\[
   W^\rho = f \gamma \sum_{i=0}^d \tau_i^{-1} E^*_i.
\]
In this line, apply $\rho^*$ to each side,
and use Lemmas \ref{lem:rhorhospre}, \ref{lem:rhorhos} to get 
the equation on the left in \eqref{eq:WWinv}.
Next we obtain the equation on the right in \eqref{eq:WWinv}.
By \eqref{eq:WrhoWsrhos},
$W^{-1} = f^{-2} \gamma^{-1} (W^*)^{\rho^*}$.
In the equation on the right in \eqref{eq:WtiEi},
apply $\rho^*$ to each side and use Lemma \ref{lem:rhorhos}
to get
$(W^*)^{\rho^*} = f \nu^{-1} \sum_{i=0}^d \tau_i A_i$.
By these comments we get the equation on the right in \eqref{eq:WWinv}.
The equations \eqref{eq:WsWsinv} are similarly obtained.
\end{proof}

We mention a lemma for later use.
Recall the scalars $\{k_i\}_{i=0}^d$  from \eqref{eq:defki}.

\begin{lemma}    \label{lem:piW}    \samepage
\ifDRAFT {\rm lem:piW}. \fi
We have
\begin{equation}
   \gamma = \nu^{-1} \sum_{i=0}^d k_i \tau_i.             \label{eq:gamma}
\end{equation}
\end{lemma}

\begin{proof}
In the equation on the right in \eqref{eq:WtiEi}, apply $\pi$ to each side,
and use \eqref{eq:defki}, \eqref{eq:fgamma}.
\end{proof}

\begin{note}
In \cite[Theorem 1.13]{C:spinLP}, the spin Leonard pairs are classified.
By that classification, there are five families of spin Leonard pairs,
called type I--V; see \cite[Lemmas 1.7--1.11]{C:spinLP}.
For each family, 
the eigenvalue sequence and the intersection numbers are explicitly given.
In terms of \cite[Section 35]{T:survey},
type I is of $q$-Racah type,
type II is of Racah type,
type III is of Krawtchouk type,
type IV is of Bannai-Ito type with odd diameter,
and type V is of Bannai-Ito type with even diameter.
In \cite[Theorem 1.18]{C:spinLP},
for each spin Leonard pair $A,A^*$ the corresponding Boltzmann pairs are given.
We say that two Boltzmann pairs for $A,A^*$
are {\em equivalent} whenever one is obtained from the other
by the adjustments in Lemma \ref{lem:WinvWsinv}.
The Leonard pair $A, A^*$ has precisely one equivalence class of Boltzmann pairs,
unless $A, A^*$ has type V and the intersection number $a_i$ is independent of
$i$ for $0 \leq i \leq d$. 
In this case $A, A^*$ has precisely two equivalence classes of Boltzmann pairs. 
For one of the equivalence classes we have 
$\tau_i = (-1)^{\lceil i/2 \rceil}$ for $0 \leq i \leq d$,
and for the other equivalence class, we have 
$\tau_i = (-1)^{\lfloor i/2 \rfloor}$ for $0 \leq i \leq d$.
\end{note}

\section{Leonard pairs of $q$-Racah type}
\label{sec:qRacah}

In this section we consider a special case of Leonard pair,
said to have $q$-Racah type \cite{Huang}.
This is the ``most general'' type of Leonard pair.
Throughout this section the following notation is in effect.
Fix a nonzero $q \in \F$ such that $q^4 \neq 1$.
Let $\Phi = (A; \{E_i\}_{i=0}^d; A^*; \{E^*_i\}_{i=0}^d)$ denote a Leonard system on $V$
with parameter array $(\{\th_i\}_{i=0}^d; \{\th^*_i\}_{i=0}^d; \{\vphi_i\}_{i=1}^d; \{\phi_i\}_{i=1}^d)$.
To avoid trivialities we assume $d \geq 1$.

\begin{definition}    {\rm \cite[Definition 5.1]{Huang} }
\label{def:qRacah}    \samepage
\ifDRAFT {\rm def:qRacah}. \fi
We say that $\Phi$ has {\em $q$-Racah type} whenever
there exist nonzero $a$, $b \in \F$ such that
\begin{align}
 \th_i &= a q^{2i-d} + a^{-1} q^{d-2i}  && (0 \leq i \leq d),     \label{eq:qRacahthi}
\\
 \th^*_i &= b q^{2i-d} + b^{-1}q^{d-2i} &&  (0 \leq i \leq d).   \label{eq:qRacahthsi}
\end{align}
In this case the scalars $a$, $b$ are unique.
\end{definition}

\begin{lemma}     {\rm \cite[Corollary 6.4]{Huang} }
\label{lem:c}    \samepage
\ifDRAFT {\rm lem:c}. \fi
Assume that $\Phi$ has $q$-Racah type.
Let $a$, $b \in \F$ denote nonzero scalars that satisfy \eqref{eq:qRacahthi}, \eqref{eq:qRacahthsi}.
Then there exists a nonzero $c \in \F$ such that for $1 \leq i \leq d$,
\begin{align}
 \vphi_i &= 
 a^{-1}b^{-1} q^{d+1}(q^i - q^{-i})(q^{i-d-1} - q^{d-i+1})(q^{-i}-a b c q^{i-d-1})(q^{-i}-a b c^{-1}q^{i-d-1}),  \label{eq:vphii}
\\
 \phi_i &=
 a b^{-1} q^{d+1}(q^i - q^{-i})(q^{i-d-1} - q^{d-i+1})(q^{-i}-a^{-1} b c q^{i-d-1})(q^{-i}-a^{-1} b c^{-1}q^{i-d-1}). \label{eq:phii}
\end{align}
Moreover, $c$ is unique up to inverse.
\end{lemma}

\begin{definition}    \label{def:Huangdata}    \samepage
\ifDRAFT {\rm def:Huangdata}. \fi
Assume that $\Phi$ has $q$-Racah type.
Let $a$, $b \in \F$ denote nonzero scalars that satisfy \eqref{eq:qRacahthi}, \eqref{eq:qRacahthsi},
and let $c \in \F$ denote a nonzero scalar that satisfies \eqref{eq:vphii}, \eqref{eq:phii}.
We call the sequence $(a,b,c,d)$ a {\em Huang data} of $\Phi$.
\end{definition}

\begin{note}    \label{note:Huang}    \samepage
\ifDRAFT {\rm note:Huang}. \fi
Assume that $\Phi$ has $q$-Racah type with Huang data $(a,b,c,d)$.
Then $(a,b,c^{-1},d)$ is a Huang data for $\Phi$, and $\Phi$ has no further
Huang data.
\end{note}

For the rest of this section, assume that
$\Phi$ has $q$-Racah type with Huang data $(a,b,c,d)$.

\begin{lemma}   {\rm \cite[Definition 7.1]{Huang} }
\label{lem:condabc}   \samepage
\ifDRAFT {\rm lem:condabc}. \fi
The following {\rm (i)--(iii)} hold:
\begin{itemize}
\item[\rm (i)]
$q^{2i} \neq 1$ for $1 \leq i \leq d$;
\item[\rm (ii)]
Neither of $a^2$, $b^2$ is among $q^{2d-2}, q^{2d-4}, \ldots, q^{2-2d}$;
\item[\rm (iii)]
None of $a b c$, $a^{-1} b c$, $a b^{-1} c$, $a b c^{-1}$ is among $q^{d-1}, q^{d-3}, \ldots, q^{1-d}$.
\end{itemize}
\end{lemma}

\begin{lemma}    \label{lem:Huang2}    \samepage
\ifDRAFT {\rm lem:Huang2}. \fi
The Leonard system
\begin{equation}
  (-A; \{E_i\}_{i=0}^d;  - A^*; \{E^*_i\}_{i=0}^d)                 \label{eq:LS-A}
\end{equation}
has $q$-Racah type with Huang data $(-a, -b, c, d)$.
\end{lemma}

\begin{proof}
Use Definitions \ref{def:qRacah}, \ref{def:Huangdata}.
\end{proof}

Recall the intersection numbers  $\{b_i\}_{i=0}^{d-1}$, $\{c_i\}_{i=1}^d$ of $\Phi$
from \eqref{eq:matrixAAs}.

\begin{lemma}   \label{lem:bici}    \samepage
\ifDRAFT {\rm lem:bici}. \fi
We have
\begin{align*}
b_i &=
 \frac{ (q^{i-d}-q^{d-i}) (b q^{i-d} - b^{-1} q^{d-i}) ( a b q^i - c q^{d-i-1})(a b q^i - c^{-1} q^{d-i-1}) }
        {a b q^{d-1} (b q^{2i-d} - b^{-1} q^{d-2i}) (b q^{2i-d+1} - b^{-1} q^{d-2i-1} ) },        
\\
c_i &= 
 \frac{ (q^i-q^{-i}) (b q^i - b^{-1} q^{-i}) (b q^i - a c q^{d-i+1}) (b q^i - a c^{-1} q^{d-i+1}) }
        {a b q^{d+1} (b q^{2i-d} - b^{-1} q^{d-2i}) (b q^{2i-d-1} - b^{-1} q^{d-2i+1}) }        
\end{align*}
for $1 \leq i \leq d-1$, and
\begin{align*}
 b_0 &= 
  \frac{(q^{-d} - q^d) (a b - c q^{d-1}) (a b - c^{-1} q^{d-1} ) }
         {a b q^{d-1} (b q^{1-d} - b^{-1} q^{d-1}) },                                          
\\
 c_d &= 
  \frac{(q^d - q^{-d}) (a c - b q^{d-1}) (a c^{-1} - b q^{d-1}) }
         {a b q^{d-1} (b q^{d-1} - b^{-1} q^{1-d}) }.                                    
\end{align*}
\end{lemma}

\begin{proof}
Evaluate \eqref{eq:ci0}, \eqref{eq:bi0} using \eqref{eq:qRacahthi}--\eqref{eq:phii}.
\end{proof}

We recall some notation.
For $\alpha \in \F$,
\begin{align*}
  (\alpha; q)_n &= (1-\alpha) (1-\alpha q) \cdots (1- \alpha q^{n-1})    &  n &= 0, 1, 2, \ldots
\end{align*}
We interpret $(\alpha;q)_0 = 1$.
Recall the scalar $\nu$ be from \eqref{eq:nu}.

\begin{lemma}   \label{lem:qRacahnu1}    \samepage
\ifDRAFT {\rm lem:qRacahnu1}. \fi
We have
\begin{equation*}
 \nu = \frac{a^d b^d c^d q^{d(1-d)} (a^{-2};q^2)_d (b^{-2}; q^2)_d }
                 { (a b^{-1} c q^{1-d}; q^2)_d (a^{-1} b c q^{1-d}; q^2)_d}.
\end{equation*}
\end{lemma}

\begin{proof}
Evaluate \eqref{eq:nu2} using \eqref{eq:qRacahthi}--\eqref{eq:phii}.
\end{proof}

Recall the scalars  $\{k_i\}_{i=0}^d$ from \eqref{eq:defki}.

\begin{lemma}       \label{lem:kiqRacah}    \samepage
\ifDRAFT {\rm lem:kiqRacah}. \fi
We have $k_0 = 1$, and
\[
k_i =
 \frac{q^{2 i d} (1-b^2 q^{4i-2d}) \, (b^2 q^{2-2d};q^2)_i \, (q^{-2d}; q^2)_i \, (a b c q^{1-d};q^2)_i \, (a b c^{-1} q^{1-d}; q^2)_i }
        {a^{2i} (1-b^2 q^{2i-2d}) \, (b^2 q^2; q^2)_i \, (q^2; q^2)_i \, (a^{-1} b c q^{1-d}; q^2)_i \, (a^{-1} b c^{-1} q^{1-d}; q^2)_i}
\]
for $1 \leq i \leq d-1$, and
\[
k_d = 
 \frac{q^{2 d^2} (b^2 q^{2-2d};q^2)_d \, (q^{-2d}; q^2)_d \, (a b c q^{1-d};q^2)_d \, (a b c^{-1} q^{1-d}; q^2)_d }
        {a^{2d} (b^2; q^2)_d \, (q^2; q^2)_d \,  (a^{-1} b c q^{1-d}; q^2)_d \, (a^{-1} b c^{-1} q^{1-d}; q^2)_d}.
\]
\end{lemma}

\begin{proof}
Evaluate \eqref{eq:kiparam} using \eqref{eq:qRacahthi}--\eqref{eq:phii}.
\end{proof}

We now turn our attention to the self-dual case.

\begin{lemma}    \label{lem:selfdualab}    \samepage
\ifDRAFT {\rm lem:selfdualab}. \fi
The Leonard system $\Phi$ is self-dual if and only if $a=b$.
\end{lemma}

\begin{proof}
By Lemma \ref{lem:selfdualparam} and Definition \ref{def:qRacah}.
\end{proof}

We now turn our attention to the spin case.

\begin{lemma}    {\rm \cite[Theorem 1.13]{C:spinLP} }
\label{lem:spinLP0}    \samepage
\ifDRAFT {\rm lem:spinLP0}. \fi
The following hold.
\begin{itemize}
\item[\rm (i)]
Assume $d \geq 3$, and $A,A^*$ is a spin Leonard pair.
Then $b=a$ and $c \in \{a, a^{-1}, -a, -a^{-1} \}$.
\item[\rm (ii)]
Assume that $b=a$ and $c \in \{a, a^{-1}, -a, -a^{-1} \}$.
Then $A,A^*$ is a spin Leonard pair.
\end{itemize}
\end{lemma}

\begin{note}    \label{note:Huang2}    \samepage
\ifDRAFT {\rm note:Huang2}. \fi
We have a comment about Lemma \ref{lem:spinLP0}.
Assume that $b=a$ and $c \in \{-a, -a^{-1} \}$.
By Note \ref{note:Huang} we may assume that $c=-a$.
Replacing $\Phi$ by \eqref{eq:LS-A} we may assume that $c=a$.
\end{note}

In view of Note \ref{note:Huang2}
we focus on the case $a=b=c$.

\begin{lemma}    \label{lem:qRacahbi}    \samepage
\ifDRAFT {\rm lem:qRacahbi}. \fi
Assume that $a=b=c$.
Then
\begin{align*}
b_0 &= \frac{(q^{-d}-q^d)(a^3 -  q^{d-1}) } { a (a+  q^{d-1}) },                       
\\
b_i &= 
  \frac{(q^{i-d} - q^{d-i} ) (a q^{i-d} - a^{-1} q^{d-i} ) (a^3 - q^{d-2i-1} ) }
         {a (a q^{2i-d} - a^{-1} q^{d-2i}) (a + q^{d-2i-1}) }
                                                       &&  (1 \leq i \leq d-1),             
\\
c_i &=
 \frac{a (q^i - q^{-i})(a q^i - a^{-1}q^{-i})(a^{-1} - q^{d-2i+1} ) }
        {(a q^{2i-d} - a^{-1} q^{d-2i} ) (a + q^{d-2i+1} ) }
                                                            &&  (1 \leq i \leq d-1),       
\\
c_d &= \frac{(q^{-d}-q^d)(a-q^{d-1}) }{ q^{d-1} (a + q^{1-d}) }.             
\end{align*}
\end{lemma}

\begin{proof}
Set $a=b=c$ in Lemma \ref{lem:bici}.
\end{proof}

\begin{lemma}    \label{lem:qRacahnu}    \samepage
\ifDRAFT {\rm lem:qRacahnu}. \fi
Assume that $a=b=c$.
Then
\begin{equation}
 \nu = \frac{a^{3d} (a^{-2};q^2)_d^2 } {q^{d(d-1)} (a q^{1-d};q^2)_d^2}.   \label{eq:nu3}
\end{equation}
\end{lemma}
 
\begin{proof}
Set $a=b=c$ in Lemma \ref{lem:qRacahnu1}.
\end{proof}

\begin{lemma}       \label{lem:ki3}    \samepage
\ifDRAFT {\rm lem:ki3}. \fi
Assume that $a=b=c$.
Then $k_0 = 1$, and
\[
k_i =
 \frac{q^{2 i d} (1-a^2 q^{4i-2d}) \, (a^2 q^{2-2d};q^2)_i \, (q^{-2d}; q^2)_i \, (a^3 q^{1-d};q^2)_i  }
        {a^{2i} (1-a^2 q^{2i-2d}) \, (a^2 q^2; q^2)_i \,  (q^2; q^2)_i \,  (a^{-1} q^{1-d}; q^2)_i}
\]
for $1 \leq i \leq d-1$, and
\[
k_d = 
 \frac{q^{2 d^2} (a^2 q^{2-2d};q^2)_d \, (q^{-2d}; q^2)_d \, (a^3 q^{1-d};q^2)_d  }
        {a^{2d} (a^2; q^2)_d \, (q^2; q^2)_d  \, (a^{-1} q^{1-d}; q^2)_d}.
\]
\end{lemma}

\begin{proof}
Set $a=b=c$ in  Lemma \ref{lem:kiqRacah}.
\end{proof}

\begin{lemma}  {\rm   \cite[Theorem 1.18]{C:spinLP} }
\label{lem:W}    \samepage
\ifDRAFT {\rm lem:W}. \fi
Assume that $a=b=c$.
Let $f$, $\{\tau_i\}_{i=0}^d$ denote nonzero scalars in $\F$ such that $\tau_0 = 1$. 
Define 
\begin{align}
   W &= f \sum_{i=0}^d  \tau_i E_i, &
   W^* &= f \sum_{i=0}^d \tau_i E^*_i.              \label{eq:defWpre}
\end{align}
\begin{itemize}
\item[\rm (i)]
Assume that $W, W^*$ is a Boltzmann pair for  $A,A^*$.
Then one of the following holds provided that $d \geq 3$:
\begin{align}
  \tau_i &= (-1)^i a^{-i} q^{i(d-i)}   &&  (0 \leq i \leq d),          \label{eq:defti}
\\
  \tau_i ^{-1} &=  (-1)^i a^{-i} q^{i(d-i)}  &&  (0 \leq i \leq d).    \label{eq:deftiinv}
\end{align}
\item[\rm (ii)]
Assume that one of \eqref{eq:defti}, \eqref{eq:deftiinv} holds.
Then  $W, W^*$ is a balanced Boltzmann pair for  $A,A^*$.
\end{itemize}
\end{lemma}

\begin{note}    \label{note:ti}    \samepage
\ifDRAFT {\rm note:ti}. \fi
Referring to Lemma \ref{lem:W},
if \eqref{eq:deftiinv} holds then \eqref{eq:defti} holds for the
Boltzmann pair $W^{-1}, (W^*)^{-1}$.
For this reason we will focus on \eqref{eq:defti}.
\end{note}

\begin{lemma}    \label{lem:new}    \samepage
\ifDRAFT {\rm lem:new}. \fi
Assume that $b=a$  and define $\{\tau_i\}_{i=0}^d$ by \eqref{eq:defti}.
Let $0 \neq f \in \F$
and define $W$, $W^*$ as in \eqref{eq:defWpre}.
Assume that $W,W^*$ is a Boltzmann pair for $A,A^*$.
Then $c \in \{a, a^{-1} \}$.
\end{lemma}

\begin{proof}
First assume that $d \geq 3$.
By way of contradiction assume that $c \not\in \{a, a^{-1} \}$.
By Lemma \ref{lem:spinLP0}(i) we have $c \in \{-a, - a^{-1} \}$
and $\text{\rm Char}(\F) \neq 2$.
By Note \ref{note:Huang} we may assume that $c = -a$.
Thus $\Phi$ has Huang data $(a,a,-a,d)$.
By Lemma \ref{lem:Huang2} the Leonard system \eqref{eq:LS-A}
has Huang data $(-a,-a,-a,d)$.
By Definition \ref{def:balanced} and Lemma \ref{lem:affine}
the pair $W, W^*$ is a Boltzmann pair for $-A, -A^*$.
Applying Lemma \ref{lem:W}(i) to the Leonard system \eqref{eq:LS-A}, one of the following holds:
\begin{align*}
  \tau_i &= (-1)^i (-a)^{-i} q^{i (d-i)}   &&   (0  \leq i \leq d),  
\\
  \tau_i ^{-1} &=  (-1)^i (-a)^{-i} q^{i (d-i)}   &&   (0  \leq i \leq d).  
\end{align*}
In either case, comparing the values of $\tau_1$, $\tau_3$ with \eqref{eq:defti} 
we get a contradiction by $\text{\rm Char}(\F) \neq 2$ and Lemma \ref{lem:condabc}.
We have shown the assertion for the case $d \geq 3$.
Next assume that $d \leq 2$.
Consider the matrices that represent the elements
$A$, $A^*$, $\{E_i\}_{i=0}^d$, $\{E^*_i\}_{i=0}^d$, $W$, $W^*$
with respect to the $\Phi$-split basis for $V$.
We represent these matrices in terms of $a$, $c$, $q$ using \eqref{eq:qRacahthi}--\eqref{eq:phii}.
The matrices representing $A,A^*$ are as in Lemma \ref{lem:split}.
The matrices representing $\{E_i\}_{i=0}^d$, $\{E^*_i\}_{i=0}^d$ can be
obtained using \eqref{eq:Ei0}.
The matrices representing $W, W^*$ are obtained by \eqref{eq:defWpre}.
Now compute each side of \eqref{eq:AWsW} in matrix form,
and compare the $(d,0)$-entries to get $(a-c) (a c -1)=0$.
Thus $c \in \{a, a^{-1} \}$.
\end{proof}

The following lemma is a reformulation of a special case of \cite[2.22]{GR}.

\begin{lemma}  
\label{lem:GR}    \samepage
\ifDRAFT {\rm lem:GR}. \fi
Let $\alpha$, $\beta$ denote commuting indeterminates.
Then
\begin{align}
\sum_{i=0}^{d}
 \frac{(1-\alpha q^{-4i}) \, (\alpha^{-1} q^2; q^2)_i \, (\beta^{-1}; q^2)_i \, (q^{-2d}; q^2)_i \, q^{i(2d-i+3)}  (-\beta)^i }
        {(1-\alpha q^{-2i})\, (q^2,q^2)_i \, (\beta q^2/\alpha; q^2)_i \, (q^{2d+2}/\alpha; q^2)_i  }
 = \frac{(\alpha q^{-2d}, q^2)_d }
            {(\alpha q^{-2d}/\beta, q^2)_d }.            \label{eq:GR}
\end{align}
\end{lemma}
 
\begin{proof}
Using 
\[
   (x; q^2)_i = (x^{-1};q^{-2})_i (- x)^i q^{i(i-1)}.
\]
one routinely checks that the left-hand side of \eqref{eq:GR} is equal to
\[
\sum_{i=0}^d
  \frac{ (\alpha q^{-2}; q^{-2})_i \, (1- \alpha q^{-4i}) \, (\beta; q^{-2})_i \, (q^{2d}; q^{-2})_i }
         {(q^{-2}; q^{-2})_i \, (1-\alpha q^{-2i}) \, (\alpha q^{-2}/\beta; q^{-2})_i  \, (\alpha q^{-2d-2}; q^{-2})_i }
  \left( - \frac{ \alpha q^{-2d-2} } { \beta } \right)^i q^{i(1-i)}.
\]
By this and \cite[2.22]{GR}, the left-hand side of \eqref{eq:GR} is equal to
\[
    \frac{ (\alpha q^{-2}; q^{-2} )_d }
           { (\alpha q^{-2}/\beta; q^{-2})_d}.
\]
One checks that
\[
 (\alpha q^{-2}; q^{-2})_d = (\alpha q^{-2d}; q^2)_d,   \qquad\qquad
 (\alpha q^{-2}/\beta; q^{-2})_d = (\alpha q^{-2d}/\beta; q^2)_d.
\]
The result follows.
\end{proof}

\begin{lemma}   \label{lem:sumkiti}    \samepage
\ifDRAFT {\rm lem:sumkiti}. \fi
Referring to Lemma \ref{lem:W},
assume that the $\{\tau_i\}_{i=0}^d$ satisfy \eqref{eq:defti}.
Then
\begin{align}
 \sum_{i=0}^d k_i \tau_i &=      
  \frac{ (a^{-2}; q^2)_d } { (a q^{1-d}; q^2)_d }.                      \label{eq:sumkiti3}
\end{align}
Moreover the scalar $\gamma$ from Lemma \ref{lem:defgamma} satisfies
\begin{align}
\gamma &= \frac{q^{d(d-1)} (a q^{1-d};q^2)_d} {a^{3d} (a^{-2};q^2)_d }.   \label{eq:gamma3}
\end{align}
\end{lemma}

\begin{proof}
We first show \eqref{eq:sumkiti3}.
Using the data in Lemma \ref{lem:ki3}
we find that the left-hand side of \eqref{eq:sumkiti3} is equal to the sum
\begin{equation}
\begin{array}{l}
\displaystyle
1 +
 \sum_{i=1}^{d-1}
 \frac{(1-\alpha^{-1} q^{4i}) \, (\alpha^{-1} q^2; q^2)_i \, (\beta^{-1}; q^2)_i \, (q^{-2d}; q^2)_i \, (-\beta)^i \, q^{i(2d-i+1)}}
        {(1-\alpha^{-1} q^{2i})\, (q^2,q^2)_i \, (\beta q^2/\alpha; q^2)_i \, (q^{2d+2}/\alpha; q^2)_i   }
\\
\displaystyle
\quad  +
  \frac{(\alpha^{-1} q^2; q^2)_d \, (\beta^{-1}; q^2)_d \, (q^{-2d}; q^2)_d \, (- \beta)^d \,  q^{d(d+1)} }
        {(q^2; q^2)_d \, (\beta q^2/\alpha; q^2)_d \, (q^{2d}/\alpha; q^2)_d },  
\end{array} 
            \label{eq:sum}
\end{equation}                    
where  $\alpha = a^{-2} q^{2d}$,  $\beta = a^{-3} q^{d-1}$.
One routinely checks that  the sum \eqref{eq:sum} is equal to
the the left-hand side of \eqref{eq:GR}.
Thus \eqref{eq:sumkiti3} holds by Lemma \ref{lem:GR}.
To get \eqref{eq:gamma3},
evaluate \eqref{eq:gamma} using \eqref{eq:nu3}, \eqref{eq:sumkiti3}.
\end{proof}

\section{Type II matrices and spin models}
\label{sec:spin}

We now turn our attention to type II matrices and spin models.
For the rest of this paper,
fix a finite nonempty set $X$.
Let $\Mat$ denote the $\C$-algebra consisting of the
matrices that have all entries in $\C$ and whose rows and columns
are indexed by $X$.
For $\bR \in \Mat$ and $x$, $y \in X$ the $(x,y)$-entry of $\bR$ is denoted by $\bR (x,y)$.
Let $\I$ (resp.\ $\J$) denote the identity matrix  (resp.\ all $1$'s matrix) in $\Mat$.
For $\bR, \bS \in \Mat$ let $\bR \circ \bS$ denote their Hadamard (entry-wise) product.
Let $\V$ denote the vector space over $\C$ consisting of
the column vectors whose entries are indexed by $X$.
Note that $\dim \V = |X|$.
The algebra $\Mat$ acts on $\V$ by left multiplication.
For $y \in X$ define $\widehat{y} \in \V$ that has $y$-entry $1$ and all other entries $0$.
Note that $\{\widehat{y}\}_{y \in X}$ form a basis for $\V$.
For a real number $\alpha > 0$ let $\alpha^{1/2}$ denote the positive square root of $\alpha$.

\begin{definition}     {\rm  \cite[Definition 2.1]{Jones} }
\label{def:type2}    \samepage
\ifDRAFT {\rm def:type2}. \fi
A matrix $\W \in \Mat$ is said to be {\em type II} whenever 
$\W$ is symmetric with all entries nonzero and 
\begin{align}
  \sum_{y \in X} \frac{\W (a,y)} {\W(b,y)}
 &= |X| \delta_{a,b}    &&   (a,b \in X).                \label{eq:type2}
\end{align}
\end{definition}

For a symmetric $\W \in \Mat$ with all entries nonzero,
define $\W^- \in \Mat$ by
\begin{align}
   \W^- (y,z) &= (\W(y,z) )^{-1}   &&  (y,z \in X).    \label{eq:defW-}
\end{align}
Referring to the equation \eqref{eq:type2},
for $a,b \in X$ the left-hand side is equal to the $(a,b)$-entry of $\W \W^-$.

\begin{lemma}    \label{lem:type2W-}    \samepage
\ifDRAFT {\rm lem:type2W-}. \fi
Assume $\W \in \Mat$ is symmetric with all entries nonzero.
Then the following {\rm (i)--(iii)} are equivalent:
\begin{itemize}
\item[\rm (i)]
$\W$ is type II;
\item[\rm (ii)]
$\W \W^- = |X| \,\I$;
\item[\rm (iii)]
$\W \W^- = \alpha \I$ for some $\alpha \in \C$.
\end{itemize}
\end{lemma}

\begin{proof}
(i) $\Rightarrow$ (ii)
By \eqref{eq:type2} and \eqref{eq:defW-}.

(ii) $\Rightarrow$ (iii)
Clear.

(iii) $\Rightarrow$ (i)
By the comment below \eqref{eq:defW-}.
\end{proof}

\begin{note}    \label{note:type2}    \samepage
\ifDRAFT {\rm note:type2}. \fi
Assume $\W \in \Mat$ is type II.
Then $\alpha \W$ is type II for $0 \neq \alpha \in \C$.
Moreover,  $\W^{-1}$ and $\W^-$ are type II.
\end{note}

\begin{definition}  {\rm \cite[Section 1]{N:analgebra} }
\label{def:NW}    \samepage
\ifDRAFT {\rm def:NW}. \fi
Assume $\W \in \Mat$ is type II.
For $b,c \in X$ define $\text{\bf u}_{b,c} \in \V$  by
\begin{align}
    \text{\bf u}_{b,c} = \sum_{y \in X} \frac{\W(b,y)}{\W(c,y)} \, \widehat{y}.    \label{eq:defubc}
\end{align}
Define
\[
   N(\W) = \{ B \in \Mat \,|\, 
    \text{$B$ is symmetric}, \; B \mathbf{u}_{b,c} \in \C \mathbf{u}_{b,c} \text{ for all $b,c \in X$} \}.
\]
\end{definition}

\begin{lemma}  {\rm  \cite[Theorem 6]{N:analgebra} }
\label{lem:Nomuraalg}    \samepage
\ifDRAFT {\rm lem:Nomuraalg}. \fi
Assume  $\W \in \Mat$ is type II.
Then $N(\W)$ is a commutative subalgebra of $\Mat$ that contains $\J$ and is 
closed under the Hadamard product.
\end{lemma}

We have a comment.

\begin{lemma}    \label{lem:NW-}    \samepage
\ifDRAFT {\rm lem:NW-}. \fi
Assume $\W \in \Mat$ is type II. 
Then $N(\alpha \W)= N(\W)$ for $0 \neq \alpha \in \C$.
Moreover each of $N(\W^-)$, $N(\W^{-1})$ is equal to $N(\W)$.
\end{lemma}

\begin{proof}
Clearly $N(\alpha \W)=N(\W)$.
For $b,c \in X$ define $\text{\bf u}^-_{b,c} \in \V$ by
\begin{align*}
    \text{\bf u}^-_{b,c} = \sum_{y \in X} \frac{\W^-(b,y)}{\W^-(c,y)} \, \widehat{y}. 
\end{align*}
Using \eqref{eq:defW-}, \eqref{eq:defubc} we obtain
\[
 \text{\bf u}^-_{b,c} = \sum_{y \in X} \frac{\W(c,y)} {\W(b,y)} \widehat{y} = \text{\bf u}_{c,b}.
\]
By this and Definition \ref{def:NW} we find $N(\W^-) = N(\W)$.
We have $\W^- = |X| \W^{-1}$ by Lemma \ref{lem:type2W-},
so $N(\W^{-1}) = N(\W^-)$.
The result follows.
\end{proof}

\begin{definition}  {\rm  \cite[Definition 2.1]{Jones} }
\label{def:spin}    \samepage
\ifDRAFT {\rm def:spin}. \fi
A matrix $\W \in \Mat$ is called a {\em spin model} whenever
$\W$ is type II and
\begin{align}
 \sum_{y \in X} \frac{\W(a,y) \W(b,y)} {\W(c,y)}
 &= |X|^{1/2} \; \frac{\W(a,b)} {\W(a,c) \W(b,c)}   &&  (a,b,c \in X).     \label{eq:type3}
\end{align}
\end{definition}

\begin{note}
The equation \eqref{eq:type3} is often called the type III condition.
\end{note}

\begin{note}    {\rm  \cite[Section 2]{Jones} (see also \cite[Proposition 2.1]{KMW}). }
\label{note:spin}    \samepage
\ifDRAFT {\rm note:spin}. \fi
Assume $\W \in \Mat$ is a spin model.
Then $-\W$ is a spin model.
Moreover $\W^-$ is a spin model.
\end{note}

\begin{lemma}  {\rm  \cite[Lemma 8]{N:analgebra} }
\label{lem:Nomuraalg2}    \samepage
\ifDRAFT {\rm lem:Nomuraalg2}. \fi
Assume $\W \in \Mat$ is a spin model.
Then $\W \in N(\W)$.
\end{lemma}

\section{Hadamard matrices and spin models}

In this section we consider a certain class of type II matrices and spin models.

\begin{definition}    \label{def:Hmatrix}    \samepage
\ifDRAFT {\rm def:Hmatrix}. \fi
A matrix $\bH \in \Mat$ is called {\em Hadamard}
whenever every entry is $\pm 1$ and $\bH \bH^{\sf t} = |X| \,\I$.
\end{definition}

\begin{example}    \label{exam1}    \samepage
\ifDRAFT {\rm exam1}. \fi
The matrix
\[
 \bH = 
 \begin{pmatrix}
  1 & -1 & -1 & -1  \\
  -1 & 1 & -1 & -1  \\
 -1 & -1 & 1 & -1  \\
  -1 & -1 & -1 & 1
 \end{pmatrix}
\]
is Hadamard.
\end{example}

\begin{lemma}    \label{lem:Htype2}    \samepage
\ifDRAFT {\rm lem:Htype2}. \fi
A symmetric Hadamard matrix is type II.
\end{lemma}

\begin{proof}
By Definitions \ref{def:type2} and \ref{def:Hmatrix}.
\end{proof}

\begin{lemma}    \label{lem:Htype}    \samepage
\ifDRAFT {\rm lem:Htype}. \fi
For $\W \in \Mat$ and $0 \neq \alpha \in \C$ the following are equivalent:
\begin{itemize}
\item[\rm (i)]
$\W$ is type II with all entries $\pm \alpha$;
\item[\rm (ii)]
there exists a symmetric Hadamard matrix $\bH$ such that $\W = \alpha \bH$.
\end{itemize}
\end{lemma}

\begin{proof}
(i) $\Rightarrow$ (ii)
Assume that $\W$ is type II with all entries $\pm \alpha$.
Define $\bH = \alpha^{-1} \W$.
Then the entries of $\bH$ are $\pm 1$.
Moreover, by Note \ref{note:type2}, $\bH$ is type II.
Applying Lemma \ref{lem:type2W-} to $\bH$ and using $\bH^- =\bH$,
we obtain $\bH^2 = |X| \I$.
The matrix $\bH$ is symmetric since it is type II.
By these comments $\bH$ is a symmetric Hadamard matrix.

(ii) $\Rightarrow$ (i)
Assume that there exists a symmetric Hadamard matrix $\bH$ such that
$\W = \alpha \bH$.
Then $\W$ is symmetric and satisfies \eqref{eq:type2}.
So $\W$ is type II.
Clearly the entries of $\W$ are $\pm \alpha$.
\end{proof}

\begin{definition}    \label{def:Htype}   \samepage
\ifDRAFT {\rm def:Htype}. \fi
A type II matrix $\W \in \Mat$ is said to have {\em Hadamard type}
whenever there exists a symmetric Hadamard matrix $\bH$ and  $0 \neq \alpha \in \C$ 
such that $\W = \alpha \bH$.
\end{definition}

We have an example of a spin model having Hadamard type.

\begin{example}     \label{example2}    \samepage
\ifDRAFT {\rm example2}. \fi
Let the matrix $\bH$ be from Example \ref{exam1}.
Then the matrix $\W  = \sqrt{-1} \, \bH$ is a spin model of Hadamard type.
\end{example}

\begin{note}   \label{noteH}    \samepage
\ifDRAFT {\rm noteH}. \fi
Spin models of Hadamard type sometimes cause technical problems,
so occasionally we will assume that a spin model under discussion does not
have Hadamard type.
\end{note}

\section{Distance-regular graphs}
\label{sec:DRG}

We now turn our attention to distance-regular graphs
\cite{BI, BCN, DKT}.
Let $\Gamma$ denote an undirected,  connected graph, without loops or multiple edges,
with vertex set $X$.
For $x,y \in X$ let $\partial(x,y)$ denote the path-length distance between $x$ and $y$.
By the {\em diameter} of $\Gamma$ we mean $D = \max \{\partial(x,y) \,|\, x,y \in X\}$.
To avoid trivialities, we always assume $D \geq 1$.
For $x \in X$  define
\begin{align*}
 \Gamma_i (x) &= \{y \in X \,|\, \partial (x,y) = i\}   &&  (0 \leq i \leq D).
\end{align*}
The graph $\Gamma$ is said to be {\em regular}
whenever $\k = |\Gamma_1 (x)|$ is independent of $x \in X$.
In this case,
$\k$ is called the {\em valency} of $\Gamma$.
The graph $\Gamma$ is said to be {\em distance-regular} whenever
for $0 \leq h, i, j \leq D$ and vertices $x,y \in X$ at distance $\partial(x,y)=h$,
the number
\begin{align*}
  \p^h_{i j} &= |\Gamma_i(x) \cap \Gamma_j(y)|
\end{align*}
is independent of $x$, $y$.
The integers $\p^h_{i j}$ are called the {\em intersection numbers} of 
$\Gamma$.

For the rest of this section, assume that $\Gamma$ is distance-regular.
Define 
\begin{align}
 \k_i &= \p^0_{i i}   &&  (0 \leq i \leq D).                             \label{eq:defki2}
\end{align}
We have $|\Gamma_i (x)| = \k_i$ for $x \in X$ and $0 \leq i \leq D$.
Observe that $\k_0=1$, and that $\Gamma$ is regular with valency $\k=\k_1$.
Moreover
\begin{equation*} 
  \sum_{i=0}^D \k_i = |X|.                        
\end{equation*}
For $0 \leq h,i,j \leq D$ the scalar
$\p^h_{i j}$ is zero (resp.\ nonzero) if one of $h,i,j$ is
greater than (resp.\ equal to) the sum of the other two.
Define
\begin{align}
 \c_i &= \p^i_{1, i-1}  && (1 \leq i \leq D),      \label{eq:defci}
\\
 \a_i &= \p^i_{1 i} && (0 \leq i \leq D),           \label{eq:defai}
\\
 \b_i &= \p^i_{1, i+1}  && (0 \leq i \leq D-1).                \label{eq:defbi}
\end{align}
For notational convenience define $\c_0 = 0$ and $\b_D = 0$.
We have $\c_i \b_{i-1} \neq 0$ for $1 \leq i \leq D$.
Moreover
\begin{align}
\a_0 &= 0, &
\b_0 &= \k, &
\c_1 &= 1, &
\c_i + \a_i + \b_i &= \k  \quad (0 \leq i \leq D).                     \label{eq:a0b0c1}
\end{align}
By \cite[Chapter 4,  (1c)]{BCN},
\begin{align*}
 \k_i &= \frac{\b_0 \b_1 \cdots \b_{i-1} }
                {\c_1 \c_2 \cdots \c_i }               &&  (0 \leq i \leq D).       
\end{align*}

For $0 \leq i \leq D$ define $\A_i \in \Mat$ that has $(x,y)$-entry $1$
if $\partial(x,y)=i$ and $0$ if $\partial(x,y) \neq i$ $(x,y \in X)$.
Note that $\A_i$ is symmetric.
The matrix $\A_i$ is called the {\em $i^\text{\rm th}$ distance-matrix} of $\Gamma$.
Note that $\A_0 = \I$, and $\A_1$ is the adjacency matrix of $\Gamma$.
Observe that
\begin{align}
  \A_i \circ \A_j &= \delta_{i,j} \A_i  \qquad (0 \leq i,j \leq D),
  \qquad\qquad
  \sum_{i=0}^D \A_i = \J,     \label{eq:Ai2}
\\
  \A_i \A_j &= \sum_{h=0}^D \p^h_{ij} \A_h   \qquad\qquad
                                             (0 \leq i,j \leq D).    \label{eq:AiAj2}
\end{align}
By the equation on the left in \eqref{eq:Ai2} the matrices
$\{\A_i\}_{i=0}^D$ are linearly independent.
Let $\M$ denote the subspace of $\Mat$ with basis $\{\A_i\}_{i=0}^D$.
We have $\p^h_{i j} = \p^h_{j i}$ for $0 \leq h,i,j \leq D$.
By this and \eqref{eq:AiAj2}, $\M$ is a commutative subalgebra of $\Mat$.
By \eqref{eq:Ai2},
$\M$ is closed under $\circ$ and contains $\J$.
We call $\M$ the {\em Bose-Mesner algebra of $\Gamma$}.

By \cite[Theorem 2.6.1(ii)]{BCN}
there exists a basis $\{\E_i\}_{i=0}^D$ for the $\C$-vector space $\M$ such that
\begin{align}
\E_0 &= |X|^{-1} \J, &
\E_i \E_j &= \delta_{i,j} \E_i \quad (0 \leq i, j \leq D), &
\sum_{i=0}^D \E_i &= \I.                                                    \label{eq:Ei}
\end{align}
We have
\begin{equation}
    \V = \sum_{i=0}^D \E_i \V   \qquad\qquad \text{(direct sum)}.   \label{eq:VsumEi}
\end{equation}
In the above sum, the summands are the common eigenspaces  for $\M$.
For $0 \leq i \leq D$ the matrix
$\E_i$ is the projection onto the eigenspace $\E_i \V$.
We call $\{\E_i\}_{i=0}^D$ the {\em primitive idempotents} of $\Gamma$.
We call $\E_0$ the {\em trivial primitive idempotent}.
The primitive idempotents are unique up to ordering of the nontrivial
primitive idempotents.

Since $\M$ is closed under $\circ$,
there exist  $\q^h_{i j} \in \C$ $(0 \leq h,i,j \leq D)$ such that
\begin{align}
  \E_i \circ \E_j &= |X|^{-1} \sum_{h=0}^D \q^h_{i j} \E_h  &&  (0 \leq i,j \leq D).   \label{eq:EicircEj}
\end{align}
The scalars $\q^h_{i j}$ are called the {\em Krein parameters} of $\Gamma$ \cite{BI}.
By \cite[Theorem 3.8]{BI} each Krein parameter is real and nonnegative.
Define
\begin{align}
   \m_i &= q^0_{i i}   &&  (0 \leq i \leq D).          \label{eq:defmi}
\end{align}
We have $\m_0 = 1$.
By \cite[Proposition 3.7]{BI},  $\m_i = \text{\rm rank} (\E_i)$ $(0 \leq i \leq D)$.
So by \eqref{eq:VsumEi},
\begin{equation*}
  \sum_{i=0}^D \m_i = |X|.
\end{equation*}

Since each of $\{\A_i\}_{i=0}^D$ and $\{\E_i\}_{i=0}^D$ is a basis for $\M$,
there exist matrices  $P$, $Q \in \text{\rm Mat}_{D+1}(\C)$ such that
\begin{align}
  \A_j &= \sum_{i=0}^D P_{i j} \E_i, &
  \E_j &= |X|^{-1} \sum_{i=0}^D Q_{i j} \A_i  &&   (0 \leq j \leq D).   \label{eq:defPQ}
\end{align}
Using \eqref{eq:defPQ} we find $P Q = |X| I$.
Using $\A_0 = \I$ and $\E_0 =|X|^{-1} \J$ we obtain
\begin{align*}
 P_{i 0} &= 1,  &  Q_{i 0} &= 1  &&  (0 \leq i \leq D).          
\end{align*}
By \cite[Lemma 2.2.1(ii)]{BCN},
\begin{align}
   \k_j &= P_{0j},   &  \m_j &= Q_{0j}   &&   (0 \leq j \leq D).           \label{eq:kjP0j}
\end{align}
Define
\begin{align}
 \th_i &= P_{i 1},  &  \th^*_i &= Q_{i 1}  && (0 \leq i \leq D).               \label{eq:defthithsi}
\end{align}
By the equation on the left in \eqref{eq:defPQ} with $j=1$,
\begin{align*}
  \A_1 &= \sum_{i=0}^D \th_i \E_i.                                 
\end{align*}
So $\th_i$ is the eigenvalue of $\A_1$ corresponding to $\E_i$  $(0 \leq i \leq D)$.
We call $\{\th_i\}_{i=0}^D$ (resp.\ $\{\th^*_i\}_{i=0}^D$) 
the {\em eigenvalue sequence} (resp. {\em dual eigenvalue sequence}) 
{\em of $\Gamma$ with respect to the ordering $\{\E_i\}_{i=0}^D$}.

\begin{definition}    \label{def:vi2}   \samepage
\ifDRAFT {\rm def:vi2}. \fi
Define polynomials $\{v_i\}_{i=0}^D$ in $\C[\lambda]$ such that
$v_0=1$ and
\begin{align*}
  \lambda v_i &= \c_{i+1} v_{i+1} + \a_i v_i + \b_{i-1} v_{i-1}   &&  (1 \leq i \leq D-1).
\end{align*}
\end{definition}

\begin{lemma}  {\rm  \cite[Lemma 3.8]{T:subconst1} }
\label{lem:vi}   \samepage
\ifDRAFT {\rm lem:vi}. \fi
The following {\rm (i)--(iv)} hold.
\begin{itemize}
\item[\rm (i)]
For $0 \leq i \leq D$ the polynomial $v_i$ has degree $i$ 
and leading coefficient $(c_1 c_2 \cdots c_i)^{-1}$. 
\item[\rm (ii)]
For $0 \leq i \leq D$,  $\A_i = v_i (\A_1)$.  
\item[\rm (iii)]
The algebra $\M$ is generated by $\A_1$.
\item[\rm (iv)]
The scalars $\{\th_i\}_{i=0}^D$ are mutually distinct.
\end{itemize}
\end{lemma}

\begin{lemma}    \label{lem:Pij}    \samepage
\ifDRAFT {\rm lem:Pij}. \fi
For the polynomials $\{v_i\}_{i=0}^D$ from Definition \ref{def:vi2},
\begin{align*}
   P_{i j} &= v_j (\th_i)   &&  (0 \leq i,j \leq D).          
\end{align*}
\end{lemma}

\begin{proof}
In the equation on the left in \eqref{eq:defPQ}, multiply each side on the right $\E_i$ to get
$\A_j \E_i = P_{i j} \E_i$.
By Lemma \ref{lem:vi}(ii),
$\A_j \E_i =  v_j (\A_1) \E_i$.
Using   $\A_1 \E_i = \th_i \E_i$ one finds that
$v_j (\A_1) \E_i = v_j (\th_i) \E_i$.
By these comments we obtain $P_{i j} = v_j (\th_i)$.
\end{proof}

We recall the dual Bose-Mesner algebra.
For the rest of this section, fix $x \in X$.
For $0 \leq i \leq D$ let $\E^*_i = \E^*_i (x)$ denote the diagonal matrix in $\Mat$
that has $(y,y)$-entry $1$ if $\partial(x,y)=i$ and $0$ if $\partial (x,y) \neq i$ $(y \in X)$.
By construction,
\begin{align*}
 \E^*_i \E^*_j &= \delta_{i,j} \E^*_i \qquad (0 \leq i,j \leq D),
&
 \sum_{i=0}^D \E^*_i &= \I.
\end{align*}
Consequently $\{\E^*_i\}_{i=0}^D$ form a basis for a commutative subalgebra $\M^* = \M^*(x)$ of $\Mat$.
We call $\M^*$ the {\em dual Bose-Mesner algebra of $\Gamma$ with respect
to $x$}.
We call $\{\E^*_i\}_{i=0}^D$  the {\em dual primitive idempotents of $\Gamma$
with respect to $x$}.
For $0 \leq i \leq D$ we have
\[
  \E^*_i \V = \text{Span} \{ \widehat{y} \,|\, y \in \Gamma_i (x) \}
\]
and $\k_i = \text{rank} (\E^*_i) $.
Moreover
\begin{equation*}
   \V = \sum_{i=0}^D \E^*_i \V \qquad\qquad \text{(direct sum)}.      
\end{equation*}
In the above sum, the summands are the common eigenspaces for $\M^*$.
For $0 \leq i \leq D$ the matrix $\E^*_i$ is the projection onto the eigenspace $\E^*_i \V$.
For $0 \leq i \leq D$ let $\A^*_i = \A^*_i (x)$ denote the diagonal matrix in $\Mat$
whose $(y,y)$-entry is the $(x,y)$-entry of $|X| \E_i$ $(y \in X)$.
By \eqref{eq:Ei} and \eqref{eq:EicircEj},
\begin{align}
  \A^*_i \A^*_j = \sum_{h=0}^D \q^h_{i j} \A^*_h   \qquad\qquad  (0 \leq i,j \leq D),     \label{eq:AsiAsj}
\\
 \A^*_0 = \I, \qquad\qquad\qquad
 \sum_{i=0}^D \A^*_i = |X| \E^*_0.             \notag        
\end{align}
By \eqref{eq:defPQ} we find that for $0 \leq j \leq D$,
\begin{align}
  \E^*_j &= |X|^{-1} \sum_{i=0}^D P_{i j} \A^*_i,
&
  \A^*_j &= \sum_{i=0}^D Q_{i j} \E^*_i.         \label{eq:Esi}
\end{align}
By \eqref{eq:Esi} the matrices $\{\A^*_i\}_{i=0}^D$ form a basis for the
$\C$-vector space $\M^*$.
We call $\{\A^*_i\}_{i=0}^D$ the {\em dual distance-matrices of $\Gamma$
with respect to $x$}.
We call $\A^*_1$ the {\em dual adjacency matrix} of $\Gamma$ with respect to $x$.
By the equation on the right in \eqref{eq:Esi} with $j=1$,
\[
  \A^*_1 = \sum_{i=0}^D \th^*_i \E^*_i.
\]
So $\th^*_i$ is the eigenvalue of $\A^*_1$ corresponding to $\E^*_i$ $(0 \leq i \leq D)$.

The algebras $\M$ and $\M^*$ are related as follows.
By \cite[Lemma 3.2]{T:subconst1},
\begin{align*}
  \E^*_i \A_h \E^*_j = 0  \qquad \text{if and only if} \qquad \p^h_{i j} =0 \qquad (0 \leq h,i,j \leq D),
\\
  \E_i \A^*_h \E_j = 0  \qquad \text{if and only if} \qquad \q^h_{i j} =0 \qquad (0 \leq h,i,j \leq D).
\end{align*}

\begin{lemma}   \label{lem:AiEs0E02}    \samepage
\ifDRAFT {\rm lem:AiEs0E02}. \fi
For $0 \leq i \leq D$,
\begin{align}
  \A_i \E^*_0 \E_0 &= \E^*_i \E_0,
&
 \E_i \E^*_0 \E_0 &= |X|^{-1} \A^*_i \E_0,          \label{eq:AiEs0E02}
\\
  \A^*_i \E_0 \E^*_0 &= \E_i \E^*_0,
&
 \E^*_i \E_0 \E^*_0 &= |X|^{-1} \A_i \E^*_0.        \label{eq:AsiE0Es02}
\end{align}
\end{lemma}

\begin{proof}
For $y,z \in X$, compare the $(y,z)$-entry of each side using the definition of $\A^*_i$ and $\E^*_i$.
\end{proof}

\begin{lemma}   \label{lem:E0Es0E0}    \samepage
\ifDRAFT {\rm lem:E0Es0E0}. \fi
We have
\begin{align*}
  |X| \, \E_0 \E^*_0 \E_0 &= \E_0,   &
  |X| \, \E^*_0 \E_0 \E^*_0 &= \E^*_0.
\end{align*}
\end{lemma}

\begin{proof}
For the equation on the right in \eqref{eq:AiEs0E02}, \eqref{eq:AsiE0Es02} set $i=0$
and use $\A_0 = \I$, $\A^*_0 = \I$.
\end{proof}

Let $\T= \T (x)$ denote the subalgebra of $\Mat$ generated by $\M$ and $\M^*$.
We call $\T$ the {\em Terwilliger algebra of $\Gamma$ with respect to $x$} \cite{T:subconst1}.

By a {\em $\T$-module} we mean a subspace of $\V$ that is invariant
under the action of $\T$.
The $\T$-module $\V$ is said to be  {\em standard}.
A $\T$-module $U$ is said to be {\em irreducible} whenever $U \neq 0$ and
there is no $\T$-submodule of $U$ other than $0$, $U$.

\begin{lemma}  {\rm   \cite[Lemma 3.4]{T:subconst1}  }
\label{lem:decomp}   \samepage
\ifDRAFT {\rm lem:decomp}. \fi
The following hold.
\begin{itemize}
\item[\rm (i)]
The standard $\T$-module $\V$ is a direct sum of  irreducible $\T$-modules.
\item[\rm (ii)]
Each irreducible $\T$-module $U$ is a direct sum of the nonzero subspaces
among $\E_0 U, \E_1 U$, $\ldots, \E_D U$,
and a direct sum of the nonzero subspaces among
$\E^*_0 U, \E^*_1 U, \ldots, \E^*_D U$.
\end{itemize}
\end{lemma}

Let $U$ denote an irreducible $\T$-module.
We say that $U$ is {\em thin} (resp.\ {\em dual thin}) whenever
$\dim \E^*_i U \leq 1$ (resp. $\dim \E_i U \leq 1$) for $0 \leq i \leq D$.
Define the {\em diameter} $d=d(U)$ and {\em dual diameter} $d^* = d^* (U)$
by
\begin{align*}
  d &= |\{ i \,|\, 0 \leq i \leq D, \; \E^*_i U \neq 0\}| - 1,
&
  d^* &= |\{ i \,|\, 0 \leq i \leq D, \; \E_i U \neq 0\}| - 1.
\end{align*}
If $U$ is thin (resp.\ dual thin) then the dimension of $U$ is $d+1$ (resp.\ $d^* + 1$).
If $U$ is both thin and dual thin,
then $d=d^*$.

We recall the primary $\T$-module \cite[Lemma 3.6]{T:subconst1}.
Define 
\[
\text{\bf 1} = \sum_{y \in X} \widehat{y}.
\]
Observe that the entries of $\text{\bf 1}$ are all $1$.
 
\begin{lemma}  {\rm  \cite[Lemma 3.6]{T:subconst1} }
\label{lem:primary}    \samepage
\ifDRAFT {\rm lem:primary}. \fi
For $0 \leq i \leq D$,
\begin{align*}
  \A_i \widehat{x} &= \E^*_i \text{\bf 1},  &
  \A^*_i \text{\bf 1} &= |X| \E_i \widehat{x}.
\end{align*}
Moreover,
$\M \widehat{x} = \M^* \text{\bf 1}$
is a thin, dual-thin, irreducible $\T$-module of  diameter $D$.
\end{lemma}

\begin{definition}   \label{def:primary}    \samepage
\ifDRAFT {\rm def:primary}. \fi
The $\T$-module in Lemma \ref{lem:primary} is called {\em primary}.
Let $\pU$ denote the primary $\T$-module.
\end{definition}

\begin{lemma}     \label{lem:faithful0}    \samepage
\ifDRAFT {\rm lem:faithful0}. \fi
The following hold:
\begin{itemize}
\item[\rm (i)]
the map $\M \to \pU$, $\B \mapsto \B \widehat{x}$ is a $\C$-linear bijection;
\item[\rm (ii)]
the map $\M^* \to \pU$, $\B \mapsto \B \text{\bf 1}$ is a $\C$-linear bijection.
\end{itemize}
\end{lemma}

\begin{proof}
(i)
The map is $\C$-linear by construction,
and surjective by Lemma \ref{lem:primary}.
It is a bijection since $\M$, $\pU$ have the same dimension.

(ii) 
The map is $\C$-linear by construction,
and surjective by Lemma \ref{lem:primary}.
It is a bijection since $\M^*$, $\pU$ have the same dimension.
\end{proof}

\begin{corollary}     \label{cor:faithful}    \samepage
\ifDRAFT {\rm cor:faithful}. \fi
The following hold.
\begin{itemize}
\item[\rm (i)]
Let $\B \in \M$ such that $\B \pU = 0$.
Then $\B=0$.
\item[\rm (ii)]
Let $\B \in \M^*$ such that $\B \pU=0$.
Then $\B=0$.
\end{itemize}
\end{corollary}

\begin{proof}
(i)
We have $\B \widehat{x}=0$ since $\widehat{x} \in \pU$,
so $\B=0$ in view of Lemma \ref{lem:faithful0}(i).

(ii)
We have $\B \text{\bf 1} = 0$ since $\text{\bf 1} \in \pU$,
so $\B=0$ in view of Lemma \ref{lem:faithful0}(ii).
\end{proof}

Next we discuss general irreducible $\T$-modules.
For an irreducible $\T$-module $U$,
define the {\em endpoint} $r=r(U)$ by
\begin{align*}
  r &= \text{\rm min} \{i \,|\, 0 \leq i \leq D, \; \E^*_i U \neq 0 \}.
\end{align*}

\begin{lemma}  {\rm \cite[Lemma 3.9]{T:subconst1} }
\label{lem:Tmodule1}    \samepage
\ifDRAFT {\rm lem:Tmodule1}. \fi
Let $U$ denote an irreducible $\T$-module 
with endpoint $r$ and diameter $d$.
Then the following hold.
\begin{itemize}
\item[\rm (i)]
We have $\E^*_i U \neq 0$ if and only if $r \leq i \leq r+d$ $(0 \leq i \leq D)$.
\item[\rm (ii)]
If $U$ is thin then $U$ is dual thin.
\end{itemize}
\end{lemma}

Next we recall the $Q$-polynomial property.
The ordering $\{E_i\}_{i=0}^D$ is said to be {\em $Q$-polynomial} whenever
$\q^h_{ij}$ is zero (resp.\ nonzero) if one of $h,i,j$ is greater than (resp.\ equal to)
the sum of the other two $(0 \leq h,i,j \leq D)$.
The graph $\Gamma$ is said to be {\em $Q$-polynomial with respect to 
 $\{\E_i\}_{i=0}^D$} whenever the ordering $\{\E_i\}_{i=0}^D$ is  $Q$-polynomial.
For the rest of this section,
assume that $\Gamma$ is $Q$-polynomial with respect to $\{\E_i\}_{i=0}^D$.
Define
\begin{align}
 \c^*_i &= \q^i_{1, i-1}  && (1 \leq i \leq D),      \label{eq:defcsi}
\\
 \a^*_i &= \q^i_{1 i} && (0 \leq i \leq D),           \label{eq:defasi}
\\
 \b^*_i &= \q^i_{1, i+1}  && (0 \leq i \leq D-1).                \label{eq:defbsi}
\end{align}
For notational convenience define $\c^*_0 = 0$ and $\b^*_D = 0$.
We have $\c^*_i \b^*_{i-1} \neq 0$ for $1 \leq i \leq D$.
Abbreviate $\m = \m_1$.
By \cite[Proposition 3.7]{BI},
\begin{align}
\a^*_0 &= 0, &
\b^*_0 &= \m, &
\c^*_1 &= 1, &
\c^*_i + \a^*_i + \b^*_i &= \m  \quad (0 \leq i \leq D).
\end{align}
By \cite[Lemma 2.3.1(iv)]{BCN},
\begin{align*}
 \m_i &= 
 \frac{\b^*_0 \b^*_1 \cdots \b^*_{i-1} }
        {\c^*_1 \c^*_2 \cdots \c^*_i}    &&  (0 \leq i \leq D).             
\end{align*}

\begin{definition}    \label{def:vsi}   \samepage
\ifDRAFT {\rm def:vsi}. \fi
Define polynomials $\{v^*_i\}_{i=0}^D$ in $\C[\lambda]$ such that
$v^*_0=1$ and
\begin{align*}
  \lambda v^*_i &= \c^*_{i+1} v^*_{i+1} + \a^*_i v^*_i + \b^*_{i-1} v^*_{i-1}   &&  (1 \leq i \leq D-1).
\end{align*}
\end{definition}

\begin{lemma}  {\rm  \cite[Lemma 3.11]{T:subconst1} }
\label{lem:vsi}   \samepage
\ifDRAFT {\rm lem:vsi}. \fi
The following {\rm (i)--(iv)} hold.
\begin{itemize}
\item[\rm (i)]
For $0 \leq i \leq D$ the polynomial $v^*_i$ has degree $i$ and
leading coefficient $(\c^*_1 \c^*_2 \cdots \c^*_i)^{-1}$.
\item[\rm (ii)]
For $0 \leq i \leq D$,  $\A^*_i = v^*_i (\A^*_1)$.  
\item[\rm (iii)]
The algebra $\M^*$ is generated by $\A^*_1$.
\item[\rm (iv)]
The scalars $\{\th^*_i\}_{i=0}^D$ are mutually distinct.
\end{itemize}
\end{lemma}

\begin{lemma}    \label{lem:Qij}    \samepage
\ifDRAFT {\rm lem:Qij}. \fi
For the polynomials $\{v^*_i\}_{i=0}^D$ from Definition \ref{def:vsi},
\begin{align*}
   Q_{i j} &= v^*_j (\th^*_i)   &&  (0 \leq i,j \leq D).          
\end{align*}
\end{lemma}

\begin{proof}
Similar to the proof of Lemma \ref{lem:Pij}.
\end{proof}

Let $U$ denote an irreducible $\T$-module.
Define the {\em dual endpoint} $s=s(U)$ by
\begin{align*}
  s &= \text{\rm min} \{i \,|\, 0 \leq i \leq D, \; \E_i U \neq 0 \}.
\end{align*}

\begin{lemma}  {\rm \cite[Lemma 3.9]{T:subconst1} }
\label{lem:Tmodule2}    \samepage
\ifDRAFT {\rm lem:Tmodule2}. \fi
Let $U$ denote an irreducible $T$-module
with dual endpoint $s$ and dual diameter $d^*$.
Then the following hold.
\begin{itemize}
\item[\rm (i)]
We have $\E_i U \neq 0$ if and only if $s \leq i \leq s+d^*$ $(0 \leq i \leq D)$.
\item[\rm (ii)]
If $U$ is dual thin then $U$ is thin.
\end{itemize}
\end{lemma}

\begin{lemma}   {\rm  \cite[Lemmas 3.9, 3.12]{T:subconst1} }
\label{lem:U}    \samepage
\ifDRAFT {\rm lem:U}. \fi
Let $U$ denote a thin irreducible $\T$-module with endpoint $r$,
dual endpoint $s$, and diameter $d$.
Then the sequence
\begin{equation}
    (\A_1; \{\E_{s+i}\}_{i=0}^d; \A^*_1; \{\E^*_{r+i}\}_{i=0}^d)        \label{eq:UPhi}
\end{equation}
acts on $U$ as a Leonard system,
with eigenvalue sequence $\{\th_{s+i} \}_{i=0}^d$ and dual eigenvalue sequence
$\{ \th^*_{r+i} \}_{i=0}^d$.
\end{lemma}

Consider the Leonard system \eqref{eq:UPhi} on $U$.
Any object $\omega$ associated with this Leonard system will be denoted
by $\omega(U)$.

Recall the primary $\T$-module $\pU$.
By Lemma \ref{lem:U} the sequence
\begin{equation}
   (\A_1; \{\E_i\}_{i=0}^D; \A^*_1; \{\E^*_i\}_{i=0}^D)    \label{eq:U0Phi}
\end{equation}
acts on $\pU$ as a Leonard system,
with eigenvalue sequence $\{\th_i\}_{i=0}^D$ and dual eigenvalue sequence
$\{\th^*_i\}_{i=0}^D$.

\begin{lemma}    \label{lem:nu3}    \samepage
\ifDRAFT {\rm lem:nu3}. \fi
We have
\begin{equation*}
   \nu (\pU) = |X|.            
\end{equation*}
\end{lemma}

\begin{proof}
By Lemma \ref{lem:nuE0Es0E0},
$\nu(\pU) \E_0 \E^*_0 \E_0 = \E_0$ holds on $\pU$.
Now by Lemma \ref{lem:E0Es0E0} and Corollary \ref{cor:faithful} we obtain the result.
\end{proof}

\begin{lemma}    \label{lem:AiEs0E03}    \samepage
\ifDRAFT {\rm lem:AiEs0E03}. \fi
For the Leonard system in \eqref{eq:U0Phi} on $\pU$,
consider the elements $\{A_i\}_{i=0}^D$ and
$\{A^*_i\}_{i=0}^D$ from Definition \ref{def:Ai}.
Then for $0 \leq i \leq D$, 
$\A_i$ acts on $\pU$ as  $A_i$,
and $\A^*_i$ acts on $\pU$ as $A^*_i$.
\end{lemma}

\begin{proof}
By Lemma \ref{lem:rhorhos2} the following holds on $\pU$:
\begin{equation*}
   \E^*_i \E_0 \E^*_0 = \nu(\pU)^{-1} A_i \E^*_0.                 
\end{equation*}
Comparing this with the equation on the right in \eqref{eq:AsiE0Es02}, 
and using Lemma \ref{lem:nu3}, we find that
$A_i = \A_i$ on $\E^*_0 \pU$.
The subspace $\E^*_0 \pU$ is spanned by $\widehat{x}$,
so $A_i =\A_i$ on $\pU$ in view of Lemma \ref{lem:faithful0}(i).
The result for $\A^*_i$ is similarly obtained.
\end{proof}

\begin{lemma}    \label{lem:phijqhij}    \samepage
\ifDRAFT {\rm lem:phijqhij}. \fi
For $0 \leq h, i, j \leq D$,
\begin{align*}
 p^h_{i j} (\pU) &= \p^h_{i j}, &
 q^h_{i j} (\pU) &= \q^h_{i j}.
\end{align*}
\end{lemma}

\begin{proof}
Compare the equations \eqref{eq:AiAj}, \eqref{eq:AsiAsj2} with
 \eqref{eq:AiAj2}, \eqref{eq:AsiAsj} and use Lemma \ref{lem:AiEs0E03}.
\end{proof}

\begin{lemma}    \label{lem:ciaibi}    \samepage
\ifDRAFT {\rm lem:ciaibi}. \fi
We have 
\begin{align}
c_i (\pU) &=\c_i  & c^*_i (\pU) &=\c^*_i    && (1 \leq i \leq D),   \label{eq:cipU}  \\
a_i (\pU) &= \a_i & a^*_i (\pU) &= \a^*_i &&  (0 \leq i \leq D),    \label{eq:aipU}  \\
b_i (\pU) &= \b_i & b^*_i (\pU) &= \b^*_i && (0 \leq i \leq D-1).   \label{eq:bipU}
\end{align}
\end{lemma}

\begin{proof}
We first obtain the equations on the left in \eqref{eq:cipU}--\eqref{eq:bipU}.
Apply Lemma \ref{lem:AA1} to the Leonard system \eqref{eq:U0Phi} on $\pU$,
and use Lemma \ref{lem:AiEs0E03} to find that
$\A_1 = c_1(\pU) \A_1 + a_0 (\pU) \I$ holds on $\pU$.
By this and Corollary \ref{cor:faithful}(i) we obtain $c_1 (\pU) = 1$ and $a_0 (\pU)=0$.
Now compare the equations in Lemma \ref{lem:phij} with \eqref{eq:defci}--\eqref{eq:defbi},
and use Lemma \ref{lem:phijqhij}.
This gives the equations on the left in \eqref{eq:cipU}--\eqref{eq:bipU}.
Next we obtain the equations on the right in \eqref{eq:cipU}--\eqref{eq:bipU}.
Apply Lemma \ref{lem:AA1} to the dual of the Leonard system \eqref{eq:U0Phi} on $\pU$,
and use Lemma \ref{lem:AiEs0E03} to find that
$\A^*_1 = c^*_1 (\pU) \A^*_1 + a^*_0 (\pU) \I$ holds on $\pU$.
By this and Corollary \ref{cor:faithful}(ii) we obtain
 $c^*_1 (\pU) = 1$ and $a^*_0 (\pU)=0$.
Now apply Lemma \ref{lem:phij} to the dual of the Leonard system \eqref{eq:U0Phi} on $\pU$ to get
\begin{align*}
c^*_i (\pU) &= q^i_{1,i-1} (\pU)  && (1 \leq i \leq D), \\
a^*_i (\pU) &= q^i_{1,i} (\pU)   && (0 \leq i \leq D), \\
b^*_i (\pU) &= q^i_{1,i+1} (\pU) && (0 \leq i \leq D-1).
\end{align*}
Now compare these equations with \eqref{eq:defcsi}--\eqref{eq:defbsi},
and use Lemma \ref{lem:phijqhij}.
This gives the equations on the right in \eqref{eq:cipU}--\eqref{eq:bipU}.
\end{proof}

\begin{lemma}    \label{lem:nuki}    \samepage
\ifDRAFT {\rm lem:nuki}. \fi
For $0 \leq i \leq D$,
\begin{align}
k_i (\pU) &= \k_i,    &   k^*_i (\pU) &= \m_i.    \label{eq:kipU}
\end{align}
\end{lemma}

\begin{proof}
To get the equation on the left in \eqref{eq:kipU},
apply Lemma \ref{lem:kip0ii} to the Leonard system \eqref{eq:U0Phi} on $\pU$,
and compare the result with \eqref{eq:defki2} using Lemma \ref{lem:phijqhij}.
To get the equation on the right in \eqref{eq:kipU},
apply Lemma \ref{lem:kip0ii} to the dual of the Leonard system \eqref{eq:U0Phi} on $\pU$,
and compare the result with \eqref{eq:defmi} using Lemma \ref{lem:phijqhij}.
\end{proof}

\begin{definition}    \label{def:qRacahDRG}    \samepage
\ifDRAFT {\rm def:qRacahDRG}. \fi
Let $0 \neq q \in \C$.
The $Q$-polynomial ordering $\{\E_i\}_{i=0}^D$ is said to have {\em $q$-Racah type}
whenever there exists an affine transformation of the Leonard system
\eqref{eq:U0Phi} on $\pU$, that has $q$-Racah type.
\end{definition}

\section{Formally self-dual distance-regular graphs}
\label{sec:DRGsd}

In this section we discuss a type of distance-regular graph, said to be formally self-dual.
Let $\Gamma$ denote a distance-regular graph with diameter $D \geq 1$.
Recall the distance-matrices $\{\A_i\}_{i=0}^D$ and 
primitive idempotents $\{\E_i\}_{i=0}^D$ in the Bose-Mesner algebra $\M$
of $\Gamma$.
Recall the matrices $P$, $Q \in \text{\rm Mat}_{D+1}(\C)$ from \eqref{eq:defPQ}.
Fix $x \in X$ and abbreviate $\T=\T(x)$, $\M^* = \M^*(x)$,
and
$\E^*_i = \E^*_i (x)$, $\A^*_i = \A^*_i (x)$ for $0 \leq i \leq D$.

\begin{definition}    \label{def:FSD}    \samepage
\ifDRAFT {\rm def:FSD}. \fi
The graph $\Gamma$ is said to be {\em formally self-dual with respect to $\{\E_i\}_{i=0}^D$}
whenever $P=Q$.
\end{definition}

The following fact is well-known; see \cite[Section 2.3]{BCN}.
We give a short proof for completeness.

\begin{lemma}    \label{lem:duality}    \samepage
\ifDRAFT {\rm lem:duality}. \fi
Assume that $\Gamma$ is formally self-dual with respect to
$\{\E_i\}_{i=0}^D$.
Then $\p^h_{i j} = \q^h_{i j}$ for $0 \leq h,i,j \leq D$.
\end{lemma}

\begin{proof}
Consider the $\C$-linear map $\sigma : \M \to \M^*$ that sends 
$\E_i \mapsto \E^*_i$ for $0 \leq i \leq D$.
Clearly $\sigma$ is bijective.
We have 
$\E_i \E_j = \delta_{i,j} \E_i$ and $\E^*_i \E^*_j = \delta_{i,j} \E^*_i$ for $0 \leq i,j \leq D$.
Thus  $\sigma$ is an isomorphism of $\C$-algebras.
By the equation on the left in \eqref{eq:defPQ} and the
equation on the right in \eqref{eq:Esi} with $P=Q$,
we find that $\sigma$ sends $\A_j \mapsto \A^*_j$ for $0 \leq j \leq D$.
Now we get the result by \eqref{eq:AiAj2} and  \eqref{eq:AsiAsj}. 
\end{proof}

By Lemma \ref{lem:duality}, if $\Gamma$ is formally self-dual with respect to $\{\E_i\}_{i=0}^D$,
then $\Gamma$ is $Q$-polynomial with respect to $\{\E_i\}_{i=0}^D$.

\begin{lemma}    \label{lem:Asivi}   \samepage
\ifDRAFT {\rm lem:Asivi}. \fi
Assume that $\Gamma$ is formally self-dual with respect to $\{\E_i\}_{i=0}^D$.
Let the polynomials $\{v_i\}_{i=0}^D$ (resp.\ $\{v^*_i\}_{i=0}^D$) be from Definition \ref{def:vi2}
(resp.\ Definition \ref{def:vsi}).
Then $v_i = v^*_i$ for $0 \leq i \leq D$.
\end{lemma}

\begin{proof}
Recall the intersection numbers $\{ \c_i \}_{i=1}^D$, $\{ \a_i\}_{i=0}^D$, $\{\b_i \}_{i=0}^{D-1}$
and the Krein parameters $\{ \c^*_i \}_{i=1}^D$, $\{ \a^*_i\}_{i=0}^D$, $\{\b^*_i \}_{i=0}^{D-1}$
of $\Gamma$.
By Lemma \ref{lem:duality},
$\p^h_{i j} = \q^h_{i j}$ for $0 \leq h,i,j \leq D$.
By \eqref{eq:defci}--\eqref{eq:defbi} and \eqref{eq:defcsi}--\eqref{eq:defbsi}
we see that
$\c_i = \c^*_i$ $(1 \leq i \leq D)$,
$\a_i = \a^*_i$ $(0 \leq i \leq D)$,
$\b_i = \b^*_i$ $(0 \leq i \leq D-1)$.
Now $v_i = v^*_i$ $(0 \leq i \leq D)$ by Definitions \ref{def:vi2}, \ref{def:vsi}.
\end{proof}

The formally self-dual condition is characterized as follows.
Recall the eigenvalue sequence $\{\th_i\}_{i=0}^D$ and  dual eigenvalue sequence
$\{\th^*_i\}_{i=0}^D$ of $\Gamma$ with respect to $\{\E_i\}_{i=0}^D$.

\begin{proposition}    \label{prop:DRGsd}    \samepage
\ifDRAFT {\rm prop:DRGsd}. \fi
The following {\rm (i)--(iii)} are equivalent:
\begin{itemize}
\item[\rm (i)]
$\Gamma$ is formally self-dual with respect to $\{\E_i\}_{i=0}^D$;
\item[\rm (ii)]
$\Gamma$ is $Q$-polynomial with respect to $\{\E_i\}_{i=0}^D$
and $\th_i = \th^*_i$ for $0 \leq i \leq D$;
\item[\rm (iii)]
the sequence \eqref{eq:U0Phi} acts on the primary $\T$-module $\pU$ as a self-dual
Leonard system.
\end{itemize}
\end{proposition}

\begin{proof}
(i) $\Rightarrow$ (ii)
As we saw above Lemma \ref{lem:Asivi}, $\Gamma$ is $Q$-polynomial with respect to
$\{\E_i\}_{i=0}^D$.
By Definition \ref{def:FSD} we have $P=Q$.
By this and \eqref{eq:defthithsi} we obtain $\th_i =\th^*_i$ $(0 \leq i \leq D)$.

(ii) $\Rightarrow$ (iii)
Let $\Phi$ denote the Leonard system in \eqref{eq:U0Phi} on $\pU$.
Then $\Phi$ has eigenvalue sequence $\{\th_i\}_{i=0}^D$
and dual eigenvalue sequence $\{\th^*_i\}_{i=0}^D$.
We assume $\th_i = \th^*_i$ for $0 \leq i \leq D$,
so  $\Phi$ is self-dual by Lemma \ref{lem:selfdualparam}.

(iii) $\Rightarrow$ (i)
We show that $P=Q$.
Let $\sigma$ denote the duality of the Leonard system in \eqref{eq:U0Phi}.
Then, on $\pU$,  $\sigma$ swaps $\E_i$ and $\E^*_i$ for $0 \leq i \leq D$.
By Lemmas \ref{lem:selfdualAi} and \ref{lem:AiEs0E03},
on $\pU$, $\sigma$ swaps $\A_i$ and $\A^*_i$ for $0 \leq i \leq D$.
Now comparing \eqref{eq:defPQ} and \eqref{eq:Esi}, we obtain $P=Q$.
\end{proof}

\section{From a spin model to spin Leonard pairs}
\label{sec:spindrg}

Let $\Gamma$ denote a distance-regular graph with vertex set $X$ and
diameter $D \geq 1$.
In this section we recall what it means for a spin model $\W$ to be afforded by $\Gamma$.
For such $\W$ and all $x \in X$,
we construct a spin Leonard pair on each irreducible $\T(x)$-module.

Recall the Bose-Mesner algebra $\M$ of $\Gamma$.
Let $\W$ denote a spin model in $\Mat$.
Recall the algebra $N(\W)$ from Definition \ref{def:NW} and Lemma \ref{lem:Nomuraalg}.

\begin{definition}    \label{def:afford}   \samepage
\ifDRAFT {\rm def:afford}. \fi
We say that {\em $\Gamma$ affords $\W$} whenever
$\W \in \M \subseteq N(\W)$.
\end{definition}

\begin{lemma}    \label{lem:afford}    \samepage
\ifDRAFT {\rm lem:afford}. \fi
Assume that $\W$ is afforded by $\Gamma$.
Then each of the spin models  $-\W$, $\W^-$ is afforded by $\Gamma$.
\end{lemma}

\begin{proof}
By Definition \ref{def:afford}, $\W \in \M \subseteq N(\W)$.
We have $-\W \in \M$, and $N(-\W) = N(\W)$ by Lemma \ref{lem:NW-}.
Thus $-\W \in \M \subseteq N(-\W)$, and so $-\W$ is afforded by $\Gamma$.
We have $\W^{-1} \in \M$, since
for any invertible $B \in \Mat$ the element $B^{-1}$ is a polynomial in $B$.
By this and Lemma \ref{lem:type2W-}(ii), $\W^- \in \M$.
By Lemma \ref{lem:NW-}, $N(\W^-) = N(\W)$.
Thus $\W^- \in \M \subseteq N(\W^-)$,
and so $\W^-$ is afforded by $\Gamma$.
\end{proof}

For the rest of this section,
assume that $\W$ is afforded by $\Gamma$.

\begin{definition}   {\rm \cite[Section 3.1]{JMN} }
 \label{def:Psi}   \samepage
\ifDRAFT {\rm def:Psi}. \fi
We define a map $\Psi : N(\W) \to \Mat$ as follows.
For $\B \in N(\W)$ and $b,c \in X$ the $(b,c)$-entry of $\Psi(\B)$
is the eigenvalue of $\B$ for the eigenvector $\text{\bf u}_{b,c}$
from Definition \ref{def:NW}.
So
\begin{align*}
  \B \text{\bf u}_{b,c} &= \Psi(\B)(b,c) \text{\bf u}_{b,c}.
\end{align*}
Note that the map $\Psi$ is $\C$-linear.
\end{definition}

Recall the distance-matrices $\{\A_i\}_{i=0}^D$ of $\Gamma$.

\begin{lemma}  {\rm  \cite[Lemma 5.1]{CN:hyper} }
\label{lem:MPsi} \samepage
\ifDRAFT {\rm lem:MPsi}. \fi
There exists an ordering $\{\E_i\}_{i=1}^D$ of the nontrivial primitive
idempotents of $\Gamma$ such that
\begin{align*}
  \Psi (\A_i) &= |X| \E_i   &&  (0 \leq i \leq D).           
\end{align*}
Moreover $\Gamma$ is formally self-dual with respect to $\{\E_i\}_{i=0}^D$.
\end{lemma}

For the rest of this section, fix the ordering of the primitive idempotents
from Lemma \ref{lem:MPsi}.
Since $\{\E_i\}_{i=0}^D$ is a basis for $\M$ and $\W$ is an invertible
element in $\M$,
there exist nonzero scalars $f$, $\{\tau_i\}_{i=0}^D$ in $\C$ such that
$\tau_0 = 1$ and
\begin{equation}
   \W = f \sum_{i=0}^D \tau_i \E_i.             \label{eq:defWfti}
\end{equation}

\begin{lemma}    \label{lem:W-2}    \samepage
\ifDRAFT {\rm lem:W-2}. \fi
We have
\begin{align*}
  \W^{-1} &= f^{-1} \sum_{i=0}^D \tau_i^{-1} \E_i,     &
  \W^-  &= |X| f^{-1} \sum_{i=0}^D \tau_i^{-1} \E_i.
\end{align*}
\end{lemma}

\begin{proof}
The first equation follows from \eqref{eq:defWfti}.
The second equation follows from the first equation and Lemma \ref{lem:type2W-}(ii).
\end{proof}

For the rest of this section fix $x \in X$.
Abbreviate $\T=\T(x)$ and $\A^*_i = \A^*_i (x)$, $\E^*_i = \E^*_i(x)$
for $0 \leq i \leq D$.
Define
\begin{equation}    \label{def:Ws}
   \W^* = f \sum_{i=0}^D \tau_i \E^*_i.                
\end{equation}
Note that $\W^*$ is invertible.

\begin{lemma}    {\rm  \cite[Theorem 5.3, Corollary 5.4]{CW} }
\label{lem:balenced2} \samepage
\ifDRAFT {\rm lem:balanced2}. \fi
We have
\begin{align}
  \A_1 \W^* \W &= \W^* \W \A^*_1,          \label{eq:A1WsW=WsWAs1}
\\
   \W \W^* \W &= \W^* \W \W^*.              \label{eq:balanced2}
\end{align}
\end{lemma}

\begin{lemma}   {\rm  \cite[Lemma 2.7]{C:thin} }
\label{lem:Cur}    \samepage
\ifDRAFT {\rm lem:Cur}. \fi
We have
\begin{align}
\W &= |X|^{1/2} f^{-1} \sum_{i=0}^D \tau_i^{-1} \A_i,  &
\W^{-1} &= |X|^{-3/2} f \sum_{i=0}^D \tau_i \A_i,                              \label{eq:WWs}
\\
\W^* &= |X|^{1/2} f^{-1} \sum_{i=0}^D \tau_i^{-1} \A^*_i, &
(\W^*)^{-1} &= |X|^{-3/2} f \sum_{i=0}^D \tau_i \A^*_i.                  \label{eq:WWs2}
\end{align}
\end{lemma}

\begin{lemma}    \label{lem:WsW-0}    \samepage
\ifDRAFT {\rm lem:WsW-0}. \fi
For $y \in X$,
\begin{align*}
  \W^* (y,y) &= \frac{ |X|^{1/2} } {\W (x,y)} .                       
\end{align*}
\end{lemma}

\begin{proof}
Let $i = \partial (x,y)$.
By \eqref{def:Ws} we obtain $\W^*(y,y) = f \tau_i$.
By the equation on the left in \eqref{eq:WWs} we obtain
$\W(x,y ) = |X|^{1/2} f^{-1} \tau_i^{-1}$.
The result follows.
\end{proof}

\begin{lemma}    \label{lem:f}    \samepage
\ifDRAFT {\rm lem:f}. \fi
The scalar $f$ satisfies
\begin{equation}
    f^{-2} = |X|^{-3/2} \sum_{i=0}^D \k_i \tau_i.                 \label{eq:f-2}
\end{equation}
\end{lemma}

\begin{proof}
Recall the primary $\T$-module $\pU$.
By Lemma \ref{lem:U} the sequence
\[
 (\A_1; \{\E_i\}_{i=0}^D; \A^*_1; \{\E^*_i\}_{i=0}^D) 
\]
acts on $\pU$ as a Leonard system.
By Theorem \ref{thm:WtiinvAi} and Lemmas \ref{lem:piW}, \ref{lem:nu3},
\begin{align*}
   \W &= f \gamma \sum_{i=0}^D \tau_i^{-1} \A_i,  &
   \gamma &= |X|^{-1} \sum_{i=0}^D \k_i \tau_i.
\end{align*}
Comparing this with the equation on the left in \eqref{eq:WWs}, we get the result.
\end{proof}

We now describe the irreducible $\T$-modules.
The following result is essentially due to Curtin \cite[Theorem 8.3]{C:thin}.

\begin{proposition}     \label{prop:C2}    \samepage
\ifDRAFT {\rm prop:C2}. \fi
Every irreducible $\T$-module is thin, provided that
$\W$ does not have Hadamard type.
\end{proposition}

\begin{proof}
We refer to the equation on the left in \eqref{eq:WWs}.
We claim that
\begin{equation}
\text{ $\tau_{i-1}$, $\tau_i$, $\tau_{i+1}$ are mutually distinct for $1 \leq i \leq D-1$}.  \label{eq:aux}
\end{equation} 
By way of contradiction, assume that there exists an integer $j$ $(1 \leq j \leq D-1)$
such that $\tau_{j-1}$, $\tau_j$, $\tau_{j+1}$ are not mutually distinct.
Then either 
(i) $\tau_j \in \{\tau_{j-1}, \tau_{j+1}\}$; or (ii) $\tau_{j-1} = \tau_{j+1}$.
For the moment assume (i).
Then by \cite[Corollary 4.6]{CN:someformula} we get $\tau_i \in \{\tau_0, -\tau_0\}$ 
for $0 \leq i \leq D$,
contradicting Lemma \ref{lem:Htype}.
Next assume (ii).
Then $\tau_1 \in \{\tau_0, - \tau_0\}$ by \cite[Corollary 4.7]{CN:someformula}.
By this and \cite[Corollary 4.6]{CN:someformula} we get $\tau_i \in \{\tau_0, - \tau_0\}$ 
for $0 \leq i \leq D$,
contradicting Lemma \ref{lem:Htype}.
The claim is proved.
In \cite[Theorem 3.8]{C:thin} Curtin showed that every irreducible $\T$-module is thin under the
condition that 
\begin{equation}
  \tau_i \not\in \{\tau_0, - \tau_0\} \; \text{ for }  1 \leq i \leq D.     \label{eq:aux2}
\end{equation}
Analyzing the proof of this theorem,
we find that this theorem is implied by Lemmas 3.4 and 3.7 in \cite{C:thin}.
Analyzing the proof of these lemmas, we find that these lemmas still hold
if we replace the condition \eqref{eq:aux2} with \eqref{eq:aux}.
Thus the conclusion of \cite[Theorem 3.8]{C:thin} holds under the assumption \eqref{eq:aux}.
The result follows.
\end{proof}

Let $U$ denote a thin irreducible $\T$-module with endpoint $r$, dual endpoint $s$,
and diameter $d$.
By Lemmas \ref{lem:U} the sequence
\begin{equation}
   (\A_1; \{\E_{s+i}\}_{i=0}^d; \A^*_1; \{\E^*_{r+i}\}_{i=0}^d)                \label{eq:LS5}
\end{equation}
acts on $U$ as a Leonard system.
By  \cite[Lemma 4.4, Theorems 4.1, 5.5]{CN:hyper} 
we have $r=s$, and the Leonard system \eqref{eq:LS5} on $U$ is self-dual.

\begin{lemma}    \label{lem:spinLP3}    \samepage
\ifDRAFT {\rm lem:spinLP3}. \fi
The pair $\A_1, \A^*_1$ acts on $U$ as a spin Leonard pair,
and $\W, \W^*$ acts on $U$ as a balanced Boltzmann pair for this
Leonard pair.
\end{lemma}

\begin{proof}
By Definition \ref{def:spinLP} and \eqref{eq:A1WsW=WsWAs1}
we see that the pair $\A_1, \A^*_1$ acts on $U$ as a spin Leonard pair,
and the pair $\W, \W^*$ acts on $U$ as a Boltzmann pair for this Leonard pair.
This Boltzmann pair is balanced by Definition \ref{def:balanced} and \eqref{eq:balanced2}.
\end{proof}

\section{From a spin model to spin Leonard pairs; the $q$-Racah case}
\label{sec:spinqRacah}

Let $\Gamma$ denote a distance-regular graph with vertex set $X$,
diameter $D \geq 3$, and valency at least $3$.
Recall the distance-matrices $\{\A_i\}_{i=0}^D$ 
in the Bose-Mesner algebra $\M$ of $\Gamma$.
Let $\W$ denote a spin model afforded by $\Gamma$.
Fix an ordering $\{\E_i\}_{i=0}^D$ of the primitive idempotents of $\Gamma$ 
from Lemma \ref{lem:MPsi}.
Note that $\Gamma$ is formally self-dual with respect to $\{\E_i\}_{i=0}^D$.
Let $\{\th_i\}_{i=0}^D$ denote the eigenvalue sequence of $\Gamma$
with respect to the ordering $\{\E_i\}_{i=0}^D$.
As we saw below Lemma \ref{lem:MPsi},
there exist nonzero scalars $f$, $\{\tau_i\}_{i=0}^D$ in $\C$ such that $\tau_0=1$ and
\[
  \W = f \sum_{i=0}^D \tau_i \E_i.
\]
Fix $x \in X$.
Abbreviate $\T=\T(x)$ and $\A^*_i = \A^*_i(x)$, $\E^*_i = \E^*_i(x)$ for
$0 \leq i \leq D$.
Define
\[
  \W^* = f \sum_{i=0}^D \tau_i \E^*_i.
\]
Recall the primary $\T$-module $\pU$.
By Proposition \ref{prop:DRGsd}  the sequence
\begin{equation*}
 (\A_1; \{\E_i\}_{i=0}^D; \A^*_1; \{E^*_i \}_{i=0}^D)            
\end{equation*}
acts on $\pU$ as a self-dual Leonard system.
This Leonard system has  eigenvalue sequence  $\{\th_i\}_{i=0}^D$.
Fix $0 \neq q \in \C$,
and assume that the ordering $\{\E_i\}_{i=0}^D$ has $q$-Racah type.

\begin{lemma}    \label{lem:alphabeta}    \samepage
\ifDRAFT {\rm lem:alphabeta}. \fi
There exist scalars $a$, $\alpha$, $\beta$ in $\C$ with $a$, $\alpha$ nonzero
that satisfy the following {\rm (i), (ii)}.
\begin{itemize}
\item[\rm (i)]
The sequence 
\begin{equation}
 (\A; \{\E_i\}_{i=0}^D; \A^*; \{\E^*_i\}_{i=0}^D)    \label{eq:LSonpU}
\end{equation}
acts on $\pU$ as a Leonard system of $q$-Racah type,
where
\begin{align}
  \A_1 &= \alpha \A + \beta \I, &
  \A^*_1 &= \alpha \A^* + \beta \I.                  \label{eq:defAAs0}
\end{align}
\item[\rm (ii)]
The Leonard system \eqref{eq:LSonpU} on $\pU$ has Huang data $(a,a,a,D)$.
\end{itemize}
\end{lemma}

\begin{proof}
By Definition \ref{def:qRacahDRG} there exists $\alpha$, $\beta \in \C$ with $\alpha \neq 0$
such that the sequence \eqref{eq:LSonpU} acts on $\pU$ as a Leonard system of $q$-Racah type, 
where $\A$, $\A^*$ are from \eqref{eq:defAAs0}.
This Leonard system is self-dual.
By Definition \ref{def:qRacah} there exists $0 \neq a \in \C$ such that
\begin{align}
 \th_i &= \alpha (a q^{2i-D} + a^{-1} q^{D-2i}) + \beta   &&  (0 \leq i \leq D).   \label{eq:thi2}
\end{align}
By \eqref{eq:defAAs0} and \eqref{eq:thi2} this Leonard system has eigenvalue 
sequence $\{a q^{2i-D} + a^{-1} q^{D-2i}\}_{i=0}^D$.
By Lemmas \ref{lem:affine}, \ref{lem:spinLP3} 
the pair $\A, \A^*$ acts on $\pU$ as a spin Leonard pair.
By these comments and Lemmas \ref{lem:selfdualab}, \ref{lem:spinLP0}(i)
the Leonard system $\eqref{eq:LSonpU}$ has Huang data $(a,a,c,D)$ with 
$c \in \{a,  a^{-1}, -a, - a^{-1} \}$.
By Note \ref{note:Huang} we may assume $c \in \{a, -a\}$.
Replacing $\A, \A^*$ with $-\A, - \A^*$ and replacing $\alpha$ with $- \alpha$,
we may assume by Lemma \ref{lem:Huang2} that $c = a$.
The result follows.
\end{proof}

For the rest of this section, let the scalars $a, \alpha,\beta$ be as in
Lemma \ref{lem:alphabeta}.

\begin{lemma}    {\rm \cite[Corollaries 6.5, 6.7]{CW} }
\label{lem:aq}    \samepage
\ifDRAFT {\rm lem:aq}. \fi
We have
\begin{align}
  q^{2i} &\neq 1 && (1 \leq i \leq D),                \label{eq:restq}
\\
  a^2 q^{2i} &\neq 1  &&  (1-D \leq i \leq D-1),              \label{eq:resta}
\\
 a^3 q^{2i-D-1} &\neq 1  &&  (1 \leq i \leq D).               \label{eq:resta2}
\end{align}
\end{lemma}

\begin{proof}
Apply Lemma \ref{lem:condabc} to the Leonard system \eqref{eq:LSonpU} on $\pU$.
\end{proof}

\begin{lemma}    \label{lem:ti}    \samepage
\ifDRAFT {\rm lem:ti}. \fi
One of the following holds:
\begin{align}
  \tau_i &= (-1)^i a^{-i} q^{i (D-i)}   &&  (0 \leq i \leq D),                \label{eq:ti}
\\
  \tau_i^{-1} &= (-1)^i a^{-i} q^{i (D-i)}   &&  (0 \leq i \leq D).          \label{eq:tiinv}
\end{align}
\end{lemma}

\begin{proof}
Apply Lemma \ref{lem:W}(i) to the Leonard system  \eqref{eq:LSonpU} on $\pU$,
and use Lemma \ref{lem:alphabeta}.
\end{proof}

\begin{note}    \label{note:tiinv}    \samepage
\ifDRAFT {\rm note:tiinv}. \fi
Replacing $\W$ by $\W^-$ if necessary,
and using Lemmas \ref{lem:afford}, \ref{lem:W-2},
we may assume that \eqref{eq:ti} holds.
We adopt this assumption for the rest of the section.
\end{note}

We mention a lemma for later use.
For the moment, let $D$, $d$, $r$ denote arbitrary nonnegative integers
such that $r+d \leq D$,
and let $a$, $q$ denote arbitrary nonzero scalars in $\C$.

\begin{lemma}    \label{lem:tr+i}    \samepage
\ifDRAFT {\rm lem:tr+i}. \fi
With the above notation, define  $\widetilde{a} = a q^{2r+d-D}$.
Then the following hold.
\begin{itemize}
\item[\rm (i)]
Define
\begin{align*}
   \th_i &= a q^{2i-D} + a^{-1} q^{D-2i}    && (0 \leq i \leq D).
\end{align*}
Then
\begin{align*}
  \th_{r+i} &= \widetilde{a} q^{2i-d} + (\widetilde{a})^{-1} q^{d-2i}  &&  (0 \leq i \leq d).
\end{align*}
\item[\rm (ii)]
Define
\begin{align*}
   \tau_i &= (-1)^i a^{-i} q^{i (D-i)}    &&  (0 \leq i \leq D).
\end{align*}
Then
\begin{align}
   \tau_{r+i} &= \tau_r \widetilde{\tau}_i    &&   (0 \leq i \leq d),     \label{eq:tr+i}
\end{align}
where
\begin{align}
  \widetilde{\tau}_i &= (-1)^i ( \widetilde{a})^{-i} q^{i(d-i)}  &&  (0 \leq i \leq d).    \label{eq:deftti}
\end{align}
\end{itemize}
\end{lemma}

\begin{proof}
Routine verification.
\end{proof}

We return our attention to $\Gamma$.
Let $U$ denote a thin irreducible $\T$-module, with endpoint $r$ and diameter $d$.
Note  that $U$ has dual endpoint $r$ by the sentence below \eqref{eq:LS5}.
Consider the sequence
\[
   \Phi = ( \A_1; \{\E_{r+i}\}_{i=0}^d; \A^*_1; \{\E^*_{r+i} \}_{i=0}^d).
\]

\begin{lemma}    \label{lem:Phi}    \samepage
\ifDRAFT {\rm lem:Phi}. \fi
The sequence $\Phi$ acts on $U$ as a self-dual Leonard system,
with eigenvalue sequence  $\{\th_{r+i} \}_{i=0}^d$.
\end{lemma}

\begin{proof}
Apply Lemma \ref{lem:U} to the Leonard system $\Phi$ on $U$.
\end{proof}

\begin{lemma}    \label{lem:WonU}    \samepage
\ifDRAFT {\rm lem:WonU}. \fi
On $U$ we have
\begin{align}
 \W &= f \tau_r \sum_{i=0}^d \widetilde{\tau}_i   \E_{r+i}, &
 \W^* &= f \tau_r \sum_{i=0}^d \widetilde{\tau}_i \E^*_{r+i},    \label{eq:WonU}
\end{align}
where $\{ \widetilde{\tau}_i\}_{i=0}^d$ are from \eqref{eq:deftti}.
\end{lemma}

\begin{proof}
By construction, on $U$ we have
\begin{align*}
 \W &= f \sum_{i=0}^d \tau_{r+i} \E_{r+i},  &
 \W^* &= f \sum_{i=0}^d \tau_{r+i} \E^*_{r+i}.
\end{align*}
By this and \eqref{eq:tr+i} we get the result.
\end{proof}

Our next goal is to obtain the intersection numbers of the Leonard system $\Phi$ on $U$.

\begin{lemma}     \label{lem:tildePhi}     \samepage
\ifDRAFT {\rm lem:tildePhi}. \fi
The sequence 
\begin{equation}
    (\A; \{\E_{r+i} \}_{i=0}^d; \A^*; \{\E^*_{r+i}\}_{i=0}^d)              \label{eq:tildePhi}
\end{equation}
acts on $U$ as a Leonard system of $q$-Racah type with Huang data
\begin{equation*}
   (a q^{2r+d-D}, a q^{2r+d-D}, a q^{2r+d-D}, d).                   
\end{equation*}
\end{lemma}

\begin{proof}
By Lemma \ref{lem:Phi} and \eqref{eq:defAAs0}
the sequence \eqref{eq:tildePhi} acts on $U$
as a self-dual Leonard system.
For $0 \leq i \leq D$ define
\begin{align*}
  \widetilde{\th}_i &= a q^{2i-D} + a^{-1} q^{D-2i},
\end{align*}
and observe by \eqref{eq:thi2} that
$\th_i = \alpha \widetilde{\th}_i + \beta$.
Now by \eqref{eq:defAAs0} and Lemma \ref{lem:Phi} 
the Leonard system \eqref{eq:tildePhi} on $U$
has eigenvalue sequence and dual eigenvalue sequence $\{\widetilde{\th}_{r+i}\}_{i=0}^d$.
By Lemma \ref{lem:tr+i}(i)
 the Leonard system \eqref{eq:tildePhi} on $U$ has $q$-Racah type
with Huang data $(\widetilde{a}, \widetilde{a}, \widetilde{c},d)$ for some nonzero
$\widetilde{c} \in \C$,
where $\widetilde{a}$ is from Lemma \ref{lem:tr+i}.
By Lemmas \ref{lem:affine}, \ref{lem:spinLP3} the pair $\A, \A^*$ acts on $U$
as a spin Leonard pair, and $\W, \W^*$ acts on $U$
as a Boltzmann pair for this Leonard pair.
By this and Lemmas \ref{lem:new}, \ref{lem:WonU} we obtain
$\widetilde{c} \in \{ \widetilde{a}, (\widetilde{a})^{-1} \}$.
By this and Note \ref{note:Huang} the sequence
$(\widetilde{a}, \widetilde{a}, \widetilde{a}, d)$ is a Huang data for $\Phi$ on $U$.
The result follows.
\end{proof}

Recall the intersection numbers $\{c_i (U)\}_{i=1}^d$,
 $\{b_i(U)\}_{i=0}^{d-1}$ for the Leonard system $\Phi$ on $U$.

\begin{lemma}  {\rm \cite[Theorem 8.5]{CW} }
\label{lem:ciUbiU}    \samepage
\ifDRAFT {\rm lem:ciUbiU}. \fi
For the Leonard system $\Phi$ on $U$
the intersection numbers  satisfy
\begin{align}
b_i(U) &=
 \frac{ \alpha (q^{i-d} - q^{d-i}) (a q^{2r+i -D} - a^{-1} q^{D-2r-i})(a^3-q^{3D-2d-6r-2i-1}) }
        {a q^{D-d-2r} (a q^{2r+2i-D} - a^{-1} q^{D-2r-2i} ) (a + q^{D-2r-2i-1} ) },          \label{eq:biU}
\\
c_i(U) &=
 \frac{ \alpha a (q^{i}- q^{-i})(a q^{d+2r+i-D} - a^{-1} q^{D-d-2r-i})
                  (a^{-1} - q^{2d-D+2r-2i+1}) }
        {q^{d-D+2r} (a q^{2r+2i-D} - a^{-1} q^{D-2r-2i}) (a + q^{D-2r-2i+1} ) }      \label{eq:ciU}
\end{align}
for $1 \leq i \leq d-1$ and
\begin{align}
b_0(U) &=
  \frac{ \alpha (q^{-d} - q^d) ( a^3 - q^{3D-2d-6r-1} ) }
         {a q^{D-d-2r} (a + q^{D-2r-1}) },          \label{eq:b0U}
\\
c_d(U) &=
 \frac{ \alpha  (q^{-d} - q^d)(a - q^{D-2r-1}) }
        { q^{d-1} (a + q^{D-2d-2r+1}) }.                        \label{eq:cdU}
\end{align}
\end{lemma}

\begin{proof}
For the Leonard system \eqref{eq:tildePhi} on $U$,
compute the intersection numbers using Lemmas \ref{lem:qRacahbi} and \ref{lem:tildePhi}.
Adjust these intersection numbers using \eqref{eq:defAAs0} to get the result.
\end{proof}

Recall the intersection numbers 
$\{\c_i\}_{i=1}^D$,  $\{\b_i\}_{i=0}^{D-1}$ of $\Gamma$.

\begin{lemma}   {\rm \cite[Theorem 1.1]{CN:someformula} }
\label{lem:cibi}    \samepage
\ifDRAFT {\rm lem:cibi}. \fi
The intersection numbers of $\Gamma$ satisfy
\begin{align}
\b_0 &=
  \frac{ \alpha (q^{-D} - q^D)(a^3 - q^{D-1}) }
         {a (a + q^{D-1}) },                              \label{eq:b0}
\\
\b_i &=
  \frac{ \alpha (q^{i-D} - q^{D-i})(a q^{i-D} - a^{-1} q^{D-i}) (a^3 - q^{D-2i-1}) }
         {a (a q^{2i-D} - a^{-1} q^{D-2i}) (a + q^{D-2i-1}) }    && (1 \leq i \leq D-1),   \label{eq:bi}
\\
\c_i &=
 \frac{\alpha a (q^{i} - q^{-i})(a q^i - a^{-1} q^{-i}) (a^{-1} - q^{D-2i+1} ) }
        { (a q^{2i-D} - a^{-1} q^{D-2i}) (a + q^{D-2i+1} ) } &&  (1 \leq i \leq D-1),    \label{eq:ci}
\\
\c_D &=
 \frac{ \alpha (q^{-D} - q^D) (a - q^{D-1}) }
        {q^{D-1} (a + q^{1-D}) }.                           \label{eq:cd}
\end{align}
\end{lemma}

\begin{proof}
Apply Lemma \ref{lem:ciUbiU} with $U=\pU$ and use Lemma \ref{lem:ciaibi}.
\end{proof}

\begin{lemma}      {\rm \cite[Theorem 6.4, 6.8]{CW} }
\label{lem:thi}    \samepage
\ifDRAFT {\rm lem:thi}. \fi
We have
\begin{align}
 \alpha &= 
  \frac{(a q^{2-D} - a^{-1} q^{D-2}) ( a + q^{D-1}) }
         {q^{D-1} (q^{-1} - q)(a q - a^{-1} q^{-1}) (a - q^{1-D}) },        \label{eq:alpha}
\\
 \beta &=
   \frac{q (a+a^{-1}) (a + q^{-D-1}) (a q^{2-D}- a^{-1} q^{D-2}) }
            { (q-q^{-1}) (a - q^{1-D}) (a q - a^{-1} q^{-1}) }.                \label{eq:beta}
\end{align}
\end{lemma}

\begin{proof}
By \eqref{eq:a0b0c1}, $\c_1=1$.
By this and \eqref{eq:ci} we obtain \eqref{eq:alpha}.
By \eqref{eq:a0b0c1}, \eqref{eq:kjP0j}, \eqref{eq:defthithsi} we have $\b_0 = \th_0$.
By this and \eqref{eq:thi2}, \eqref{eq:b0}, \eqref{eq:alpha} we obtain 
\eqref{eq:beta}.
\end{proof}

\begin{lemma}    \label{lem:X}    \samepage
\ifDRAFT {\rm lem:X}. \fi
We have
\begin{align*}
  |X| &= \frac{a^{3D}  (a^{-2}; q^2)_D^2 }
                  {q^{D(D-1)} (a q^{1-D};q^2)_D^2}.
\end{align*}
\end{lemma}

\begin{proof}
By Lemmas \ref{lem:qRacahnu} and \ref{lem:nu3}.
\end{proof}
 
Recall the scalars $\{\k_i\}_{i=0}^D$ from \eqref{eq:defki2}.

\begin{lemma}    \label{lem:ki4}    \samepage
\ifDRAFT {\rm lem:ki4}. \fi
We have $\k_0 = 1$, and
\[
\k_i =
 \frac{q^{2 i D} (1-a^2 q^{4i-2D}) \, (a^2 q^{2-2D};q^2)_i \, (q^{-2D}; q^2)_i \, (a^3 q^{1-D};q^2)_i  }
        {a^{2i} (1-a^2 q^{2i-2D}) \, (a^2 q^2; q^2)_i \,  (q^2; q^2)_i \,  (a^{-1} q^{1-D}; q^2)_i}
\]
for $1 \leq i \leq D-1$, and
\[
\k_D = 
 \frac{q^{2 D^2} (a^2 q^{2-2D};q^2)_D \, (q^{-2D}; q^2)_D \, (a^3 q^{1-D};q^2)_D  }
        {a^{2D} (a^2; q^2)_D \, (q^2; q^2)_D  \, (a^{-1} q^{1-D}; q^2)_D}.
\]
\end{lemma}

\begin{proof}
By Lemmas \ref{lem:ki3} and \ref{lem:nuki}.
\end{proof}

\begin{lemma}    \label{lem:sumkiti4}    \samepage
\ifDRAFT {\rm lem:sumkiti4}. \fi
We have
\begin{align*}
 \sum_{i=0}^D \k_i \tau_i &=      
  \frac{ (a^{-2}; q^2)_D } { (a q^{1-D}; q^2)_D }.                     
\end{align*}
\end{lemma}

\begin{proof}
By Lemmas \ref{lem:sumkiti} and \ref{lem:nuki}.
\end{proof}

We have been discussing the $q$-Racah case.
For a discussion of the non $q$-Racah cases, see \cite{CN:someformula}.

\section{From spin Leonard pairs to a spin model}
\label{sec:spinLPtospinmodel}

In Sections \ref{sec:spindrg}, \ref{sec:spinqRacah} we considered a distance-regular
graph $\Gamma$ that affords a spin model $\W$.
We used $\W$ to construct a family of spin Leonard pairs.
We now reverse the logical direction,
and obtain a spin model afforded by $\Gamma$ from some spin Leonard pairs.
In this section we work with general $\Gamma$,
and in the next section we consider the $q$-Racah case.
For this section our main results are Theorems \ref{thm:main0} and \ref{thm:main}.

Let $\Gamma$ denote a distance-regular graph with vertex set $X$
and diameter $D \geq 1$.
Recall the distance-matrices $\{\A_i\}_{i=0}^D$ and
the primitive idempotents $\{\E_i\}_{i=0}^D$ in the Bose-Mesner algebra $\M$ of $\Gamma$.
Assume that $\Gamma$ is formally self-dual with respect to $\{\E_i\}_{i=0}^D$.
For $x \in X$, consider the dual Bose-Mesner algebra $\M^*(x)$
and the Terwilliger algebra $\T(x)$.
For $0 \leq i \leq D$ we have the elements $\E^*_i(x)$ and $\A^*_i(x)$ from
the paragraph below Lemma \ref{lem:Pij}.
Let $f$, $\{\tau_i\}_{i=0}^D$ denote nonzero scalars in $\C$ such that $\tau_0 = 1$.
Define
\begin{equation*}
   \W = f \sum_{i=0}^D \tau_i \E_i.                        
\end{equation*}
For $x \in X$ define
\begin{equation}
   \W^*(x) = f \sum_{i=0}^D \tau_i \E^*_i(x).            \label{eq:defWs}
\end{equation}
Note that $\W$ and $\W^*(x)$ are invertible.
For the rest of this section, we make the following assumption.

\begin{assumption}       \label{assump0}    \samepage
\ifDRAFT {\rm assump0}. \fi
Assume that for all  $x \in X$ and  all irreducible $\T(x)$-modules $U$,
\begin{itemize}
\item[\rm (i)]
$U$ is thin;
\item[\rm (ii)]
$U$ has the same endpoint and dual-endpoint;
\item[\rm (iii)]
the pair $\A_1, \A^*_1(x)$ acts on $U$ as a spin Leonard pair,
and  $\W, \W^*(x)$ acts on $U$ as a balanced Boltzmann pair for this spin Leonard pair;
\item[\rm (iv)]
$f$ satisfies \eqref{eq:f-2}.
\end{itemize}
\end{assumption}

Under Assumption \ref{assump0} we show that $\W$ is a spin model afforded by $\Gamma$.
From now until the beginning of Theorem \ref{thm:main},
fix $x \in X$ and abbreviate $\W^* = \W^*(x)$,
$\T = \T(x)$ and $\E^*_i = \E^*_i (x)$, $\A^*_i = \A^*_i (x)$ for $0 \leq i \ \leq D$.

\begin{lemma}    \label{lem:type2pre}    \samepage
\ifDRAFT {\rm lem:type2pre}. \fi
We have 
\begin{align}
  \W &= |X|^{1/2} f^{-1}  \sum_{i=0}^D \tau_i^{-1} \A_i,  &
  \W^{-1} &= |X|^{-3/2} f  \sum_{i=0}^D \tau_i \A_i.                                \label{eq:WWinv2}
\end{align}
\end{lemma}

\begin{proof}
We first show that \eqref{eq:WWinv2} holds on the primary $\T$-module $\pU$.
By Proposition \ref{prop:DRGsd}(i), (iii)
the sequence
\begin{equation}
    (\A_1; \{\E_i\}_{i=0}^D; \A^*_1; \{\E^*_i\}_{i=0}^D)                \label{eq:LS6}
\end{equation}
acts on $\pU$ as a self-dual Leonard system.
Let $\sigma : \text{\rm End}(\pU) \to \text{\rm End}(\pU)$ denote the duality of this Leonard system.
So $\sigma$ swaps the action of $\A_1$, $\A^*_1$ on $\pU$,
and swaps the action of $\E_i$, $\E^*_i$ on $\pU$ for $0 \leq i \leq D$.
By Assumption \ref{assump0}(iii), $\A_1, \A^*_1$ acts on $\pU$ as a spin Leonard pair,
and $\W, \W^*$ acts on $\pU$ as a balanced Boltzmann pair for this Leonard pair.
In Lemmas \ref{lem:PSL}, \ref{lem:PSL2} we used this Boltzmann pair to obtain a duality
for this Leonard pair,
which must coincide with the above duality $\sigma$ by the uniqueness of the duality.
By this and \eqref{eq:defEsi}, \eqref{eq:Phi2} we find that the results below \eqref{eq:Phi2}
in Section \ref{sec:spinLP} apply to the Leonard system \eqref{eq:LS6} on $\pU$.
Apply Theorem \ref{thm:WtiinvAi}, Lemma \ref{lem:piW} and \eqref{eq:f-2} to this Leonard system,
and use Lemmas \ref{lem:nu3}, \ref{lem:AiEs0E03},  \ref{lem:nuki}
to show that \eqref{eq:WWinv2} holds on $\pU$.
The result follows in view of Corollary \ref{cor:faithful}.
\end{proof}

\begin{lemma}    \label{lem:type2}    \samepage
\ifDRAFT {\rm lem:type2}. \fi
The matrix $\W$ is type II.
\end{lemma}

\begin{proof}
We invoke Lemma \ref{lem:type2W-}.
By construction $\W$ is symmetric.
Each entry of $\W$ is nonzero by the equation on the left in \eqref{eq:WWinv2}.
By that equation and \eqref{eq:defW-} we obtain
\begin{equation}
   \W^- = |X|^{-1/2} f  \sum_{i=0}^D \tau_i \A_i.                      \label{eq:W-}
\end{equation}
Comparing \eqref{eq:W-} with the equation on the right in \eqref{eq:WWinv2},
we obtain 
$\W^{-} = |X| \W^{-1}$.
Therefore $\W \W^- = |X| I$.
The result follows by Lemma \ref{lem:type2W-}.
\end{proof}

\begin{lemma}   \label{lem:WWsW} \samepage
\ifDRAFT {\rm lem:WWsW}. \fi
We have
\begin{equation} 
  \W \W^* \W = \W^* \W \W^*.             \label{eq:WWsW}
\end{equation}
\end{lemma}

\begin{proof}
By Assumption \ref{assump0}(iii) the pair $\W, \W^*$ acts on each irreducible $\T$-module
as a balanced Boltzmann pair.
The result follows in view of Lemma \ref{lem:decomp}(i).
\end{proof}

\begin{lemma}    \label{lem:WsW-}    \samepage
\ifDRAFT {\rm lem:WsW-}. \fi
For $y \in X$,
\begin{align*}
  \W^* (y,y) &= \frac{ |X|^{1/2} } {\W (x,y)} .                       
\end{align*}
\end{lemma}

\begin{proof}
Let $i = \partial (x,y)$.
By \eqref{eq:defWs} we obtain $\W^*(y,y) = f \tau_i$.
By the equation on the left in \eqref{eq:WWinv2} we obtain
$\W(x,y ) = |X|^{1/2} f^{-1} \tau_i^{-1}$.
The result follows.
\end{proof}

\begin{lemma}   \label{lem:type3}     \samepage
\ifDRAFT {\rm lem:type3}. \fi
For $a,b \in X$,
\[
\sum_{y \in X} \frac{\W(a,y) \W(b,y) } { \W(x,y) }
 = |X|^{1/2} \, \frac{ \W(a,b) } { \W(a,x) \W(b,x) } .      
\]
\end{lemma}

\begin{proof}
For each side of \eqref{eq:WWsW},
compute the $(a,b)$-entry and evaluate the result using Lemma \ref{lem:WsW-}.
Here are some details:
\begin{align*}
  (\W \W^* \W)(a,b)
 &= \sum_{y \in X} \W(a,y) \W^*(y,y) \W(y,b) 
\\
 &= |X|^{1/2}  \,\sum_{y \in X} \frac{\W(a,y) \W(b,y)}{\W(x,y)},
\end{align*}
and
\begin{align*}
 (\W^* \W \W^*)(a,b)
 &= \W^*(a,a) \W(a,b) \W^*(b,b)
\\
 &= |X| \, \frac{\W(a,b)}{\W(a,x) \W(b,x)}.
\end{align*}
\end{proof}

\begin{theorem}    \label{thm:main0}    \samepage
\ifDRAFT {\rm thm:main0}. \fi
The matrix $\W$ is a spin model.
\end{theorem}

\begin{proof}
The matrix $\W$ satisfies the requirements of Definition \ref{def:spin}
by Lemmas \ref{lem:type2} and \ref{lem:type3}.
\end{proof}

Next we show that the spin model $\W$ is afforded by $\Gamma$.

\begin{lemma}    \label{lem:WWsbEi}   \samepage
\ifDRAFT {\rm lem:WWsbEi}. \fi
For $0 \leq i \leq D$, 
\begin{equation}
  \W \W^* \E_i = \E^*_i \W \W^*.            \label{eq:WWsEiEsiWWs}
\end{equation}
\end{lemma}

\begin{proof}
Let $U$ denote an irreducible $\T$-module with endpoint $r$ and diameter $d$.
By Assumption \ref{assump0}(i), (ii) we see that $U$ is thin with dual endpoint $r$.
By these comments and Lemma \ref{lem:U} the sequence
\[
   (\A_1; \{\E_{r+j}\}_{j=0}^d; \A^*_1; \{\E^*_{r+j} \}_{j=0}^d)
\]
acts on $U$ as a Leonard system.
By Assumption \ref{assump0}(iii) the pair $\W, \W^*$ acts $U$ as a balanced Boltzmann pair
for the Leonard pair $\A_1, \A^*_1$ on $U$.
By these comments and Lemma \ref{lem:EiWsW},
\begin{align*}
    \W \W^* \E_{r+j} &= \E^*_{r+j} \W \W^*   
\end{align*}
holds on $U$ for $0 \leq j \leq d$.
Thus \eqref{eq:WWsEiEsiWWs} holds on each irreducible $\T$-module.
Now \eqref{eq:WWsEiEsiWWs} holds by Lemma \ref{lem:decomp}(i).
\end{proof}

\begin{theorem}    \label{thm:main}    \samepage
\ifDRAFT {\rm thm:main}. \fi
The spin model $\W$ is afforded by $\Gamma$.
\end{theorem}

\begin{proof}
By Definition \ref{def:afford} we must show that $\W \in \M \subseteq N(\W)$.
By construction $\W \in \M$.
To obtain $\M \subseteq N(\W)$, we show that for all $b,c \in X$ the vector $\text{\bf u}_{b,c}$ 
is contained in $\E_i \V$, where $i = \partial (b,c)$.
Using Definition \ref{def:NW} and Lemma \ref{lem:WsW-},
\[
  \text{\bf u}_{b,c} = (\W^*(b))^{-1} \W^* (c) \text{\bf 1}
  = (\W \W^*(b))^{-1} \W \W^*(c) \text{\bf 1}.
\]
Using Lemma \ref{lem:WWsbEi} with $i=0$,
\[
 \W \W^*(c) \text{\bf 1} \in \W \W^*(c) \E_0 \V 
 = \E^*_0 (c) \W \W^*(c)  \V =  \E^*_0 (c)  \V
  = \text{\rm Span} \{\widehat{c} \} \subseteq \E^*_i(b) \V.
\]
By the above comments and Lemma \ref{lem:WWsbEi},
\[
  \text{\bf u}_{b,c} \in (\W \W^*(b))^{-1} \E^*_i (b)  \V 
  = \E_i  (\W \W^*(b))^{-1} \V = \E_i \V.
\] 
We have shown that $\text{\bf u}_{b,c}$ is contained in $\E_i \V$
and is therefore a common eigenvector for $\M$.
So $\M \subseteq N(\W)$ in view of Definition \ref{def:NW}.
The result follows.
\end{proof}

\section{From spin Leonard pairs to a spin model; the $q$-Racah case}
\label{sec:construction}

In this section we start with a suitable distance-regular graph $\Gamma$
of $q$-Racah type, and obtain a spin model afforded by $\Gamma$.
The result is given in Theorem \ref{thm:main3}.

Let $\Gamma$ denote a distance-regular graph with vertex set $X$,
diameter $D \geq 1$,
and intersection numbers  $\{\c_i\}_{i=1}^D$, $\{\a_i\}_{i=0}^D$, $\{\b_i\}_{i=0}^{D-1}$.
Recall the distance-matrices  $\{\A_i\}_{i=0}^D$
and primitive idempotents $\{\E_i\}_{i=0}^D$ in the Bose-Mesner algebra $\M$ of $\Gamma$. 
Throughout this section we make the following assumptions.

\begin{assumption}    \label{assump}    \samepage
\ifDRAFT {\rm assump}. \fi
Assume that $\Gamma$ is formally self-dual with respect to $\{\E_i\}_{i=0}^D$,
with eigenvalue sequence  $\{\th_i\}_{i=0}^D$.
Assume that there exist nonzero scalars $a$, $q \in \C$
that satisfy \eqref{eq:thi2}--\eqref{eq:resta2} and
 \eqref{eq:b0}--\eqref{eq:cd} with
$\alpha$, $\beta$ from \eqref{eq:alpha}, \eqref{eq:beta}.
Assume that for all $x \in X$ and all irreducible $\T(x)$-modules $U$,
\begin{itemize}
\item[\rm (i)]
$U$ is thin;
\item[\rm (ii)]
$U$ has the same endpoint and dual-endpoint;
\item[\rm (iii)]
the intersection numbers $\{c_i(U)\}_{i=1}^d$, $\{b_i(U)\}_{i=0}^{d-1}$ satisfy
\eqref{eq:biU}--\eqref{eq:cdU}.
\end{itemize}
\end{assumption}

Under Assumption \ref{assump} we construct a spin model $\W$
that is afforded by $\Gamma$.

\begin{definition}     \label{def:main3}    \samepage
\ifDRAFT {\rm def:main3}. \fi
Define scalars $\{\tau_i\}_{i=0}^D$ in $\C$ by
\begin{align}
   \tau_i &=  (-1)^i a^{-i} q^{i(D-i)}   &&  (0 \leq i \leq D).         \label{eq:ti3}
\end{align}
Define $f \in \C$ such that
\begin{equation}
  f^2 = \frac{ |X|^{3/2} (a q^{1-D}; q^2)_D } 
                 { (a^{-2}; q^2)_D }.                 \label{eq:f}
\end{equation}
Note that $f$, $\{\tau_i\}_{i=0}^D$ are nonzero.
Define
\begin{equation*}
  \W = f \sum_{i=0}^D \tau_i E_i.                        
\end{equation*}
\end{definition}

We are going to show that $\W$ is a spin model afforded by $\Gamma$.
Fix $x \in X$ and abbreviate $\T = \T(x)$
and $\E^*_i = \E^*_i (x)$, $\A^*_i = \A^*_i (x)$ for $0 \leq i \leq D$.
Recall the scalars $\alpha$, $\beta$ from \eqref{eq:alpha}, \eqref{eq:beta}.
Define $\A \in \M$ and $\A^* \in \M^*$ such that
\begin{align}
 \A_1 &= \alpha \A + \beta \I,  &
 \A^*_1 &= \alpha \A^* + \beta \I.           \label{eq:defAAs}
\end{align}
Let $U$ denote an irreducible $\T$-module with endpoint $r$ and diameter $d$.
By Lemma \ref{lem:U} and since $U$ is thin, each of the sequences
\begin{align*}
  \Phi &=  (\A_1; \{\E_{r+i}\}_{i=0}^d; \A^*_1 ; \{\E^*_{r+i}\}_{i=0}^d),
\\
 \widetilde{\Phi} &=  (\A; \{\E_{r+i}\}_{i=0}^d; \A^* ; \{\E^*_{r+i}\}_{i=0}^d)
\end{align*}
acts on $U$ as a Leonard system.
By construction these Leonard systems are self-dual.

\begin{lemma}    \label{lem:tPhi}     \samepage
\ifDRAFT {\rm lem:tPhi}. \fi
The Leonard system $\widetilde{\Phi}$ on $U$ has $q$-Racah type with Huang
data 
\begin{equation}
   ( a q^{2r+d-D}, a q^{2r+d-D}, a q^{2r+d-D},  d).      \label{eq:Huangdata}
\end{equation}
\end{lemma}

\begin{proof}
By \cite[Lemma 7.3]{Huang} there exists a
Leonard system $\Psi$ of $q$-Racah type with Huang data \eqref{eq:Huangdata}.
By Lemma \ref{lem:selfdualab} this Leonard system is self-dual.
The eigenvalue sequence and intersection numbers of $\Psi$ are 
obtained using Definition \ref{def:Huangdata} and Lemma \ref{lem:qRacahbi}.
The eigenvalue sequence and intersection numbers of $\Phi$ on $U$
are given in Assumption \ref{assump}.
Adjusting these using \eqref{eq:defAAs} we obtain the eigenvalue sequence
and intersection numbers of $\widetilde{\Phi}$ on $U$.
These match the eigenvalue sequence and intersection numbers of $\Psi$.
By Lemmas \ref{lem:isomorphic}, \ref{lem:parray} the Leonard system $\Psi$
is isomorphic to the Leonard system $\widetilde{\Phi}$ on $U$.
The result follows.
\end{proof} 

Define
\begin{align*}
   \W^* &= f \sum_{i=0}^D \tau_i \E^*_i.
\end{align*}

\begin{lemma}    \label{lem:WWs3}    \samepage
\ifDRAFT {\rm lem:WWs3}. \fi
On $U$ we have
\begin{align}
  \W &= f \tau_r \sum_{i=0}^d \widetilde{\tau}_i \E_{r+i},   &
  \W^* &= f \tau_r \sum_{i=0}^d \widetilde{\tau}_i  \E^*_{r+i},                  \label{eq:WWsaux}
\end{align}
where $\{\widetilde{\tau}_i\}_{i=0}^d$ are from \eqref{eq:deftti}.
\end{lemma}

\begin{proof}
Similar to the proof of Lemma \ref{lem:WonU}.
\end{proof}

\begin{lemma}    \label{lem:spinLP}    \samepage
\ifDRAFT {\rm lem:spinLP}. \fi
The pair $\A_1, \A^*_1$ acts on $U$ as a spin Leonard pair,
and the pair $\W, \W^*$ acts on $U$ as a balanced Boltzmann pair of this
Leonard pair.
\end{lemma}

\begin{proof}
By Lemmas \ref{lem:spinLP0}(ii), \ref{lem:tPhi} the pair $\A, \A^*$ acts on $U$
as a spin Leonard pair.
By \eqref{eq:deftti} and  Lemmas \ref{lem:W}(ii), \ref{lem:WWs3}
the pair $\W,\W^*$ acts on $U$ as a balanced Boltzmann pair for the
Leonard pair $\A,\A^*$ on $U$.
The result follows in view of Lemma \ref{lem:affine}.
\end{proof}

Recall the integers $\{\k_i\}_{i=0}^D$ from  \eqref{eq:defki2}.

\begin{lemma}    \label{lem:f2}    \samepage
\ifDRAFT {\rm lem:f2}. \fi
The scalar $f$ from Definition \ref{def:main3} satisfies
\[
  f^{-2} = |X|^{-3/2} \sum_{i=0}^D \k_i \tau_i.
\]
\end{lemma}

\begin{proof}
By Lemma \ref{lem:tPhi} the sequence
\begin{equation*}
  (\A; \{\E_i\}_{i=0}^D; \A^*; \{\E^*_i\}_{i=0}^D)  
\end{equation*}
acts on the primary $\T$-module $\pU$
as a Leonard system of $q$-Racah type with Huang data $(a,a,a,D)$.
Applying Lemma \ref{lem:sumkiti} to this Leonard system we obtain
\begin{equation}
 \sum_{i=0}^D \k_i \tau_i 
  = \frac{ (a^{-2}; q^2)_D }
            { (a q^{1-D}; q^2 )_D }.                   \label{eq:f2aux2}
\end{equation}
The result follows from \eqref{eq:f} and \eqref{eq:f2aux2}.
\end{proof}

\begin{theorem}    \label{thm:main3}    \samepage
\ifDRAFT {\rm thm:main3}. \fi
The matrix $\W$ in Definition \ref{def:main3} is a spin model
afforded by $\Gamma$.
\end{theorem}

\begin{proof}
For $x \in X$ let $U$ denote an irreducible $\T(x)$-module.
By Lemmas \ref{lem:affine} and \ref{lem:spinLP}
the pair $\A_1, \A^*_1(x)$ acts on $U$ as a spin Leonard pair,
and $\W, \W^*(x)$ acts on $U$ as a balanced Boltzmann pair for this Leonard pair.
By Lemma \ref{lem:f2} the scalar $f$ satisfies \eqref{eq:f-2}.
Thus Assumption \ref{assump0} is satisfied.
Now $\W$ is a spin model by Theorem \ref{thm:main0}.
This spin model is afforded by $\Gamma$ by Theorem \ref{thm:main}.
\end{proof}

\begin{note} 
The spin model $\W$ in Theorem \ref{thm:main3} does not have Hadamard type,
since $\tau_1 \neq \pm \tau_0$ by 
\eqref{eq:resta} and \eqref{eq:ti3}.
\end{note}

\section{Examples}
\label{sec:example}

The article \cite[Section 9]{CN:someformula}  includes a list of 
the known spin models that are contained 
in the Bose-Mesner algebra of a distance-regular graph $\Gamma$.
In this section we display the listed spin models
for which $\Gamma$ has $q$-Racah type.
Each displayed example satisfies the requirements of Assumption \ref{assump}.
For each displayed example we describe $\Gamma$,
and give the parameters $q$, $a$ that satisfy \eqref{eq:thi2}--\eqref{eq:resta2},
 \eqref{eq:b0}--\eqref{eq:cd}.
We now display the examples.
The presentation of the display is explained in Example 0.
\begin{itemize}

\item[(0)]
Name of the graph $\Gamma$ [Reference to definition].
\begin{itemize}
\item
diameter $D$,
intersection numbers $\{\b_0, \ldots, \b_{D-1}; \c_1, \ldots, \c_D\}$,
number of vertices $|X|$.
\item
parameters $q$, $a$.
\end{itemize}

\item[\rm (1)]
Complete graph $K_n$ $(n \geq 2)$ \cite[Appendix A1]{BCN}.

\begin{itemize}
\item
$D = 1$, 
$\{n-1; 1\}$,
$|X|=n$.

\item
$q$ is any nonzero complex number such that $q^4 \neq 1$,
and $a$ is any complex number such that  $a \neq 1$ and $a+a^{-1}=n-2$.
\end{itemize}

\item[\rm (2)]
$4$-cycle $C_4$.

\begin{itemize}
\item
$D=2$,  
$\{2,1; 1,2\}$,
$|X|=4$.

\item
$q$ is any nonzero complex number such that $q^4 \neq 1$,
and $a$ is any complex number such that  $a^2=-1$.
\end{itemize}

\item[(3)]
Higman-Sims graph \cite{HS}.
\begin{itemize}
\item
$D=2$, 
$\{22,21; 1, 6\}$,
$|X|=100$.

\item
$q$ is any complex number such that $q^2+ q^{-2}=-3$,
and $a = - q^{-3}$.
\end{itemize}

\item[(4)]
Hadamard graph $H_m$ $(m \geq 2)$ \cite{Ito}.
\begin{itemize}
\item
$D=4$,  
$\{4m, 4m-1, 2m, 1; 1, 2m, 4m-1,4m\}$,
$|X|=16m$.

\item
$q$ is any complex number such that $q^4+q^{-4}=4m-2$,
and $a$ is any complex number such that  $a^2=-1$.

\end{itemize}

\item[(5)]
Double cover of the Higman-Sims graph \cite[Appendix A.5]{BCN}.
\begin{itemize}
\item
$D=5$,
$\{22,21,16,6,1;1,6,16,21,22\}$,
$|X|=200$.

\item
$q$ is any complex number such that $q^2+q^{-2}=3$,
and $a$ is any complex number such that $a^2= -1$.
\end{itemize}

\item[(6)]
Odd cycle $C_{2m+1}$  $(m \geq 2)$.

\begin{itemize}
\item
$D=m$,
$\{2,1, \ldots, 1; 1,1, \ldots,1\}$,
$|X|=2m+1$.

\item
$q \in \C$ is any primitive $2m+1$ root of unity,
and $a= - q^m$.  
\end{itemize}

\item[(7)]
Even cycle $C_{2m}$ $(m \geq 3)$.

\begin{itemize}
\item
$D=m$,
$\{2,1,\ldots,1; 1,\ldots,1,2\}$,
$|X|=2m$.

\item
$q \in \C$ is any primitive $4m$ root of unity,
and $a$ is any complex number such that $a^2= -1$.
\end{itemize}

\end{itemize}

By Theorem \ref{thm:main3},
for each of the above Examples (1)--(7),
the matrix $\W$ from Definition \ref{def:main3} is a spin model afforded by $\Gamma$.

\bigskip\bigskip\noindent
Kazumasa Nomura\\
Tokyo Medical and Dental University\\
Kohnodai, Ichikawa, 272-0827, Japan\\
Email: knomura@pop11.odn.ne.jp

\bigskip\noindent
Paul Terwilliger\\
Department of Mathematics\\
University of Wisconsin\\
480 Lincoln Drive\\ 
Madison, WI 53706, USA\\
Email: terwilli@math.wisc.edu

\bigskip\noindent
{\bf Keywords.}
Leonard pair, spin model, distance-regular graph, Bose-Mesner algebra
\\
\noindent
{\bf 2010 Mathematics Subject Classification.} 05E30, 15A21

\end{document}